\documentclass[fontsize=11pt]{scrartcl}

\usepackage{amsmath}
\usepackage{amssymb}
\usepackage{amsthm}
\usepackage{stmaryrd}
\usepackage{bbm}
\usepackage[colorlinks]{hyperref}
\usepackage{mathtools}
\usepackage{mathrsfs}
\usepackage{color}
\usepackage{graphicx}
\usepackage{a4,fullpage}
\usepackage{enumerate}
\usepackage[normalem]{ulem}

\usepackage{color}

\usepackage{tikz}
\usetikzlibrary{decorations.pathmorphing,decorations.pathreplacing,calc,arrows.meta}

\DeclareMathOperator{\dist}{dist}

\newcommand{\bE}{\ensuremath{\mathbb{E}}}

\newcommand{\bN}{\ensuremath{\mathbb{N}}}

\newcommand{\bP}{\ensuremath{\mathbb{P}}}

\newcommand{\bR}{\ensuremath{\mathbb{R}}}

\newcommand{\bZ}{\ensuremath{\mathbb{Z}}}

\newcommand{\ind}[1]{\mathbbm{1}_{\{#1\}}} 
\newcommand{\indset}[1]{\mathbbm{1}_{#1}} 

\newcommand{\cC}{\ensuremath{\mathcal{C}}}

\newcommand{\cP}{\ensuremath{\mathcal{P}}}

\newcommand{\ann}{{\mathrm{ann}}}
\newcommand{\pre}{{\mathrm{pre}}}
\newcommand{\que}{{\mathrm{que}}}
\newcommand{\annpre}{{\mathrm{ann}\times\mathrm{pre}}}
\newcommand{\boxpre}{{\mathrm{box-que}\times\mathrm{pre}}}

\newcommand{\abs}[1]{\vert #1 \vert}
\newcommand{\Abs}[1]{\Bigl\vert #1 \Bigr\vert}

\newcommand{\norm}[1]{\left\Vert #1 \right\Vert}

\newcommand{\floor}[1]{\left\lfloor #1 \right\rfloor}

\newcommand{\ddx}[1][1]{\ifnum#1=1 \frac{d}{dx} \else \frac{d^{#1}}{dx^{#1}} \fi}
\newcommand{\ddy}[1][1]{\ifnum#1=1 \frac{d}{dy} \else \frac{d^{#1}}{dy^{#1}} \fi}
\newcommand{\ddt}[1][1]{\ifnum#1=1 \frac{d}{dt} \else \frac{d^{#1}}{dt^{#1}} \fi}

\newcommand{\compl}{\mathsf{C}}

\theoremstyle{plain}
\newtheorem{thm}{Theorem}[section]  
\newtheorem{prop}[thm]{Proposition}
\newtheorem{cor}[thm]{Corollary}
\newtheorem{lem}[thm]{Lemma}

\theoremstyle{definition}
\newtheorem{defn}[thm]{Definition}

\theoremstyle{remark}
\newtheorem{rem}[thm]{Remark}

\numberwithin{equation}{section}

\allowdisplaybreaks[4]

\title{Local limit theorems for a directed random walk on the backbone of
  a supercritical oriented percolation cluster}
\author{Stein Andreas Bethuelsen\footnote{Matematisk institutt, Universitetet i Bergen, \texttt{Stein.Bethuelsen@uib.no}}, Matthias
  Birkner\footnote{Institut f\"{u}r Mathematik, Universit\"{a}t Mainz, \texttt{birkner@mathematik.uni-mainz.de}}, \\ Andrej Depperschmidt\footnote{Department
    of Mathematics, University of Erlangen-Nuremberg,
    \texttt{depperschmidt@math.fau.de}} {} and Timo Schl\"{u}ter\footnote{Institut f\"{u}r Mathematik, Universit\"{a}t Mainz, \texttt{tischlue@uni-mainz.de}}}
\date{\today}

\begin{document}
\maketitle

\begin{abstract}
  We consider a directed random walk on the backbone of the
  supercritical oriented percolation cluster in dimensions $d+1$ with
  $d \ge 3$ being the spatial dimension. For this random walk we prove
  an annealed local central limit theorem and a quenched local limit
  theorem. The latter shows that the quenched transition probabilities
  of the random walk converge to the annealed transition probabilities
  reweighted by a function of the medium centred at the target
  site. This function is the density of the unique measure which is
  invariant for the point of view of the particle, is absolutely
  continuous with respect to the annealed measure and satisfies
  certain concentration properties.
\end{abstract}



\section{Introduction}

Random walks in a static or dynamic random environment arise in
different models from physical and biological sciences. The
investigation of such random walks under different conditions on the
environment has been an active research area with a lot of recent
progress. In this paper, we analyse a \emph{directed} random walk on
the backbone of a supercritical oriented percolation cluster on
$\bZ^d \times \bZ$. This random walk was introduced and studied in
\cite{BirknerCernyDepperschmidtGantert2013}. There it was shown that
the random walk satisfies a law of large numbers and a quenched
central limit theorem for all spatial dimensions $d \ge 1$. The main
purpose of this work is to extend these results and derive a quenched
\emph{local} limit theorem. For this, we will have to restrict
ourselves to spatial dimensions $d \ge 3$. Analogous results for a
class of ballistic random walks in uniformly elliptic i.i.d.\ random
environments were recently obtained in
\cite{BergerCohenRosenthal2016}. This paper has been an inspiration
and a guide for the present study.

\subsection{The model and background}
\label{sec:model-background}

Consider a discrete space-time field
$\omega\coloneqq \{\omega(x,n): (x,n)\in \bZ^d \times \bZ\}$ of
i.i.d.\ Bernoulli random variables with parameter $p \in [0,1]$ taking
values in $\Omega \coloneqq \{0,1\}^{\bZ^d \times \bZ}$, which is
defined on some (large enough) probability space equipped with a
probability measure $\bP$.

As common in percolation theory, a space-time site
$(x,n) \in \bZ^d \times \bZ$ is said to be \emph{open} if
$\omega(x,n)=1$ and \emph{closed} if $\omega(x,n)=0$. A directed
\emph{open path} (with respect to $\omega$) from $(x,m)$ to $(y,n)$
for $m \le n$ is a space-time sequence $(x_m,m),\dots, (x_n,n)$ such
that $x_m=x$, $x_n=y$, $\norm{x_k-x_{k-1}} \le 1$ for
$k=m+1, \dots, n$ and $\omega(x_k,k)=1$ for all $k=m,\dots,n$. Here,
and in the following $\norm{\cdot}$ denotes the $\sup$-norm on
$\bR^d$. We will write $(x,m) \xrightarrow{\omega} (y,n)$ if such an
open path exists and $(x,m) \xrightarrow{\omega} \infty$ if there
exists at least one infinite directed open path starting at $(x,m)$,
i.e.\ if for each $n>m$ there is $y\in \bZ^d$ so that
$(x,m) \xrightarrow{\omega} (y,n)$.

\smallskip

It is well known that there is $p_c = p_c(d) \in (0,1)$ such that
$\bP\big((0,0) \xrightarrow{\omega} \infty \big) > 0$ if and only if
$p > p_c$; see e.g.\ Theorem~1 in \cite{GrimmettHiermerDPRW2000}. We
consider here only the case of a fixed $p \in (p_c,1]$. We define by
\begin{align}
  \label{def:cluster}
  \cC \coloneqq \big\{(x,n) \in \bZ^d \times \bZ : (x,n)
  \xrightarrow{\omega} \infty \big\}
\end{align}
the \emph{backbone of the space-time cluster of the oriented
  percolation}, i.e.\ the set of all space-time sites which are
connected to ``time $+\infty$'' by a directed open path. Note that
$\cC$ depends on $\omega$ and that we have $\bP(|\cC|=\infty)=1$ for
$p>p_c$. For future reference we define the process
$\xi\coloneqq (\xi_n)_{n\in \bZ}$ on $\{0,1\}^{\bZ^d}$ by
\begin{align}
  \label{eq:45}
  \xi_n(x) =\indset{\cC}\bigl((x,n)\bigr).
\end{align}
The process $\xi$ can be interpreted as the time reversal of the
stationary discrete time contact process. In particular, for any
$n\in \bZ$ the random field $\xi_n(\cdot)$ is distributed according to
the upper invariant measure of the discrete time contact process, which is
non-trivial in the case $p>p_c$. For more details we refer the reader
to Section~1 (around equation (1.2)) in
\cite{BirknerCernyDepperschmidtGantert2013}, see also
\cite{BirknerGantert2019}.

\medskip Our goal is to study the directed random walk on the cluster
$\cC$. This random walk was studied in
\cite{BirknerCernyDepperschmidtGantert2013} in the case that the
initial point of the random walk belongs to the cluster. Here we want
to compare the annealed and quenched laws for starting points without
checking whether they are on the cluster or not. Thus, we define the
random walk slightly differently: It behaves as a simple random walk
(which jumps uniformly to one of the sites in the unit ball around the
present site) as long as it is not on the cluster and once it hits the
cluster it will behave as the random walk from
\cite{BirknerCernyDepperschmidtGantert2013}. For a site
$(x,n) \in \bZ^d \times \bZ$ we define its \emph{neighbourhood} at
time $(n+1)$ by
\begin{align}
  \label{eq:13}
  U(x,n) \coloneqq \{ (y,n+1) : \norm{x-y} \le 1 \}.
\end{align}
Given $\omega$ and therefore the random cluster $\mathcal{C}$ and
$(y,m) \in \bZ^d\times \bZ$ we consider random walks
$(X_{n})_{n=m,m+1,\dots}$ with initial position $X_m=y$ and transition
probabilities for $n\ge m$ given by
\begin{align}
  \label{eq:defn_quenched_law}
  \bP(X_{n+1} = z \, \vert \, X_n = x,\omega) =
  \begin{cases}
    \abs{U(x,n) \cap \mathcal{C}}^{-1} &\text{if } (x,n) \in
    \mathcal{C} \text{ and } (z,n+1) \in U(x,n)\cap \mathcal{C},\\
    \abs{U(x,n)}^{-1} &\text{if } (x,n) \notin \mathcal{C} \text{ and
    } (z,n+1) \in U(x,n), \\
    0 &\text{otherwise.} 
  \end{cases}
\end{align}
We write $P_\omega$ for the conditional law of $\bP$ given
$\omega$ and $E_\omega$ for the corresponding expectation. In
particular, for the transition probabilities we have
\begin{align}
  \label{eq:14}
  P_\omega(X_{n+1} = z \, \vert \, X_n = x)
  = \bP(X_{n+1} = z \, \vert \, X_n = x,\omega).
\end{align}
For the above random walk starting from position $X_m= y \in \bZ^d$ at
time $m \in \bZ$ we denote by $P^{(y,m)}_\omega$ its \emph{quenched
  law} and by $E^{(y,m)}_\omega$ the corresponding expectation. The
\emph{annealed (or averaged) law} of that random walk is denoted by
$\bP^{(y,m)}$ and its expectation by $\bE^{(y,m)}$. Note that for any
$A \in \sigma(X_{n} : n=m,m+1,\dots)$ we have
\begin{align}
  \label{eq:15}
  \bP^{(y,m)} (A) = \int P_\omega^{(y,m)}(A)\, d\bP(d\omega).
\end{align}

\subsection{Main results: annealed and quenched local limit theorems}
\label{sec:main-results:-anne}
In Theorem~1.1 in \cite{BirknerCernyDepperschmidtGantert2013} it is
shown that the random walk $(X_n)$ starting in $0 \in \bZ^d$ at time
$0$ satisfies an annealed central limit theorem and the limiting law
is a non-trivial centred isotropic $d$-dimensional normal law. In
particular its covariance matrix $\Sigma$ is of the form
$\sigma^2 I_d$ for a positive constant $\sigma^2$ and the
$d$-dimensional identity matrix $I_d$. Recall that in
\cite{BirknerCernyDepperschmidtGantert2013} it is assumed that the
space-time origin is contained in $\mathcal C$ so that the random walk
starts and stays on $\mathcal C$. This is not a big constraint because
the time a random walk needs to hit the cluster $\mathcal C$ has
exponentially decaying tails; see Lemma~\ref{lem:quenched-rw-cluster}
in the appendix.

The annealed CLT from \cite{BirknerCernyDepperschmidtGantert2013} can
be strengthened to an annealed \emph{local} CLT. For a proof of the
following theorem we refer to
Section~\ref{sect:annealed_estimates}.

\begin{thm}[Annealed local CLT]
  \label{lem:annealed_local_CLT}
  For $d\ge 1$ and $\Sigma$ as above we have
  \begin{align}
    \label{eq:8}
    \lim_{n\to \infty}\sum_{x\in\bZ^d}
    \Abs{\bP^{(0,0)}(X_n=x) - \frac{1}{(2\pi
    n)^{d/2}\sqrt{\det\Sigma}}\exp\Bigl( -\frac{1}{2n}x^T\Sigma^{-1}x
    \Bigr)} = 0.
  \end{align}
\end{thm}

Theorem~3.1 in \cite{BirknerCernyDepperschmidtGantert2013} extends the
annealed CLT therein to a quenched version with the same limiting law.
Thus, the quenched and annealed laws after $N$ steps are comparable on
the level of boxes of side length $N^{1/2}$. This result was later
refined in \cite[Chapter~3]{SteibersPhD2017}, where a comparison
result between the quenched and annealed laws after $N$ steps on the
level of boxes of side length $N^{\theta/2}$ for $\theta \in (0,1)$
was obtained. (We recall this result in
Theorem~\ref{thm:Steibers_Thm3.24} below.) 

The main goal of this paper is to strengthen this further and prove a
quenched local limit theorem which is an analogue of Theorem
\ref{lem:annealed_local_CLT}. In order to state the precise result, we
need to introduce some notation. First, for
$(y,m) \in \bZ^d\times \bZ$, we define the \emph{space-time shift
  operator} $\sigma$ on $\Omega$ by
\begin{align}
  \label{eq:16}
  \sigma_{(y,m)}\omega(x,n) \coloneqq  \omega(x+y,n+m)
\end{align}
and we write $\xi_m(y;\omega)$ for $\xi_m(y)$ read off from a given
realization $\omega$ as in \eqref{def:cluster} and \eqref{eq:45}.
We define the transition kernel for the
environment seen from the particle (compare this with
\eqref{eq:defn_quenched_law}) by
\begin{align}
  \label{eq:transitionKernel}
  \mathfrak{R}f(\omega) \coloneqq  \sum_{\norm{y}\le 1}  g(y;\omega)
 f(\sigma_{(y,1)}\omega)
\end{align}
acting on bounded measurable functions $f:\Omega \to \bR$, where
\begin{align}
  \label{eq:19}
  g(y;\omega) \coloneqq
  \ind{\sum_{\norm{z}\le 1} \xi_1(z;\omega) >0,\, \omega(0,0)=1}
  \frac{\xi_1(y;\omega)}{\sum_{\norm{z}\le 1} \xi_1(z;\omega)} +
  \ind{\sum_{\norm{z}\le 1} \xi_1(z;\omega) =0 \text{ or } \omega(0,0)=0 } \frac{1}{3^d}.
\end{align}
\begin{defn}
  \label{def:invariant}
  A measure $Q$ on $\Omega$ is called \emph{invariant with respect to the
    point of view of the particle} if for every bounded continuous
  function $f:\Omega \to \bR$
  \begin{align}
    \int_\Omega \mathfrak{R} f(\omega)\, dQ(\omega) = \int_\Omega
    f(\omega)\,dQ(\omega).
  \end{align}
\end{defn}

\begin{thm}
  \label{thm:RD}
  Let $d\ge 3$ and $p \in (p_c,1]$. Then there exists a unique measure
  $Q$ on $\Omega$ which is invariant with respect to the point of view
  of the particle satisfying $Q \ll \bP$ and the concentration
  property \eqref{eq:29a} below.
\end{thm}

The main result of this paper is a quenched local limit theorem which
is an analogue of Theorem~1.11 in \cite{BergerCohenRosenthal2016} in
the case of our model.

\begin{thm}[Quenched local limit theorem]
  \label{thm:qlclt}
  Let $d\ge 3$ and $p \in (p_c,1]$, let $Q$ be the measure from
  Theorem~\ref{thm:RD} and denote by $\varphi = dQ/d\bP \in L_1(\bP)$
  the Radon–Nikodym derivative of $Q$ with respect to $\bP$. Then for
  $\bP$ almost every $\omega$ we have
  \begin{align}
    \label{eq:rd-density}
    \lim_{n\to\infty} \sum_{x\in \bZ^d} \bigl|P_\omega^{(0,0)} (X_n=x) -
    \bP^{(0,0)}(X_n=x) \varphi(\sigma_{(x,n)}\omega)\bigr| =0.
  \end{align}
\end{thm}

\begin{rem}[Uniqueness of $Q$]
  It \label{rem:Q-Uniqueness} will be proven in
  Lemma~\ref{lem:uniqueness_prefactor} that the function $\varphi$ in
  \eqref{eq:rd-density} is $\bP$ almost surely unique. Furthermore, it
  will follow from the arguments in the proofs (cf.\ also
  Remark~\ref{rem:Q-estimate}) that if a measure $Q'$ on $\Omega$ is
  invariant with respect to the point of view of the particle and
  satisfies $Q'\ll \bP$ and \eqref{eq:rd-density} with
  $\varphi'= dQ'/d\bP$ then this measure $Q'$ satisfies the
  concentration property \eqref{eq:29a} as well and thus in particular
  agrees with $Q$.
\end{rem}

\paragraph{Related literature}

Random walks in static and dynamic random environments is a very
active research area. For a review of random walks in random
environments and basic concepts and objects we refer the reader to
\cite{SznitmanRWRE2004}; for a more recent review see
\cite{DrewitzRamirez2014}.

The random walk that we consider here can be seen as a random walk in
a dynamic random environment. Its relation to random walks in dynamic
random environments in the literature is discussed in some detail in
\cite[Remark~1.7]{BirknerCernyDepperschmidtGantert2013}. The main
differences are that the random environment is not uniformly elliptic
and is not i.i.d. In fact the environment that we consider here has
even infinite range dependencies, due to the fact that the steps of
the random walk are restricted to the backbone of the oriented
percolation cluster once it hits the cluster. The environment also
does not satisfy mixing conditions such as (conditional) cone-mixing
in contrast to the model considered in
\cite{HollanderSantosSidoraciciusRWDRELLN2013}. In
\cite{BlondelHilarioTeixeira2020} a much weaker mixing assumption than
cone-mixing is introduced (literally for a continuous time model) and
our model does satisfy their assumption. However, they only prove a
LLN for a nearest neighbour random walk in $d=1$. A comprehensive
overview of the recent literature on random walks in dynamic random
environments can be found in the introduction of
\cite{BlondelHilarioTeixeira2020}. See also
\cite[Remark~1.1]{BirknerGantertSteiber}. 

Results on quenched local limit theorems for random walks in random
environments are very recent. Our research is inspired by
\cite{BergerCohenRosenthal2016} where a quenched local limit theorem
was shown (in dimension $d \ge 4$) for the case of an i.i.d.\ random
environment and where the random walk satisfies a ballisticity
criterion and has uniformly elliptic transition probabilities. (Note
that ballisticity is trivial in our model. Uniform ellipticity and the
i.i.d.\ property are not satisfied.)

Other results on local limit theorems in random environments that we
are aware of are concerned with specific models. In
\cite{deuschel2019quenched} the quenched local CLT is proven for
random walks in a time-dependent balanced random environment. In
\cite{DolgopyatGoldsheid2020} and \cite{dolgopyat2020constructive}
quenched local limit theorems are obtained for random walks in random
environments on a strip. A different class of random walks in random
media for which quenched local CLTs have been obtained are the so
called random conductance models. For a recent work in this direction
and an overview of the literature see \cite{Andres_2021} and
references therein.

\paragraph{Outlook and open questions}

While we do exhibit a measure $Q$ which is invariant with respect to
the point of view of the particle and absolutely continuous with
respect to $\bP$, we can currently establish uniqueness only in the
class of such measures satisfying the additional property
\eqref{eq:29a}, see Remark~\ref{rem:Q-Uniqueness}. Furthermore,
because of non-ellipticity, $Q$ is not equivalent to $\bP$, see the
discussion in Remark~\ref{rem:Q-estimate} below. We leave open the
questions whether property \eqref{eq:29a} is necessary for uniqueness
and whether $Q$ is equivalent to $\bP$ when restricted to the set
$\widetilde{\Omega}$ from \eqref{eq:tildeOmega} in
Remark~\ref{rem:Q-estimate}.

We restrict our analysis to the case $d \ge 3$. This is essentially
owed to the fact that Theorem~\ref{thm:Steibers_Thm3.24}, which we
quote from \cite[Thm.~3.24]{SteibersPhD2017}, is presently only
available under this assumption. It was proved there using an
environment exposure technique from \cite{BolthausenSznitman2002},
which was also used by \cite{BergerCohenRosenthal2016}, and the proof
exploited the fact that in dimension at least $3$, two independent
random walks will almost surely meet only finitely often, irrespective
of the number $N$ of steps they take.  In spatial dimension $d=2$, two
independent walks will meet infinitely often, but the number of
intersections up to time $N$ grows quite slowly (of order $\log
N$). It is conceivable that with substantial technical effort, the
proof of \cite[Thm.~3.24]{SteibersPhD2017} and also the results of the
present investigation could be adapted to cover the case $d=2$.
We leave this question for future research.
For our model, simulations suggest that Theorem~\ref{thm:qlclt} should
hold even in spatial dimension $d=1$. However, it seems that a rigorous
analysis of the case $d=1$ would require a completely new approach.

We prove in Theorem~\ref{thm:qlclt} a quenched local limit theorem for
a very specific model of a non-elliptic random walk in a non-trivial
dynamic random environment, and our proofs do exploit specific
properties of this environment, namely the oriented percolation
cluster. However, we think that this environment is prototypical for a
large class of dynamic environments which can be ``mapped'' to it by
suitable coarse-graining procedures, see
\cite{BirknerCernyDepperschmidtRWDRE2016}, Section~3 and the concrete
example in Section~4 there. It seems quite possible that given
substantial technical effort, our approach to Theorem~\ref{thm:qlclt}
could be extended to the class of environments from
\cite{BirknerCernyDepperschmidtRWDRE2016}.  We leave this for future
work.

\paragraph{Outline of the paper}
The proofs of the main results 
are long and quite technical.
Let us describe the main ideas of the proofs and explain how the paper
is organised: In Section~\ref{sec:proofs-main-results} we first give
several auxiliary results which we then use for the proofs of
Theorem~\ref{thm:RD} and of Theorem~\ref{thm:qlclt}.

\medskip \noindent \emph{Annealed estimates:} In
Section~\ref{sect:annealed_estimates} we prove several annealed
derivative estimates which build on, and extend somewhat, previous
work by \cite{SteibersPhD2017}. These estimates will be used
for the proof of the annealed local CLT,
Theorem~\ref{lem:annealed_local_CLT}, also presented in
Section~\ref{sect:annealed_estimates}. Starting with
Section~\ref{sect:box_level_comparisons} the paper is devoted to the
proofs of the auxiliary results from
Section~\ref{sec:proofs-main-results}.

\medskip \noindent \emph{Comparison of the quenched and annealed
  laws:} Lemma~\ref{claim:1}, proven in
Section~\ref{sect:box_level_comparisons}, provides a comparison
between the \emph{quenched} and \emph{annealed} laws on the level of
large (but finite) boxes. In particular it shows that the total
variation distance between $\bP(X_N \in \cdot)$ and
$P_\omega(X_N \in \cdot)$ on the level of boxes of side length
$M \gg 1$ is small with very high probability as $N\to\infty$ in a
suitably quantified way; see equation~\eqref{eq:claim1}. The starting
point of the proof of Lemma~\ref{claim:1} is
\cite[Theorem~3.24]{SteibersPhD2017}, recalled in
Theorem~\ref{thm:Steibers_Thm3.24} below, which gives an analogous
result for boxes whose size grows like $N^{\theta/2}$ with
$0 < \theta < 1$ as $N\to\infty$, and therefore much slower than the
diffusive scale $N^{1/2}$. We augment this with an iteration scheme
that is guided by the proof of Theorem~5.1 in
\cite{BergerCohenRosenthal2016}. The main argument towards the proof
of Lemma~\ref{claim:1} is formulated as
Proposition~\ref{prop_main_thm_5.1} which provides the crucial
estimate for the iteration step. The proof of that proposition is long
and relies to a large extent on ideas from
\cite{BergerCohenRosenthal2016} and is postponed to
Section~\ref{sec:proof-prop-refpr-1}. It requires a suitable control
of the density of ``good'' boxes on which an estimate as in
equation~\eqref{eq:claim1} from Lemma~\ref{claim:1} holds locally
uniformly, see Definition~\ref{def:good boxes}. This deviates from the
set-up in \cite{BergerCohenRosenthal2016} because our environment is
not i.i.d.\ and in fact here the boxes are in principle correlated
over arbitrary lengths, albeit weakly.

\medskip
\noindent
\emph{Measure for the point of view of the particle:} The function
$\varphi =dQ/d\bP$ from \eqref{eq:rd-density} is the density of a
measure $Q$ which is invariant with respect to the point of view of
the particle and absolutely continuous with respect to $\bP$. For the
existence of such a measure $Q$ we consider the quenched laws $Q_N$ of
the environment seen from the particle after $N$ steps of the walk;
see \eqref{eq:defn_Q_N}. The measure $Q$ is constructed as a weak
limit of the Ces\`aro average of the measures $Q_N$ along a
subsequence; see \eqref{eq:80} and \eqref{eq:77}. In
Proposition~\ref{lem:box_average_estimate} and
Corollary~\ref{cor:con_prop_Q} we show that averages of $dQ_N/d\bP$
and $dQ/d\bP$ over large boxes are close to one with high probability
depending on the size of the boxes. It will turn out that the measure
$Q$ which we obtain as described above is unique, i.e.\ it does not
depend on the particular subsequence; see
Remark~\ref{rem:Q-estimate}.

Proposition~\ref{lem:box_average_estimate} and
Corollary~\ref{cor:con_prop_Q} are proven in
Section~\ref{sect:quenched-annealed_coupling}. To this end we
construct a coupling of $Q_N$ and $P_N$, the law of the environment
viewed relative to the annealed walk (note that $P_N = \bP$ for all
$N$). Lemma~\ref{claim:1} allows for a coupling which puts both walks
in the same $M$-box with very high probability. We strengthen this to
a coupling which puts both walks at exactly the same spatial position
with uniformly non-vanishing probability; see the proof of
Lemma~\ref{eq:MBoxTotalVariationDistance}.

Since we average over the environment in the definition of the
annealed law of the random walk in equation~\eqref{eq:15} it is clear
that the annealed random walk does not see any specific environment.
In contrast to that the quenched random walk knows the exact
environment it walks in. So, to compare the annealed and quenched laws
of the random walk, the annealed walk needs to see the environment of
the quenched random walk. This is done through reweighting by
$\varphi$.
In particular, a consequence of multiplying the annealed law with
$\varphi$ is that this product will be zero for all space-time points
$(x,n)\in\bZ^d\times\bZ$ in which the contact process $\xi$ is $0$ in
the environment $\omega$.

In Proposition~\ref{lem:limit_Z_omega} we show that the annealed law
of the random walk at time $n$ reweighted with the function $\varphi$
converges for almost all $\omega$ to a probability law on $\bZ^d$. It
is proven in Section~\ref{sec:proof-prop-refl}.

In Lemma~\ref{lem:uniqueness_prefactor} we will see that a prefactor
$\varphi$ satisfying \eqref{eq:rd-density} is unique. A quite general
proof of that result is given in Section~\ref{sect:prefactor_unique}.

\medskip \noindent \emph{Hybrid measures:} For the proof of
Theorem~\ref{thm:qlclt}, instead of comparing the quenched and
annealed laws directly, we use the triangle inequality, some
``hybrid''
measures and space-time convolutions of quenched-annealed
measures; see Definition~\ref{defn:auxiliary_measures}. In
Proposition~\ref{lem:main}, proven in
Section~\ref{sec:proof-prop-main}, we show that the total variation
distance of some of these measures converges to $0$ as $n$, the number
of steps of the random walk, goes to infinity. An essential tool of
the proof of Proposition~\ref{lem:main} is Lemma~\ref{lem:HIG_lemma}
in which we study the total variation distance of quenched laws of two
random walks starting at different positions. The idea is to use
couplings with the annealed measures on the level of large (growing)
boxes combined with annealed derivative estimates in order to first
ensure that the two walks are in the same box with probability bounded
away from $0$. Using connectivity properties of the oriented
percolation clusters (see below) the above described procedure can be
iterated to produce a literal coupling where the two walks coincide
with high probability after sufficiently many steps.
Lemma~\ref{lem:HIG_lemma} is proven in
Section~\ref{subsec:mixing_quenched}.

\medskip \noindent \emph{Oriented percolation results:} In the
appendix, Section~\ref{sec:inters-clust-points}, we show that two
infinite percolation clusters intersect with high probability within a
finite time. This result was pointed out in
\cite{GrimmettHiermerDPRW2000}, who proved that two infinite clusters
do intersect almost surely, but without the quantification of the time
of intersection. Finally, in Section~\ref{sec:quenched-random-walk},
we show that the probability that a random walk started off the
cluster does not hit the cluster within time $t$ decays exponentially
with $t$.

\section{Proofs of the main results}
\label{sec:proofs-main-results}

In this section we collect several important auxiliary results and
present towards the end of this section how to utilise them to prove
Theorem~\ref{thm:RD} and Theorem~\ref{thm:qlclt}. The proofs of the
auxiliary results are postponed to the subsequent sections.

Our starting point is a lemma which can be seen as an adaptation of
Theorem~5.1 in \cite{BergerCohenRosenthal2016} to our setting. Recall
between \eqref{eq:14} and \eqref{eq:15} the definitions of the
quenched measure $P_{\omega}^{(x,m)}$ and the annealed measure
$\bP^{(x,m)}$ for the random walk $(X_n)_{n=m,m+1,\dots}$ with
$X_m=x$. For any positive real number $L$ we denote by $\Pi_{L}$ a
partition of $\bZ^d$ into boxes of side length $\lfloor L \rfloor$.

\begin{lem}
  Let \label{claim:1} $d \ge 3$. For $N,M \in \bN$, $c, C>0$
  denote by $K(N) \coloneqq K(N,M,c,C)$ the set of environments
  $\omega \in \Omega$ such that for every $x\in \bZ^d$ satisfying
  $\norm{x}\le N$
  \begin{align}
    \label{eq:claim1}
    \sum_{\Delta \in \Pi_{M}} \bigl|P_{\omega}^{(x,0)} (X_{N} \in \Delta ) -
    \bP^{(x,0)}(X_{N} \in \Delta) \bigr| \le \frac{C}{M^{c}} +
    \frac{C}{N^{c}}.
  \end{align}
  If $c>0$ is small enough and $C<\infty$ large enough, there
  are universal positive constants $\tilde c$, $\widetilde C$, for
  which we have
  \begin{align}
    \label{eq:78}
    \bP\bigl(K(N)\bigr)\ge 1-\widetilde CN^{-\tilde c\log N}
    \quad \text{for all } N.
  \end{align}
\end{lem}

In words, Lemma~\ref{claim:1} shows that the total variation distance
between the annealed measure $\bP^{(x,0)}(X_N \in \cdot)$ and the
quenched measure $P_\omega^{(x,0)}(X_N \in \cdot)$ on the level of
boxes of side length $M \gg 1$ is small with very high probability as
$N\to\infty$. The proof of Lemma~\ref{claim:1} is given in
Section~\ref{sect:box_level_comparisons}. It builds on a preliminary
result by Steiber \cite[Theorem~3.24]{SteibersPhD2017} which we recall
in Theorem~\ref{thm:Steibers_Thm3.24} below. The latter gives an
analogous result to Lemma~\ref{claim:1} for boxes of side length
$N^{\theta/2}$ with $0 < \theta < 1$ for large $N$. In particular, for
$N \to \infty$ the side length of these boxes grows much more slowly
than the diffusive scale $N^{1/2}$.

Lemma~\ref{claim:1} allows to construct a coupling of the quenched
walk under $P^{(x,0)}_\omega$ and the annealed walk under
$\bP^{(x,0)}$ which puts both walks in the same $M$-box with very high
probability. We strengthen this coupling to a coupling which puts both
walks at exactly the same spatial position with uniformly
non-vanishing probability; see
Lemma~\ref{eq:MBoxTotalVariationDistance} below. This, in turn, is
essential for the next statement which concerns the difference between
the annealed and quenched law of the environment viewed relative to
the walk after $N$ steps, which we denote by $P_N$ and $Q_N$
respectively. More precisely, for $N\in \bN$, we define $Q_N$ and
$P_N$ by
\begin{align}
  \label{eq:defn_P_N}
  P_N(A) \coloneqq \bE\Big[\sum_{x\in \bZ^d}
  \bP^{(0,0)}(X_N = x)\ind{\sigma_{(x,N)} \omega \in A}\Big]
\end{align}
and
\begin{align}
  \label{eq:defn_Q_N}
  Q_N(A) \coloneqq  \bE\Big[\sum_{x \in \bZ^d}
  P^{(0,0)}_\omega(X_N =x)\ind{\sigma_{(x,N)} \omega \in A} \Big].
\end{align}
Note that, in fact we have $P_N = \bP$ for all $N$; see \eqref{eq:73}.

\medskip
The following proposition is proven in
Section~\ref{sect:quenched-annealed_coupling}.
\begin{prop}
  \label{lem:box_average_estimate}
  For $M \in \bN$ let $\Delta_0(M)$ denote a $d$-dimensional cube of
  side length $M$ in $\bZ^d$ centred at the origin. There exists a
  universal constant $c>0$ so that for every $\varepsilon >0$ there is
  $M_0=M_0(\varepsilon) \in \bN$ so that for $M \ge M_0$ and all
  $N \in \bN$
  \begin{align}
    \label{eq:29}
    \bP\Bigl(\Big\lvert\frac{1}{\abs{\Delta_0(M)}}
    \sum_{x\in \Delta_0(M)}\frac{dQ_N}{d\bP}(\sigma_{(x,0)}\omega)-1
    \Big\rvert>\varepsilon \Bigr) \le M^{-c\log M}.
  \end{align}
\end{prop}

\begin{cor}
  \label{cor:finite_moments_prefactor}
  Let $d\ge3$ and $p>p_c$. Then, for every $k \in \bN$,
  $\sup_{N}\bE[(\frac{dQ_N}{d\bP})^k]< \infty$.
\end{cor}

\begin{proof}
  For $M \in \bN$ large enough,
  Proposition~\ref{lem:box_average_estimate} implies
  \begin{align*}
    \bP \Bigl(\frac{dQ_N}{d\bP}(\omega)>2(2M+1)^d\Bigr)
    \le \bP\Bigl(\frac{1}{(2M+1)^d}\sum_{x\in
    \{-M,\dots,M\}^d}\frac{dQ_N}{d\bP}(\sigma_{(x,n)}\omega)>2\Bigr)
    \le M^{-c\log M},
  \end{align*}
  which implies the assertion.
\end{proof}

We equip $\Omega$ with the product topology and consider the Ces\`aro
sequence
\begin{align}
  \label{eq:80}
  \widetilde{Q}_n \coloneqq \frac{1}{n}\sum_{N=0}^{n-1}Q_N, \quad n =1,2,\dots.
\end{align}
Using Corollary~\ref{cor:finite_moments_prefactor} and the Cauchy-Schwarz
inequality for some finite positive constant $\tilde c$ we have
\begin{align}
  \label{eq:82}
  \begin{split}
    \bE\Bigl[ \Bigl( \frac{1}{n}\sum_{N=0}^{n-1}\frac{dQ_N}{d\bP}
    \Bigr)^2 \Bigr]
    & = \frac{1}{n^2} \sum_{N,N'=0}^{n-1} \bE\Big[
    \frac{dQ_N}{d\bP}\frac{dQ_{N'}}{d\bP} \Big]
    \le 
    \tilde c.
  \end{split}
\end{align}
For $\varepsilon>0$ let $K\subset \Omega$ be a compact subset such
that $\bP(K^{\mathsf{c}})<\varepsilon$. Then by the Cauchy-Schwarz
inequality we obtain
\begin{align*}
  \widetilde{Q}_n(K^{\mathsf{c}}) = \int_\Omega \indset{K^{\mathsf{c}}}
  \frac{d\widetilde{Q}_n}{d\bP} \,d\bP \le \sqrt{\tilde c}\bP(K^{\mathsf{c}})^{1/2} =
  \sqrt{\tilde c\varepsilon}.
\end{align*}
Thus, the sequence $(\widetilde{Q}_n)_n$ is tight. In particular, there is
a weakly converging subsequence, say $(\widetilde{Q}_{n_k})_k$, and we set
\begin{align}
  \label{eq:77}
  Q \coloneqq  \lim_{k\to \infty}\widetilde{Q}_{n_k}.
\end{align}
A standard argument shows that $Q$ is invariant with respect to the
point of view of the particle; see Proposition~1.8 in
\cite{LiggettIPS1985} for an abstract argument or the proof of Lemma~1
in \cite{DrewitzRamirez2014} for the argument in the case of random
walks in random environments.
\smallskip

The proof of the following analogue of
Proposition~\ref{lem:box_average_estimate} for $Q$ instead of $Q_n$ is
given in Section~\ref{sect:quenched-annealed_coupling}.

\begin{cor}
  Recall \label{cor:con_prop_Q} the notation of
  Proposition~\ref{lem:box_average_estimate} and let $Q$ be the
  measure obtained as a limit in \eqref{eq:77}. There exists a
  universal constant $c>0$ so that for every $\varepsilon >0$ there is
  $M_0=M_0(\varepsilon) \in \bN$ and for every $M \ge M_0$ we have
  \begin{align}
    \label{eq:29a}
    \bP\Bigl(\Big\lvert\frac{1}{\abs{\Delta_0(M)}}
    \sum_{x\in \Delta_0(M)}\frac{dQ}{d\bP}(\sigma_{(x,0)}\omega)-1
    \Big\rvert>\varepsilon \Bigr) \le M^{-c\log M}.
    \end{align}
\end{cor}

\begin{proof}[Proof of Theorem~\ref{thm:RD}]
  By construction and shift invariance of $\bP$ we have $Q_N\ll \bP$
  for every $N$ and therefore $\tilde Q_n \ll \bP$ for every $n$.
  Furthermore, by \eqref{eq:82} the family of Radon-Nikodym
  derivatives $(d\tilde Q_n/d\bP)_{n=1,2,\dots}$ is uniformly
  integrable. These facts together imply that we also have $Q \ll \bP$
  for any $Q$ obtained as in \eqref{eq:77}. The concentration property
  is the assertion of Corollary~\ref{cor:con_prop_Q}. For the question
  of uniqueness of $Q$ see Remark~\ref{rem:Q-estimate} below.
\end{proof}

\begin{rem}
  \label{rem:phi_N}
  Using shift-invariance of $\bP$, it is easy to see that for $Q_N$
  from \eqref{eq:defn_Q_N} a version of ${d Q_N}/{d\bP}$ is given by
  \begin{align}
    \label{eq:phi_N}
    \varphi_N(\omega) = \sum_{x \in \bZ^d} P^{(-N,x)}_\omega(X_0=0)
  \end{align}
  (we have
  $P_{\sigma_{(-x,-N)}\omega}^{(0,0)}(X_N=x) =
  P^{(-N,-x)}_\omega(X_0=0)$, recall the notation introduced below
  \eqref{eq:14}). This formula is the analogue of
  \cite[Proposition~1.2]{BolthausenSznitman2002} in our context. In
  particular, $\varphi_N$ is a local function of the space-time values
  of $\xi$ which themselves can be obtained as limits of local
  functions of $\omega$. Thus, $dQ/d\bP$ can be considered as an almost
  sure limit of local functions of $\omega$.
\end{rem}

\medskip
\begin{rem}[Uniqueness of invariant $Q \ll \bP$ with concentration
  properties of the density]
  \leavevmode \label{rem:Q-estimate}\\ A measure $Q$ obtained as in
  \eqref{eq:77} may in principle depend on a particular subsequence.
  In the proof of Theorem~\ref{thm:qlclt} we will show that the
  density $\varphi =dQ/d\bP$ of any measure $Q$ satisfying the
  concentration property \eqref{eq:29a} also satisfies
  \eqref{eq:rd-density}. By Lemma~\ref{lem:uniqueness_prefactor}
  below, such a measure is unique. In particular, in \eqref{eq:77} we
  have weak convergence towards the unique $Q$ along any subsequence
  and therefore we have weak convergence of the Ces\`aro sequence
  $(\widetilde{Q}_n)_{n\in \bN}$ from \eqref{eq:80} towards $Q$.
  However, we currently do not know whether the sequence
  $(Q_N)_{N\in \bN}$ from \eqref{eq:defn_Q_N} converges itself.

  Using Lemma~\ref{lem:quenched-rw-cluster} and \eqref{eq:phi_N} from
  Remark~\ref{rem:phi_N} one can show that $Q$ is concentrated on
  \begin{align}
    \label{eq:tildeOmega}
    \widetilde{\Omega} = \big\{ \omega \in \Omega : \omega \text{
    contains a doubly infinite directed open path through } (0,0) \big\}
  \end{align}
  and thus $Q$ is not equivalent to $\bP$ because
  $0 < \bP(\widetilde{\Omega})<1$. Note that Kozlov's classical
  argument concerning equivalence, see e.g.\
  \cite[Thm.~2.12]{DrewitzRamirez2014}, does not apply because our
  walks are not elliptic. We do not know whether $Q$ is equivalent to
  $\bP(\,\cdot\,|\widetilde{\Omega})$.
\end{rem}

To prove Theorem~\ref{thm:qlclt} we want to make use of the good control
of the difference between the quenched and annealed law on the level
of boxes and various properties of the prefactor $\varphi$ that we
have formulated above in Lemma~\ref{claim:1} and
Corollary~\ref{cor:con_prop_Q}. Furthermore, instead of comparing
$\bP^{(0,0)}(X_N \in \cdot)$ and $P_\omega^{(0,0)}(X_N \in \cdot)$
directly, we compare both of these two measures with auxiliary
``hybrid'' measures which are introduced in the following definition.

\begin{defn}
  Let \label{defn:auxiliary_measures} $Q$ be the measure on $\Omega$
  defined in \eqref{eq:77}, which by Theorem~\ref{thm:RD} and its
  proof is invariant with respect to the point of view of the particle
  with $Q \ll \bP$. Let $\varphi = dQ/d\bP$ be the corresponding
  Radon-Nikodym derivative.
  For $\omega \in \Omega$ and a given partition $\Pi$ of $\bZ^d$ into
  boxes of a fixed side length we define the following measures on
  $\bZ^{d+1}$:
  \begin{align}
    \label{eq:measure_ann_by_prefactor}
    \nu^{\annpre}_\omega(x,n)
    & \coloneqq  \nu^{\annpre}_\omega(\{(x,n)\})
      \coloneqq \frac{1}{Z_{\omega,n}}\bP^{(0,0)}(X_n=x)\varphi(\sigma_{(x,n)}\omega),\\
    \label{eq:measure_quenched}
    \nu^{\que}_\omega(x,n)
    & \coloneqq \nu^{\que}_\omega(\{(x,n)\}) \coloneqq  P_\omega^{(0,0)}(X_n=x),\\
    \label{eq:measure_box-quenched_by_pre}
    \nu^{\boxpre}_{\omega}(x,n)
    & \coloneqq \nu^{\boxpre}_{\omega}(\{(x,n)\})
      \coloneqq P_\omega^{(0,0)}(X_n \in \Delta_x)
      \frac{\varphi(\sigma_{(x,n)}\omega)}{\sum_{y\in\Delta_x}\varphi(\sigma_{(y,n)}\omega)}.
  \end{align}
  Here,
  $Z_{\omega,n}=\sum_{x\in\bZ^d}\bP^{(0,0)}(X_n=x)\varphi(\sigma_{(x,n)}\omega)$
  is the normalizing constant in \eqref{eq:measure_ann_by_prefactor}
  and $\Delta_x$ in \eqref{eq:measure_box-quenched_by_pre} is the
  unique $d$-dimensional box that contains $x$ in the partition $\Pi$.
\end{defn}

All of the measures introduced in the above definition are different
measures of the random walk after $n$ steps:
$\nu^{\annpre}_\omega(\cdot,n)$ is the annealed measure with a
prefactor, $\nu_\omega^\que(\cdot,n)$ is the quenched measure and
$\nu_\omega^\boxpre(\cdot,n)$ is a ``hybrid'' measure, where the box
is chosen according to the quenched measure but then the point inside
the box is chosen according to the (annealed) normalised prefactor. Of
course the measure $\nu_\omega^\boxpre(\cdot,n)$ does depend on the
particular partition $\Pi$ but it will be clear from the context which
partition is used.

First we study the behaviour of the normalizing constant in
\eqref{eq:measure_ann_by_prefactor}; see
Section~\ref{sec:proof-prop-refl} for a proof of the following result.
\begin{prop}
  \label{lem:limit_Z_omega}
  For $\bP$-almost all $\omega\in\Omega$ the
  normalizing constant $Z_{\omega,n}$ satisfies
  \begin{align}
    \label{eq:81}
    \lim_{n\to\infty} Z_{\omega,n} =1.
  \end{align}
\end{prop}

The following proposition is the key result for the proof of
Theorem~\ref{thm:qlclt}. It states that for large $n$ the above
introduced measures are close to each other in a suitable norm. To
state this precisely, for $\omega \in \Omega$ and any two probability
measures $\nu^1_\omega$ and $\nu^2_\omega$ on $\bZ^d \times \bZ$ (more
precisely these are transition kernels from $\Omega$ to
$\bZ^d \times \bZ$) let the \emph{$L^1$ distance of $\nu^1_\omega$ and
  $\nu^2_\omega$} at time $n\in\bZ$ be defined by
\begin{align}
  \label{eq:40}
  \norm{\nu^1_\omega-\nu^2_\omega}_{1,n} \coloneqq \sum_{x\in\bZ^d}
  \abs{\nu^1_\omega(x,n)-\nu^2_\omega(x,n)}.
\end{align}
Furthermore, for $k\le n$ the \emph{space-time convolution of
  $\nu^1_\omega$ and $\nu^2_\omega$} is defined by
\begin{align}
  \label{eq:41}
  (\nu^1\ast \nu^2)_{\omega,k}(x,n) \coloneqq
  \sum_{y\in\bZ^{d}}\nu^1_\omega(y,n-k)\nu^2_{\sigma_{(y,n-k)}\omega}(x-y,k).
\end{align}
We can interpret \eqref{eq:41} as follows: A random walk takes $n-k$
steps in the random medium $\omega$ according to $\nu^1_\omega$, then
re-centers the medium at its current position in space-time and takes
the remaining $k$ steps according to $\nu^2_\omega$.

\begin{prop}
  Fix \label{lem:main} $0<2\delta<\varepsilon<\frac{1}{4}$, and for
  $n\in\bN$ set $k=\lceil n^\varepsilon\rceil$ and
  $\ell=\lceil n^\delta \rceil$. Let $\Pi=\Pi(\ell)$ be a partition of
  $\bZ^d$ into boxes of side length $\ell$. For $\bP$-almost every
  $\omega\in \Omega$ the measures from
  Definition~\ref{defn:auxiliary_measures} satisfy
  \begin{align}
    \label{eq:ann_to_ann-quenched}
    \tag{L1}
     \lim_{n\to\infty}
      \norm{\nu^{\ann\times\pre}_\omega-(\nu^{\ann\times\pre}\ast
      \nu^{\que})_{\omega,k}}_{1,n} & =0,\\
    \label{eq:ann-quenched_to_box-quenched}
    \tag{L2}
    \lim_{n\to \infty}\norm{(\nu^{\ann\times\pre}\ast
      \nu^{\que})_{\omega,k}-(\nu^{\mathrm{box-que}\times\pre}\ast\nu^{\que})_{\omega,k}}_{1,n}
    &  =0,\\
    \label{eq:box-quenched_to_quenched}
    \tag{L3}
    \lim_{n \to \infty}
      \norm{(\nu^{\mathrm{box-que}\times\pre}\ast\nu^{\que})_{\omega,k}
      -(\nu^{\que}\ast\nu^{\que})_{\omega,k}}_{1,n} & =0.
  \end{align}
\end{prop}
The proof of the above proposition is given in
Section~\ref{sec:proof-prop-main}. With the results stated in the
present section we can give a proof of the quenched local limit
theorem.

\begin{proof}[Proof of Theorem~\ref{thm:qlclt}]
  Using the triangle inequality we have
  \begin{align}
    \notag
    \sum_{x\in\bZ^d}
    & \abs{P^{(0,0)}_\omega(X_n=x)  -
      \bP^{(0,0)}(X_n=x)\varphi(\sigma_{(x,n)}\omega) }\\
    \label{eq:I}
    & \leq \sum_{x\in\bZ^d} \abs{ P^{(0,0)}_\omega(X_n=x) -
      (\nu^{\mathrm{box-que}\times \pre}\ast
      \nu^\que)_{\omega,k}(x,n)} \\
    \label{eq:II}
    & \quad +\sum_{x\in\bZ^d}\abs{ (\nu^{\mathrm{box-que}\times
      \pre}\ast \nu^\que)_{\omega,k}(x,n) -
      (\nu^{\ann\times\pre}\ast\nu^\que )_{\omega,k}(x,n) }\\
    \label{eq:III}
    & \quad +\sum_{x\in\bZ^d}\abs{
      (\nu^{\ann\times\pre}\ast\nu^\que )_{\omega,k}(x,n) -
      \nu^{\ann\times\pre}_\omega(x,n) }\\
    \label{eq:IV}
    & \quad + \sum_{x\in\bZ^d}\abs{
      \nu^{\ann\times\pre}_\omega(x,n)
      -\bP^{(0,0)}(X_n=x)\varphi(\sigma_{(x,n)}\omega) }.
  \end{align}
  By Proposition~\ref{lem:main} the terms in \eqref{eq:I},
  \eqref{eq:II} and \eqref{eq:III} tend to $0$ as $n$ goes to
  infinity. In order to compare \eqref{eq:I} with
  \eqref{eq:box-quenched_to_quenched} literally note that we have
  $P^{(0,0)}_\omega(X_n=x) =
  \nu^{\que}\ast\nu^{\que})_{\omega,k}(x,n)$ by construction. Finally,
  by definition of $\nu^{\ann\times\pre}_\omega(x,n)$ the term in
  \eqref{eq:IV} can be written as
  \begin{align}
    \label{eq:83}
    \Abs{
    \frac{1}{Z_{\omega,n}} -1} \sum_{x\in\bZ^d}
    \bP^{(0,0)}(X_n=x)\varphi(\sigma_{(x,n)}\omega) = \Abs{
    \frac{1}{Z_{\omega,n}} -1}Z_{\omega,n}.
  \end{align}
  By Proposition~\ref{lem:limit_Z_omega} it follows that the
  expression in \eqref{eq:83} converges to $0$ as $n$ tends to
  infinity.
\end{proof}

\section{Annealed estimates and the proof of
  Theorem~\ref{lem:annealed_local_CLT}}
\label{sect:annealed_estimates}

In this section we collect estimates for the annealed walk that will
be needed later in the proofs, and present a proof of
Theorem~\ref{lem:annealed_local_CLT}.

\begin{lem}[Annealed derivative estimates]
  \label{lem:annealed_derivative_estimates}
  For $d \ge 3$, $j=1,\dots,d$, $x,y\in \bZ^d$, $m,n \in \bZ$,
  $m \in \bZ$, $n \in \bN$ denoting by $e_j$ the $j$-th (canonical)
  unit vector we have
  \begin{align}
    \label{eq:1}
    \abs{\bP^{(y,m)}(X_{n+m} = x) -\bP^{(y+e_j,m)}(X_{n+m} = x)}
    & \le Cn^{-(d+1)/2},\\
    \label{eq:3}
    \abs{\bP^{(y,m)}(X_{n+m} = x) -\bP^{(y,m+1)}(X_{n+m} = x)}
    & \le Cn^{-(d+1)/2},\\
    \label{eq:4}
    \abs{\bP^{(y,m)}(X_{n+m} = x) -\bP^{(y,m)}(X_{n+m} = x+e_j)}
    & \le Cn^{-(d+1)/2},\\
    \label{eq:5}
    \abs{\bP^{(y,m)}(X_{n+m} = x) -\bP^{(y,m)}(X_{n-1+m} = x)}
    &\le Cn^{-(d+1)/2}.
  \end{align}
\end{lem}
\begin{proof}
  The estimates \eqref{eq:1} and \eqref{eq:3} are from
  \cite{SteibersPhD2017}; see Lemma~3.9 and its proof in Appendix~A.2
  there. By translation invariance we have
  \begin{align*}
    \bP^{(y+e_j,m)}(X_{n+m} =x) & = \bP^{(y,m)}(X_{n+m}= x-e_j)\\
    \intertext{and}
    \bP^{(y,m+1)}(X_{n+m}=x) & = \bP^{(y,m)}(X_{n-1+m}=x).
  \end{align*}
  Thus, the estimates \eqref{eq:4} and \eqref{eq:5} follow from
  \eqref{eq:1} and \eqref{eq:3}.
\end{proof}

We will also need the following generalization of  the annealed
derivate estimates in the previous lemma.
\begin{lem}
  \label{lem:additional_annealed_estimate}
  Let $\varepsilon>0$. For $n\in\bN$ large enough and every partition
  $\Pi^{(\varepsilon)}_n$ of $\bZ^d$ into boxes of side length
  $\lfloor n^\varepsilon \rfloor$, we have
  \begin{align}
    \label{eq:6}
    \sum_{\Delta\in \Pi^{(\varepsilon)}_n}\sum_{x\in\Delta}
    \max_{y\in\Delta}\bigl[\bP^{(0,0)}(X_n=y)-\bP^{(0,0)}(X_n=x)\bigr]
    \le Cn^{-\frac{1}{2}+ 3d\varepsilon}.
  \end{align}
\end{lem}
\begin{proof}
  We consider the following set of boxes around the origin of $\bZ^d$
  \begin{align}
    \label{eq:7}
    \widetilde{\Pi}^{(\varepsilon)}_n\coloneqq \{
    \Delta\in\Pi^{(\varepsilon)}_n:\Delta\cap[-\sqrt{n}\log^3
    n,\sqrt{n}\log^3n]^d\neq \emptyset \}.
  \end{align}
  With this notation we can write the sum on the left hand side of
  \eqref{eq:6} as
  \begin{align}
    \label{eq:proof_extra_annealed_estimations_1}
      & \sum_{\Delta\in \widetilde{\Pi}^{(\varepsilon)}_n}
        \sum_{x\in\Delta} \max_{y\in\Delta}\bigl[\bP^{(0,0)}(X_n=y)-\bP^{(0,0)}(X_n=x)\bigr]\\
    \label{eq:proof_extra_annealed_estimations_2}
      & \qquad
        + \sum_{\Delta\in \Pi^{(\varepsilon)}_n\setminus\widetilde{\Pi}^{(\varepsilon)}_n}
        \sum_{x\in\Delta}
        \max_{y\in\Delta}\bigl[\bP^{(0,0)}(X_n=y)-\bP^{(0,0)}(X_n=x)\bigr].
  \end{align}
  So, it is enough to prove suitable upper bounds for these two sums.
  By Lemma~3.6 from \cite{SteibersPhD2017} we have
  \begin{align}
    \label{eq:proof_extra_annealed_estimations_3}
    \sum_{\Delta \in \Pi^{(\varepsilon)}_n\setminus\widetilde{\Pi}^{(\varepsilon)}_n}
    \bP^{(0,0)}(X_n\in \Delta) \le Cn^{-c\log n}
  \end{align}
  for some positive constants $C$ and $c$. Thus, the double sum
  \eqref{eq:proof_extra_annealed_estimations_2} is bounded from above
  by
  \begin{multline*}
    \sum_{\Delta\in \Pi^{(\varepsilon)}_n\setminus\widetilde{\Pi}^{(\varepsilon)}_n}
    \sum_{x\in\Delta} \bigl[\bP^{(0,0)}(X_n\in\Delta) - \bP^{(0,0)}(X_n=x)\bigr] \\
    = \sum_{\Delta\in \Pi^{(\varepsilon)}_n\setminus\widetilde{\Pi}^{(\varepsilon)}_n}
    (\abs{\Delta}-1)\bP^{(0,0)}(X_n\in\Delta)
     \le Cn^{d\varepsilon}n^{-c\log n} \le \widetilde{C}n^{-\tilde{c}\log n}
  \end{multline*}
  for suitably chosen constants $\tilde c$ and $\widetilde C$. Using
  annealed derivative estimates from
  Lemma~\ref{lem:annealed_derivative_estimates} the double sum
  \eqref{eq:proof_extra_annealed_estimations_1} is bounded above by
  \begin{align*}
    \sum_{\Delta\in \widetilde{\Pi}^{(\varepsilon)}_n}\sum_{x\in\Delta} Cn^\varepsilon
    n^{-\frac{d+1}{2}} \le C(n^\varepsilon+\sqrt{n}\log^3 n)^d
    n^\varepsilon n^{-\frac{d+1}{2}} \le Cn^{3d\varepsilon}n^{-1/2}.
  \end{align*}
  Combination of the last two displays completes the proof.
\end{proof}

\begin{proof}[Proof of Theorem~\ref{lem:annealed_local_CLT}]
  Let $\varepsilon, \delta>0$ be small (they will later be tuned
  appropriately). Let $\Pi^{(\varepsilon)}_n$ be a partition of $\bZ^d$
  in boxes of side length $\lceil \varepsilon \sqrt{n}\, \rceil$. Let
  $C_\delta>0$ be a constant such that
  $\bP^{(0,0)}(\norm{X_n}>C_\delta \sqrt{n})<\delta$; such a constant
  exists by Lemma~3.6 from \cite{SteibersPhD2017}. Furthermore denote
  by $\Pi^{(\varepsilon,\delta)}_n$ the subset of boxes in
  $\Pi^{(\varepsilon)}_n$ intersecting
  $\{x\in \bZ^d: \norm{x} \le C_\delta \sqrt{n} \}$. Then
  \begin{align}
    &\sum_{x\in\bZ^d} \Abs{\bP^{(0,0)}(X_n=x) - \frac{1}{(2\pi
      n)^{d/2}\sqrt{\det\Sigma}}
      \exp\Bigl(-\frac{1}{2n}x^T\Sigma^{-1}x \Bigr)} \notag\\
    \label{eq:annealed_lCLT_proof_1}
    & \quad =\sum_{\Delta \in \Pi^{(\varepsilon)}_n\setminus\Pi^{(\varepsilon,\delta)}_n}
      \sum_{x  \in \Delta}\Abs{\bP^{(0,0)}(X_n=x) - \frac{1}{(2\pi
      n)^{d/2}\sqrt{\det\Sigma}}\exp\Bigl(
      -\frac{1}{2n}x^T\Sigma^{-1}x \Bigr)} \\
    \label{eq:annealed_lCLT_proof_2}
    & \qquad \quad +\sum_{\Delta \in \Pi^{(\varepsilon,\delta)}_n}
      \sum_{x \in \Delta}\Abs{\bP^{(0,0)}(X_n=x) - \frac{1}{(2\pi
      n)^{d/2}\sqrt{\det\Sigma}}
      \exp\Bigl(-\frac{1}{2n}x^T\Sigma^{-1}x \Bigr)}.
  \end{align}
  We will show that $\varepsilon$ can be chosen so small that the
  above sum is bounded by $4\delta$ for large enough $n$. We first
  find an upper bound for \eqref{eq:annealed_lCLT_proof_1}. By
  definition of $\Pi^{(\varepsilon,\delta)}_n$ if
  $\Delta \in
  \Pi^{(\varepsilon)}_n\setminus\Pi^{(\varepsilon,\delta)}_n$ then we
  have $\norm{x} > C_\delta\sqrt{n}$ for all $x\in \Delta$.
  Thus, \eqref{eq:annealed_lCLT_proof_1} is bounded from above by
  \begin{align*}
    \sum_{\substack{x\in \bZ^d\\\norm{x}>C_\delta \sqrt{n}}}
    \Bigl(\bP^{(0,0)}(X_n=x) + \frac{1}{(2\pi
    n)^{d/2}\sqrt{\det\Sigma}}\exp\Bigl( -\frac{1}{2n}x^T\Sigma^{-1}x
    \Bigr)\Bigr)
    \le \delta + C\exp\Bigl( -\frac{c}{2}C_\delta^2 \Bigr).
  \end{align*}
  By choosing $C_\delta$ large enough we can ensure that
  \eqref{eq:annealed_lCLT_proof_1} is bounded by $2\delta$.

  Turning to \eqref{eq:annealed_lCLT_proof_2} we first compare the two
  terms in $\abs{\cdot}$ with the averages over appropriate boxes.
    First, let $x \in \bZ^d$ be fixed and let
    $\Delta \in \Pi^{(\varepsilon)}_n$ be the box containing $x$.
    Using annealed derivative estimates from Lemma~3.9 in
    \cite{SteibersPhD2017} we obtain
  \begin{align*}
    \abs{\bP^{(0,0)}(X_n=x)
    & - \frac{1}{\lceil \varepsilon \sqrt{n}\,\rceil^d}
      \bP^{(0,0)}(X_n \in \Delta)}  \\
    & = \frac{1}{\lceil \varepsilon \sqrt{n} \,\rceil^d}
      \Big\lvert\sum_{y \in \Delta} \bP^{(0,0)}(X_n=x)-\bP^{(0,0)}(X_n=y)\Big\rvert\\
    & \le \frac{1}{\lceil \varepsilon \sqrt{n} \,\rceil^d}
      \sum_{y \in \Delta} \norm{x-y} n^{-(d+1)/2}
      \le \lceil \varepsilon\sqrt{n}\,\rceil \cdot n^{-(d+1)/2} \le
      \frac{\varepsilon}{n^{d/2}}.
  \end{align*}

  Now consider $\Delta \in \Pi^{(\varepsilon,\delta)}_n$. For every
  $x\in \Delta$ we have
  \begin{align*}
    \Abs{\exp & \Bigl( -\frac{1}{2n}x^T\Sigma^{-1}x \Bigr) -
      \frac{1}{\lceil \varepsilon
      \sqrt{n}\,\rceil^d}\int_{\Delta}\exp\Bigl(
      -\frac{1}{2n}y^T\Sigma^{-1}y \Bigr)\,dy }\\
    & \quad =\exp\Bigl( -\frac{1}{2n}x^T\Sigma^{-1}x \Bigr)
      \Abs{1 - \frac{1}{\lceil \varepsilon \sqrt{n}\,\rceil^d}
      \int_{\Delta}\exp\Bigl(-\frac{1}{2n}(y^T\Sigma^{-1}y
      - x^T\Sigma^{-1}x) \Bigr)\,dy} \\
    & \quad \le\exp\Bigl( -\frac{1}{2n}x^T\Sigma^{-1}x
      \Bigr)\frac{1}{\lceil \varepsilon \sqrt{n}\,\rceil^d}\\
    &\hspace{2cm}\times \int_{\Delta}\Abs{1-\exp\Bigl(
      -\frac{1}{2n}((y-x)^T\Sigma^{-1}(y-x) +
      2x^T\Sigma^{-1}(y-x))\Bigr)} \,dy\\
    & \quad \le \exp\Bigl( -\frac{1}{2n}x^T\Sigma^{-1}x
      \Bigr)\frac{1}{\lceil \varepsilon \sqrt{n}\,\rceil^d}
      \int_{\Delta}\Abs{1-\exp\Bigl( -\frac{1}{2n}(C\varepsilon^2n + C
      C_\delta\varepsilon n)\Bigr)}\,dy\\
    & \quad \le\exp\Bigl( -\frac{1}{2n}x^T\Sigma^{-1}x \Bigr) \cdot
      C\varepsilon \le C\varepsilon,
  \end{align*}
  where we have used $\norm{x-y} \le \varepsilon\sqrt{n}$ and
  $\norm{x} \le C_\delta \sqrt{n}$. Using first the triangle
  inequality and then combining the last two estimates we see that
  each summand in \eqref{eq:annealed_lCLT_proof_2} is bounded from
  above by
  \begin{align}
    \label{eq:10}
    \begin{split}
      \big\lvert \bP^{(0,0)}  & (X_n=x) - \frac{1}{ \lceil \varepsilon
        \sqrt{n}\,\rceil^d} \bP^{(0,0)}(X_n \in \Delta) \big\rvert\\
      & \quad
      + 
      \frac{1}{(2\pi n)^{d/2}\sqrt{\det \Sigma}} 
      \Big\lvert \exp\Bigl( -\frac{1}{2n}x^T\Sigma^{-1}x \Bigr) -
      \frac{1}{\lceil \varepsilon \sqrt{n}\,\rceil^d}
      \int_{\Delta}\exp\Bigl(-\frac{1}{2n}y^T\Sigma^{-1}y \Bigr)\,dy \Big\rvert\\
      & \quad
      + 
      \frac{1}{\lceil \varepsilon \sqrt{n}\,\rceil^d}
      \Big\lvert \bP^{(0,0)}(X_n \in \Delta) 
      - \frac{1}{(2\pi n)^{d/2}\sqrt{\det \Sigma}}\int_{\Delta}
      \exp\Bigl(-\frac{1}{2n}y^T\Sigma^{-1}y \Bigr)\,dy\Big\rvert\\
      &
      \le 
      \frac{C\varepsilon}{n^{d/2}} +
      \frac{C\varepsilon}{(2\pi n)^{d/2}\sqrt{\det \Sigma}}\\
      & \quad +
      \frac{1}{\lceil \varepsilon \sqrt{n}\,\rceil^d} \Big\lvert
      \bP^{(0,0)}(X_n \in \Delta) - \frac{1}{(2\pi n)^{d/2}\sqrt{\det
          \Sigma}}
      \int_{\Delta} \exp\Bigl(-\frac{1}{2n}y^T\Sigma^{-1}y \Bigr)\,dy\Big\rvert.
    \end{split}
  \end{align}
  The number of vertices summed over all $\Delta \in
  \Pi^{(\varepsilon,\delta)}_n$ is bounded
  by $((C_\delta+\varepsilon)\sqrt{n})^d\le C(C_\delta\sqrt{n})^d$.
  Thus,
  \begin{align}
    \label{eq:9}
    \sum_{\Delta \in \Pi^{(\varepsilon,\delta)}_n} \sum_{x \in \Delta}
     \Bigl(\frac{C\varepsilon}{n^{d/2}} +
          \frac{C\varepsilon}{(2\pi n)^{d/2}\sqrt{\det \Sigma}}\Bigr)
    \le C\cdot C_\delta^d \varepsilon.
  \end{align}
  Summing the last line in \eqref{eq:10} with the double sum
  $\sum_{\Delta \in \Pi^{(\varepsilon,\delta)}_n} \sum_{x \in \Delta}$
  gives
  \begin{align}
    \label{eq:11}
   \sum_{\Delta \in \Pi^{(\varepsilon,\delta)}_n} \Big\lvert \bP^{(0,0)}(X_n \in \Delta)
    - \frac{1}{(2\pi n)^{d/2}\sqrt{\det \Sigma}}\int_{\Delta}
    \exp\Bigl(-\frac{1}{2n}y^T\Sigma^{-1}y \Bigr)\,dy\Big\rvert.
  \end{align}
  By applying the annealed CLT from
  \cite{BirknerCernyDepperschmidtGantert2013} (and approximating the
  indicator $\indset{\Delta}$ appropriately by continuous and bounded
  functions) and noting that for fixed $\varepsilon$ and $\delta$ the
  set $\Pi^{(\varepsilon,\delta)}_n$ is finite implies that
  \eqref{eq:11} goes to zero as $n$ tends to infinity. In particular
  it is smaller than $\delta$ for large enough $n$.

  Combining the estimates above we obtain
  \begin{align*}
    \sum_{x\in\bZ^d} \Big\lvert\bP^{(0,0)}(X_n=x) - \frac{1}{(2\pi
    n)^{d/2}\sqrt{\det\Sigma}}\exp\Bigl( -\frac{1}{2n}x^T\Sigma^{-1}x
    \Bigr)\Big\rvert \le 2\delta + C\cdot C_\delta^d\varepsilon +\delta
    <4\delta
  \end{align*}
  for large enough $n$ and choosing $\varepsilon>0$ so that
  $C\cdot C_\delta^d \varepsilon<\delta$. This concludes the proof.
\end{proof}

\section{Proof of Lemma~\ref{claim:1}}
\label{sect:box_level_comparisons}

For the proof of of Lemma~\ref{claim:1} we follow closely the proof of
Theorem~5.1 in \cite{BergerCohenRosenthal2016} and adapt their
arguments to our model. The general idea is to implement an iteration
scheme that carries the annealed-quenched comparison from
Theorem~\ref{thm:Steibers_Thm3.24} below along a sequence of more and
more slowly growing box scales.

Let us introduce some notation first. Let $\theta>0$ be a (small)
constant to be determined in the proof. For $j\in \bN$, we set
$n_j \coloneqq \lfloor N^{\frac{1}{2^j}} \rfloor$ and
$r(N) \coloneqq \lceil \log_2(\frac{\log N)}{\theta \log M}) \rceil$.
Note that $r(N)$ is the smallest integer satisfying
$n_{r(N)}^{\theta} \le M$. Furthermore we set
\begin{align}
  \label{eq:12}
  N_0 \coloneqq N - \sum_{j=1}^{r(N)} n_j \quad \text{and} \quad
  N_k \coloneqq \sum_{j=1}^k n_j + N_0 = N_{k-1}+n_k,
  \text{ for all $1 \le k \le r(N)$.}
\end{align}
Finally, for $0 \le k \le r(N)$, abusing the notation and suppressing
the dependence on $\theta$ and $n$ we write for the rest of this
section $\Pi_{k} \coloneqq \Pi_{n_k^{\theta}}$ and define
\begin{align}
  \label{eq:42}
  \lambda_k (\omega) \coloneqq \sum_{\Delta \in \Pi_{k}}
  \big|P_{\omega}^{(0,0)}(X_{N_k} \in \Delta)- \bP^{(0,0)} (X_{N_k} \in \Delta)\big|.
\end{align}
Note in particular that $\lambda_{r(N)}$ is twice the total variation
distance between the quenched and the annealed measures on boxes of
side length $\le M$, which is the term we wish to bound from above to
show \eqref{eq:claim1}. If one wishes to be slightly more precise,
then one should replace $N_{r(N)}$ by $M$, thus obtaining the total
variation for boxes of side length $M$ exactly. This, however, does
not influence the estimates to follow.

\medskip

The proof of the following proposition is long and technical and will
be given in Section~\ref{sec:proof-prop-refpr-1}.

\begin{prop}
  There \label{prop_main_thm_5.1} exists constants $C,c,\alpha>0$ and
  events $G(N),N\in\bN,$ with $\bP(G(N))\ge 1-CN^{-c\log N}$ such that
  for all $\omega\in G(N)$ we have
  \begin{align}
        \label{eq:prop_main_thm_5.1}
    \lambda_k \le \lambda_{k-1} +Cn_k^{-\alpha},\;
    \quad \forall \: 1\le k \le r(N).
  \end{align}
  In particular,
  $\lambda_{r(N)} \le \lambda_1 + C \sum_{k=1}^{r(N)} n_k^{-\alpha}$
  for $\omega\in G(N)$.
\end{prop}

\begin{proof}[Proof of Lemma~\ref{claim:1}]
  The assertion is a consequence of
  Proposition~\ref{prop_main_thm_5.1} and can be proven analogously to
  the argument in the last part of the proof of Theorem~5.1 in
  \cite{BergerCohenRosenthal2016}, page 35.
\end{proof}

\section{Concentration from coupling: Proofs of
  Proposition~\ref{lem:box_average_estimate} and
  Corollary~\ref{cor:con_prop_Q}}
\label{sect:quenched-annealed_coupling}

In this section we prove some analogues of the results of Section~6
in \cite{BergerCohenRosenthal2016} and present proofs of
Proposition~\ref{lem:box_average_estimate} and Corollary~\ref{cor:con_prop_Q}.

\begin{lem}
  \label{lem:coupling_Theta}
  There exists a constant $c>0$ and set of environments $K(N,c)$
  satisfying
  \begin{align}
    \label{eq:27}
    \bP(K(N,c))\ge 1 -N^{-c\log N}
  \end{align}
  such that for every $\omega$ there exists a coupling
  $\Theta_{\omega,N}$ of $\bP^{(0,0)}(X_N=\cdot)$ and
  $P_\omega^{(0,0)}(X_N=\cdot)$ with the property
  \begin{align}
    \label{eq:74}
    \Theta_{\omega,N}(\Lambda) > c \; \text{ for every $\omega \in
    K(N,c)$},
  \end{align}
  where $\Lambda \coloneqq \{(x,x): x\in\bZ^d \}$.
\end{lem}
\begin{proof}
  For $\varepsilon>0$ and $M\in \bN$ denote by
  $K(N)=K(N,M,\varepsilon)$ the set of environments
  $\omega \in \Omega$ satisfying
  \begin{align}
    \label{eq:MBoxTotalVariationDistance}
    \sum_{\Delta \in \Pi_M}\abs{P^{(0,0)}_\omega(X_N \in \Delta) -
      \bP^{(0,0)}(X_N \in \Delta)} < \varepsilon,
  \end{align}
  where $\Pi_{M}$ is a partition of $\bZ^d$ into $d$-dimensional boxes
  of side length $M$. By Lemma~\ref{claim:1}, for every
  $\varepsilon \in (0,1)$ there exists a $M\in \bN$ such that
  $\bP(K(N))\ge 1-N^{-c\log N}$. On the event $K(N)$, the inequality
  \eqref{eq:MBoxTotalVariationDistance} tells us that twice the total
  variation distance between $\bP^{(0,0)}(X_N \in \cdot)$ and
  $P^{(0,0)}_\omega(X_N\in \cdot)$ on $\Pi_M$ is less than
  $\varepsilon$ and therefore there exists a coupling
  $\widetilde{\Theta}_{\omega,N,M}$ of both measures on
  $\Pi_M\times\Pi_M$ such that
  $\widetilde{\Theta}_{\omega,N,M}(\Lambda_{\Pi_M})>1-\varepsilon$,
  where $\Lambda_{\Pi_M}=\{(\Delta,\Delta): \Delta\in \Pi_M\}$.

  Using the coupling $\widetilde{\Theta}$ we construct a new coupling
  of $\bP^{(0,0)}(X_N=\cdot)$ and $P_\omega^{(0,0)}(X_N=\cdot)$ on
  $\bZ^d\times\bZ^d$ which puts positive probability on the
  diagonal $\Lambda = \{(x,x): x\in \bZ^d \}$. We define
  $\Theta_{\omega,N}$ on $\bZ^d\times \bZ^d$ by
  \begin{multline}
    \label{eq:23}
    \Theta_{\omega,N}(x,y) \coloneqq
     \sum_{\Delta,\Delta'\in \Pi_M} \widetilde{\Theta}_{\omega,N-M,M}(\Delta,\Delta')\\
    \cdot \bP^{(0,0)}(X_N = x\vert X_{N-M}\in \Delta)
        \cdot P_\omega^{(0,0)}(X_N = y\vert X_{N-M} \in \Delta').
  \end{multline}
  Since $\widetilde{\Theta}_{\omega, N-M,M}$ is a coupling of
  $\bP^{(0,0)}$ and $P_\omega^{(0,0)}$ on $\Pi_M\times \Pi_M$ one can
  easily see that by the formula of total probability
  $\Theta_{\omega,N}$ is indeed a coupling of
  $\bP^{(0,0)}(X_N = \cdot)$ and $P_\omega^{(0,0)}(X_N = \cdot)$.

  For $x\in \bZ^d$, let $\Delta_x$ be the unique cube which contains
  $x$ in the partition $\Pi_M$. Since the side length of each box in
  the partition $\Pi_M$ is $M$ it follows that the annealed random
  walk can reach $x$ from each point in the box $\Delta_x$ in less
  than $M$ steps.

  Next we want to show that the coupling gives us a positive chance
  for the two walks to end up at the same position. In
  \cite{BergerCohenRosenthal2016} this is done by showing that
  $\Theta_{\omega,N}(x,x)$ is bounded away from zero for all
  $x\in \bZ^d$. This is not true in our model because we do not have
  uniform ellipticity for the quenched measure. The idea here is to
  show that for ``typical'' $\omega$ the measure
  $\Theta_{\omega,N}(x,x)$ is bounded away from zero for ``many''
  $x\in \bZ^d$. To this end for given $\omega$ we define the set
  $\Pi^x_\omega \subset \Pi_M$ as the set of boxes $\Delta \in \Pi_M$
  satisfying
  \begin{align}
    \label{eq:20}
    P_\omega^{(0,0)}(X_N =x \vert X_{N-M}\in \Delta) > 0.
  \end{align}
  Note that if $\Pi^x_\omega=\emptyset$ for $x$ and $\omega$ then we
  have $\Theta_{\omega,N}(x,x)=0$. Furthermore, by definition of
  $P_\omega^{(0,0)}(X_N = x \vert X_{N-1}=y)$ we have
  \begin{align}
       \label{eq:22}
       P_\omega^{(0,0)}(X_N = x \vert X_{N-M} \in \Delta) \ge
       \left(\frac{1}{3^d} \right)^{M}
  \end{align}
  for all $\Delta \in \Pi^x_\omega$. Now using \eqref{eq:23},
  \eqref{eq:22} and uniform ellipticity of the annealed measure we
  obtain
  \begin{align*}
    \Theta_{\omega,N}(x,x)
    & = \sum_{\Delta \in \Pi^x_\omega}
      \widetilde{\Theta}_{\omega,N-M,M}(\Delta,\Delta)\\
    & \qquad \qquad
      \cdot \bP^{(0,0)}(X_N =x\vert X_{N-M}\in\Delta) \cdot
      P^{(0,0)}_\omega(X_N =x\vert X_{N-M}\in \Delta)\\
    &\ge \sum_{\Delta \in \Pi^x_\omega}
      \widetilde{\Theta}_{\omega,N-M,M}(\Delta,\Delta)\eta^{M}
         \left(\frac{1}{3^d}\right)^{M},
  \end{align*}
  where $\eta \in (0,1)$ is the ``uniform ellipticity bound'' of the
  annealed random walk. Now it suffices to show
  \begin{align}
    \label{eq:75}
    \sum_{x \in \bZ^d} \sum_{\Delta \in \Pi^x_\omega}
    \widetilde{\Theta}_{\omega,N-M,M}(\Delta,\Delta) \ge
    \sum_{\Delta \in \Pi_M}\widetilde{\Theta}_{\omega,N-M,M}(\Delta,\Delta).
  \end{align}
  This follows immediately if we can show that for all
  $\Delta \in \Pi_M \setminus \cup_{x \in \bZ^d}\Pi^x_\omega$ we have
  \begin{align*}
       \widetilde{\Theta}_{\omega,N-M,M}(\Delta,\Delta) = 0.
  \end{align*}
  For that consider a box
  $\Delta \in \Pi_M \setminus \cup_{x \in \bZ^d}\Pi^x_\omega$, i.e.\
  there is no $x\in \bZ^d$ with $\Delta \in \Pi^x_\omega$ for the
  fixed $\omega$. Thus, we have
  $P^{(0,0)}_\omega(X_N =x \vert X_{N-M}\in \Delta) = 0$ for all
  $x \in \bZ^d$. It follows that
  $P^{(0,0)}_\omega(X_{N-M}\in \Delta)=0$, because there can be no
  infinitely long open path starting from $\Delta$. We obtain
  \begin{align}
    \label{eq:24}
    \begin{split}
      \Theta_{\omega,N}(\Lambda)
      & = \sum_{x \in \bZ^d} \Theta_{\omega,N}(x,x) \ge \sum_{x \in \bZ^d}
      \sum_{\Delta \in \Pi^x_\omega} \widetilde{\Theta}_{\omega,N-M,M}(\Delta,\Delta)
      \eta^{M}\left(\frac{1}{3^d}\right)^{M}\\
      & \ge \sum_{\Delta \in \Pi_M} \widetilde{\Theta}_{\omega,N-M,M}(\Delta,\Delta)\eta^{M}
      \left(\frac{1}{3^d}\right)^{M} \ge (1-\varepsilon)\eta^{M}\left(\frac{1}{3^d}\right)^{M}
       \end{split}
  \end{align}
  for every $\omega \in K(N)$.
\end{proof}

Recall the definitions of $P_N$ and $Q_N$ from \eqref{eq:defn_P_N}
respectively \eqref{eq:defn_Q_N}.
Note that for every $N \in \bN$ the measure $P_N$ is in fact the
measure $\bP$ since for every measurable event $A \in \Omega$ we have
by translation invariance
\begin{align}
  \label{eq:73}
  \begin{split}
    P_N(A)
    & = \bE\Big[\sum_{x \in \bZ^d} \bP^{(0,0)}(X_N=x) \ind{\sigma_{(x,N)} \omega \in A} \Big]=
    \sum_{x \in \bZ^d} \bP^{(0,0)}(X_N=x)\bE[\ind{\sigma_{(x,N)} \omega \in A}]\\
    & =\sum_{x \in \bZ^d} \bP^{(0,0)}(X_N =x)\bP(\sigma_{(-x,-N)}A) =
    \sum_{x \in \bZ^d} \bP^{(0,0)}(X_N =x)\bP(A) = \bP(A).
  \end{split}
\end{align}

\begin{defn}
  Given two environments $\omega,\omega'\in\Omega$ we define their
  distance by
  \begin{align*}
    \dist(\omega,\omega')=\inf \bigl\{ \norm{(x,n)}: \omega'
    = \sigma_{(x,n)}\omega \bigr\},
  \end{align*}
  where the infimum over an empty set is defined to be infinity.
\end{defn}
We denote by $\Psi_N$ the coupling of $P_N$ and $Q_N$ from
Lemma~\ref{lem:coupling_Theta} extended to $\Omega\times \Omega$, that
is,
\begin{align}
  \label{eq:defn_Psi_N}
  \Psi_N(A) = \bE\Bigl[ \sum_{x,y \in \bZ^d} \Theta_{\omega,N}(x,y)
    \ind{(\sigma_{(x,N)}\omega,\sigma_{(y,N)}\omega)\in A} \Bigr].
\end{align}

The following result is an analogue to Lemma~6.6 in
\cite{BergerCohenRosenthal2016}.
\begin{lem}
  \label{lem:upper_bound_D(1)M_and_D(2)M}
  For $M ,N \in \bN$
  let $D^{(1)}_{M,N}:\Omega \to [0,\infty]$ and
  $D^{(2)}_{M,N}:\Omega \to [0,\infty]$ be defined by
  \begin{align*}
    & D^{(i)}_{M,N}(\omega_i) \coloneqq
      \bE_{\Psi_N}[\ind{\dist (\omega_1,\omega_2)>M}
      \vert \, \mathfrak{F}_{\omega_i}](\omega_i), &i=1,2,
  \end{align*}
  where $\mathfrak{F}_{\omega_1}$, $\mathfrak{F}_{\omega_2}$ are the
  $\sigma$-algebras generated by the first, respectively, second
  coordinate in $\Omega \times \Omega$ and $\Psi_N$ is defined in
  \eqref{eq:defn_Psi_N}. For $M\in \bN$, there exists an event $F_M$
  with the following properties:
  \begin{enumerate}[(1)]
  \item $\bP(F_M) \ge 1- M^{-c\log M}$.
  \item For every $\varepsilon>0$ one can choose $M=M(\varepsilon)$
    large enough 
    \begin{align}
      \label{eq:28}
      \max\Bigl \{
      D^{(1)}_{M,N}(\omega),\,
      \frac{dQ_N}{d\bP}(\omega)D^{(2)}_{M,N}(\omega) \Bigr\} \le
      \varepsilon\indset{F_M}(\omega) + \indset{F^\compl_M}(\omega).
    \end{align}
  \end{enumerate}
\end{lem}

\begin{proof}
  Let
  \begin{align*}
    F_M = \bigcap_{k > M/2}
    \Big\{& \omega \in \Omega : \forall x \in [-k,k]^d\cap \bZ^d, \\[-2ex]
          & \qquad \qquad \sum_{\Delta \in \Pi_M} \abs{\bP^{(x,0)}(X_k \in \Delta) -
            P^{(x,0)}_\omega(X_k \in \Delta)} \le \frac{C_2}{M^{c_1}}
            + \frac{C_2}{k^{c_1}} \Big\}
  \end{align*}
  where $\Pi_M$ is a partition of $\bZ^d$ into boxes of side length
  $M$ and $C_2,c_1$ are the (renamed) constants from
  Lemma~\ref{claim:1}. Thus, $\bP(F_M) \ge 1-M^{-c\log M}$. Fix
  $\varepsilon>0$. Then, by the definition of $F_M$ and the coupling
  $\widetilde{\Theta}_{\omega,k,M}$ constructed in the proof of
  Lemma~\ref{lem:coupling_Theta}, for every $\omega \in F_M$, every
  $k > M/2$ and every $x\in[-k,k]^d\cap \bZ^d$ we have
  \begin{align}
    \label{eq:21}
    \widetilde{\Theta}_{\sigma_{(x,k)}\omega,k,M}(\Lambda_{\Pi_M})>1 -
    \frac{2C_2}{M^{c_1}}>1-\varepsilon
  \end{align}
  for large enough $M$, where
  $\Lambda_{\Pi_M}=\{(\Delta,\Delta):\Delta\in \Pi_M\}$. Note that for
  $k \le M/2$ the left hand side of \eqref{eq:21} is $1$ and therefore
  \eqref{eq:28} is trivially true for $N \le M/2$.

  \medskip Let us now verify the estimates \eqref{eq:28} for
  $D^{(1)}_{M,N}$ and $\frac{dQ_N}{d\bP} D^{(2)}_{M,N}$ and $N >M/2$.
  Note that for $\bP$-almost every environment $\omega \in \Omega$ we
  have
  \begin{align}
    \label{eq:defn_D(1)MN}
    D^{(1)}_{M,N} (\omega) = \sum_{x,y \in \bZ^d}
    \Theta_{\sigma_{-(x,N)}\omega,N}(x,y)
    \ind{\norm{x-y}>M}
  \end{align}
  and for $Q_N$-almost every $\omega$ we have
  \begin{align}
    \label{eq: defn D(2)MN}
    D^{(2)}_{M,N}(\omega)= \left(\frac{dQ_N}{d\bP}(\omega)\right)^{-1}
    \sum_{x,y \in \bZ^d}
    \Theta_{\sigma_{-(y,N)}\omega,N}(x,y)\ind{\norm{x-y}>M}.
  \end{align}
  Using \eqref{eq:defn_Psi_N} we have for every measurable event
  $A \subset \Omega$
  \begin{align*}
    E_{\Psi_N} [
    &\ind{(\omega_1,\omega_2)\in A\times \Omega}
      \ind{\dist(\omega_1,\omega_2)>M}]\\
    & = \Psi_N(A\times \Omega\cap\{
      (\omega_1,\omega_2):\dist(\omega_1,\omega_2)>M \})\\
    & =\bE\Bigl[ \sum_{x,y \in \bZ^d} \Theta_{\omega,N}(x,y)
      \ind{(\sigma_{(x,N)}\omega,\sigma_{(y,N)}\omega)\in A\times\Omega}
      \ind{\dist(\sigma_{(x,N)}\omega,\sigma_{(y,N)}\omega)>M} \Bigr]\\
    & =\sum_{x,y \in \bZ^d} \bE\bigl[\Theta_{\omega,N}(x,y)
      \ind{\sigma_{(x,N)}\omega \in A\}}
      \ind{\norm{x-y}>M} \bigr] \\
    & = \sum_{x,y \in \bZ^d} \bE\bigl[\Theta_{\sigma_{-(x,N)}\omega,N}(x,y)
      \ind{\omega \in A\}} \ind{\norm{x-y} >M}\bigr],
  \end{align*}
  where the last equality follows by translation invariance of $\bP$.
  Since $\Psi_N$ is a coupling of $P_N=\bP$ and $Q_N$ the last term
  equals
  \begin{align*}
    E_{\Psi_N}\Bigl[\ind{(\omega,\omega')\in A\times \Omega}
    \sum_{x,y \in \bZ^d} \Theta_{\sigma_{-(x,N)}\omega,N}(x,y)\ind{\norm{x-y}>M} \Bigr],
  \end{align*}
  which implies \eqref{eq:defn_D(1)MN}.

  For $B_N\coloneqq\{ \omega \colon \frac{dQ_N}{d\bP}(\omega) \neq 0 \}$ we have $Q_N(B_N^\compl)= \Psi_N(\Omega\times B_N^\compl)=0$, and we get similarly
  \begin{align*}
    E_{\Psi_N}[
    & \ind{\Omega\times A}\ind{\dist(\omega_1,\omega_2)>M}]\\
    & = E_{\Psi_N}[\ind{\Omega\times A\cap B_N} \ind{\dist(\omega_1,\omega_2)>M}]\\
    & = \Psi_N(\Omega\times (A\cap B_N) \cap \{(\omega_1,\omega_2) :
      \dist(\omega_1,\omega_2)>M \})\\
    & =\bE\Bigl[ \sum_{x,y \in \bZ^d} \Theta_{\omega,N}(x,y)
      \ind{(\sigma_{(x,N)}\omega,\sigma_{(y,N)}\omega)\in \Omega\times A \cap B_N}
      \ind{\dist(\sigma_{(x,N)}\omega,\sigma_{(y,N)}\omega)>M} \Bigr]\\
    & =\bE \Bigl[ \sum_{x,y \in \bZ^d} \Theta_{\omega,N}(x,y)
      \ind{\sigma_{(y,N)}\omega \in A\cap B_N}\ind{\norm{x-y}>M} \Bigr]\\
    & =\bE \Bigl[ \sum_{x,y \in \bZ^d} \Theta_{\sigma_{-(y,N)}\omega,N}(x,y)
      \ind{\omega \in A\cap B_N}\ind{\norm{x-y}>M} \Bigr]\\
    & = E_{Q_N} \Bigl[\Bigl(\frac{dQ_N}{d\bP}\Bigr)^{-1}(\omega)\sum_{x,y \in \bZ^d}
      \Theta_{\sigma_{-(y,N)}\omega,N}(x,y)\ind{\omega \in A\cap B_N} \ind{\norm{x-y}>M} \Bigr]\\
    & = E_{\Psi_N}\Bigl[ \Bigl(\frac{dQ_N}{d\bP} \Bigr)^{-1}(\omega_2) \sum_{x,y \in \bZ^d}
      \Theta_{\sigma_{-(y,N)}\omega_2,N}(x,y)\ind{(\omega_1,\omega_2)\in
      \Omega\times(A\cap B_N)}\ind{\norm{x-y}>M} \Bigr]\\
    & = E_{\Psi_N}\Bigl[ \Bigl(\frac{dQ_N}{d\bP} \Bigr)^{-1}(\omega_2)
      \sum_{x,y \in \bZ^d} \Theta_{\sigma_{-(y,N)}\omega_2,N}(x,y)\ind{(\omega_1,\omega_2)\in
      \Omega\times A}\ind{\norm{x-y}>M} \Bigr],
  \end{align*}
  which shows~\eqref{eq: defn D(2)MN}

  If $\Theta_{\sigma_{-(x,N)}\omega,N}(x,y)>0$ then necessarily
  $x \in [-N,N]^d\cap \bZ^d$ because in $N$ steps the annealed walk
  can only reach points in this box. It follows that for large enough
  $M$, every $\omega \in F_M$ and every $N\ge M$ we have
  \begin{align*}
    \sum_{x,y \in \bZ^d}
    &\Theta_{\sigma_{-(x,N)}\omega,N}(x,y)\ind{\norm{x-y}>M}\\
    & =1- \sum_{x,y \in \bZ^d} \Theta_{\sigma_{-(x,N)}\omega,N}(x,y)\ind{\norm{x-y}\le M}\\
    & \le 1 - \min_{z\in [-N,N]^d\cap \bZ^d} \sum_{x,y \in \bZ^d}
      \Theta_{\sigma_{-(z,N)}\omega,N}(x,y)\ind{\norm{x-y}\le M}\\
    & \le 1- \min_{z\in [-N,N]^d\cap \bZ^d} \sum_{\Delta \in \Pi_M}
      \sum_{x,y \in \Delta} \Theta_{\sigma_{-(z,N)}\omega,N}(x,y)\\
    & =1- \min_{z\in [-N,N]^d\cap \bZ^d} \sum_{\Delta \in \Pi_M}
      \widetilde{\Theta}_{\sigma_{-(z,N)}\omega,N,M}(\Delta,\Delta)\\
    & =1- \min_{z\in [-N,N]^d\cap \bZ^d}
      \widetilde{\Theta}_{\sigma_{-(z,N)}\omega,N,M}(\Lambda_{\Pi_M}) < \varepsilon.
  \end{align*}
  Thus,
  \begin{align*}
    D^{(1)}_{M,N}(\omega) = \sum_{x,y \in \bZ^d
    } \Theta_{\sigma_{-(x,N)}\omega,N}(x,y)
    \ind{\norm{x-y} >M} \le \varepsilon
    \indset{F_M}(\omega) +
    \indset{F_M^\compl}(\omega).
  \end{align*}

  For $\omega \in F_M\cap B_N$ we have shown
  \begin{align*}
    \frac{dQ_N}{d\bP}(\omega)D^{(2)}_{M,N}(\omega) =
    \sum_{x,y \in \bZ^d}
    \Theta_{\sigma_{-(y,N)}\omega,N} \ind{\norm{x-y}>M}
    \le \varepsilon
  \end{align*}
  whereas for $\omega \in F_M\cap B_N^\compl$
  \begin{align*}
  \frac{dQ_N}{d\bP}(\omega)D^{(2)}_{M,N}(\omega) = 0
  \end{align*}
  and thus
  \begin{align*}
  \frac{dQ_N}{d\bP}(\omega)D^{(2)}_{M,N}(\omega) \le \varepsilon
  \indset{F_M}(\omega) +
  \indset{F_M^\compl}(\omega).
  \end{align*}
\end{proof}

\begin{proof}[Proof of Proposition~\ref{lem:box_average_estimate}]
  We follow the ideas of the proof of Lemma~6.5 in
  \cite{BergerCohenRosenthal2016}. To this end, we consider the events
  \begin{align*}
    B^-_\varepsilon
    & =\{ \omega \in \Omega: \frac{1}{\abs{\Delta_0}}\sum_{x \in \Delta_0}
      \frac{dQ_N}{d\bP}(\sigma_{(x,0)} \omega) <1-\varepsilon \}\\
    B^+_\varepsilon
    & =\{ \omega\in \Omega:\frac{1}{\abs{\Delta_0}}
      \sum_{x \in \Delta_0}\frac{dQ_N}{d\bP}(\sigma_{(x,0)}\omega)>1+\varepsilon \}.
  \end{align*}
  First we consider $B^-_\varepsilon$. We decompose this event into
  two events, first of which has probability $M^{-c\log M}$ and the
  second is a $\bP$ null set. We assume without loss of generality
  that $\Delta_0$ is centred at the (spatial) origin, set
  $M_\varepsilon = \frac{\varepsilon}{6d^2}M$, define
  $\Delta^-_0=\{x\in\bZ^d:\norm{x} < M-M_\varepsilon \}$ and
  \begin{align*}
    S^-_\varepsilon = \{ \omega\in B^-_\varepsilon: \sigma_{(x,0)}\omega \in
    F_{M_\varepsilon},\forall x\in \Delta_0 \},
  \end{align*}
  where $F_{M_\varepsilon}$ is the event from
  Lemma~\ref{lem:upper_bound_D(1)M_and_D(2)M}. Due to property (1) of
  $F_{M_\varepsilon}$ from Lemma~\ref{lem:upper_bound_D(1)M_and_D(2)M}
  \begin{align*}
    \bP(S^-_\varepsilon)
    &\ge \bP(B^-_\varepsilon) -
      \abs{\Delta_0}\bP(F_{M_\varepsilon}^\compl) \\
    & \ge \bP(B^-_\varepsilon) - M^d(M_\varepsilon)^{-c\log M_\varepsilon}
      \ge \bP(B^-_\varepsilon) - M^{-\tilde{c}\log M},
  \end{align*}
  where $\tilde{c}$ is a positive constant. Therefore it is enough to
  show that $\bP(S^-_\varepsilon)=0$.

  We claim that there exists an event $K^- \subset S^-_\varepsilon$
  such that
  \begin{align}
    \label{eq:30}
    & \bP(K^-)  \ge \bP(S^-_\varepsilon)\cdot((4d)^d\abs{\Delta_0})^{-1}\\
    \intertext{and}
    \label{eq:31}
    & \text{if $\omega,\omega'\in K^-$, $\omega\neq \omega'$, then
    $\dist(\omega,\omega')>4M$}.
  \end{align}
  For every $(x,n)\in \bZ^d\times\bZ$ let $U_{(x,n)}$ be an
  independent (of everything else defined so far) random variable
  uniformly distributed on $[0,1]$, and define
  \begin{align*}
    K^-\coloneqq  \bigl\{\omega \in S^-_\varepsilon: \forall (x,n)\in
    4\Delta_0 \,\text{ if } \, \sigma_{(x,n)}\omega \in B^-_\varepsilon\,
    \text{ then } \, U_{(x,n)} < U_{(0,0)} \bigr\}.
  \end{align*}
  This means informally, that from each family of environments whose
  distance is smaller than $4dM$ we choose one uniformly. This implies
  that property \eqref{eq:31} for $K$ holds. Property \eqref{eq:30}
  holds because due to translation invariance of $\bP$ we have
  \begin{align*}
    \bP(S^-_\varepsilon)
    \le \bP\Bigl(\bigcup_{x\in 4d\Delta_0}\sigma_{(x,0)} K^-\Bigr)
    \le \sum_{x  \in 4d\Delta_0} \bP\bigl(\sigma_{(x,0)} K^-\bigr)
    = (4d)^d\abs{\Delta_0}\bP(K^-).
  \end{align*}

  Now, let
  \begin{align*}
    G = \bigcup_{x \in \Delta_0} \sigma_{(x,0)} K^- \quad \text{and}
    \quad G^-=\bigcup_{x \in \Delta^-_0}\sigma_{(x,0)} K^-.
  \end{align*}
  By property \eqref{eq:31} of $K^-$ these are in both cases disjoint
  unions and therefore we have
  \begin{align}
    \label{eq:prob H and H-}
    \begin{split}
      \bP(G) & = \sum_{x\in \Delta_0} \bP(\sigma_{(x,0)} K^-) =
      \abs{\Delta_0} \bP(K^-)\quad \text{ and}\\
      \bP(G^-) & = \abs{\Delta^-_0} \bP(K^-) = \abs{\Delta_0}
      \bigl(1- \frac{\varepsilon}{6d^2} \bigr)^{d}\bP(K^-)
      >\bigl(1-\frac{\varepsilon}{6} \bigr) \bP(G).
    \end{split}
  \end{align}

  Going back to the definition of the event $B^-_\varepsilon$ and
  recalling that $K^- \subset S^-_\varepsilon\subset B^-_\varepsilon$ we
  obtain
  \begin{align*}
    Q_N(G)
    & = \int_G \frac{dQ_N}{d\bP}(\omega)\,d\bP(\omega) =
      \sum_{x  \in \Delta_0} \int_{\sigma_{(x,0)} K^-} \frac{dQ_N}{d\bP}(\omega)\,d\bP(\omega)
     = \int_{K^-} \sum_{x\in \Delta_0}
      \frac{dQ_N}{d\bP}(\sigma_{(x,0)} \omega)\,d\bP(\omega)\\
    & \le \int_{K^-} (1-\varepsilon)\abs{\Delta_0}\,d\bP(\omega) = (1-\varepsilon)\abs{\Delta_0}\bP(K^-)
     =(1-\varepsilon)\bP(G)
  \end{align*}
  Combining this with \eqref{eq:prob H and H-}, for small enough
  $\varepsilon>0$ we obtain
  \begin{align}
    \label{eq:placeholder}
    Q_N(G) \le (1-\varepsilon)\bP(G)
    = \frac{1-\varepsilon}{1-\varepsilon/6}\left(1-\frac{\varepsilon}{6}\right) \bP(G)
    < \frac{1-\varepsilon}{1-\varepsilon/6}\bP(G^-)
    <\left( 1- \frac{\varepsilon}{3} \right) \bP(G^-).
  \end{align}
  Let $A^-=\{(\omega,\omega'):\omega \in G^-,\omega'\notin G\}$. Then
  by \eqref{eq:prob H and H-} and \eqref{eq:placeholder}
  \begin{align}
    \label{eq:32}
    \begin{split}
      \Psi_N(A^-)
      & \ge \bP(G^-) - Q_N(G) \ge \bP(G^-)
      - \Bigl(1-\frac{\varepsilon}{3} \Bigr)\bP(G^-)\\
      & \ge \frac{\varepsilon}{3}\bP(G^-) > \frac{\varepsilon}{3}
      \Bigl(1-\frac{\varepsilon}{6} \Bigr)\bP(G)>\frac{\varepsilon}{4}\bP(G).
    \end{split}
  \end{align}
  By construction of $K^-$, for every $(\omega,\omega')\in A^-$ we have
  $\dist(\omega,\omega')>M_\varepsilon$ and, therefore,
  \begin{align}
    \label{eq:P(H) as lower bound}
    \begin{split}
      \int_G D^{(1)}_{M_\varepsilon,N}\,d\bP(\omega)
      & = \int_{G\times \Omega} D^{(1)}_{M_\varepsilon,N}\,d\Psi_N(\omega,\omega')
      \ge \int_{G^-\times \Omega}D^{(1)}_{M_\varepsilon,N}\,d\Psi_N(\omega,\omega')\\
      & = \int_{\Omega \times \Omega} E_{\Psi_N}[\ind{\dist(\omega,\omega')>M_{\varepsilon}}\, \vert
      \, \mathfrak{F}_\omega](\omega) \ind{G^-\times
        \Omega}(\omega,\omega')\,d\Psi_N(\omega,\omega')\\
      & = \int_{\Omega\times \Omega}
      E_{\Psi_N}[\ind{\dist(\omega,\omega')>M_{\varepsilon}}
      \ind{G^-\times \Omega}(\omega,\omega')\, \vert \,
      \mathfrak{F}_\omega](\omega)\,d\Psi_N(\omega,\omega')\\
      & = \int_{\Omega\times \Omega}\ind{\dist(\omega,\omega')>M_{\varepsilon}}
      \ind{G^-\times \Omega}(\omega,\omega')\, d\Psi_N(\omega,\omega')\\
      &\ge \int_{\Omega\times \Omega}\ind{\dist(\omega,\omega')>M_{\varepsilon}}
      \ind{A^-}(\omega,\omega')\, d\Psi_N(\omega,\omega')\\
      & = \int_{\Omega \times \Omega} \indset{A^-}(\omega,\omega') \, d\Psi_N(\omega,\omega')\\
      & = \Psi_N(A^-)>\frac{\varepsilon}{4}\bP(G).
    \end{split}
  \end{align}
  Since $G\subset F_{M_\varepsilon}$ by definition, using
  Lemma~\ref{lem:upper_bound_D(1)M_and_D(2)M} with $M_\varepsilon$ and
  $\varepsilon/5$ instead of $M$ and $\varepsilon$ we obtain
  \begin{align}
    \label{eq:35}
    \int_G D^{(1)}_{M_\varepsilon,N} (\omega) \,d\bP(\omega)
    \le \int_G \frac{\varepsilon}{5}\mathbbm{1}_{F_{M_\varepsilon}}(\omega)
    + \mathbbm{1}_{F^\compl_{M_\varepsilon}}(\omega)\,d \bP(\omega)
    = \int_G \frac{\varepsilon}{5}\,d\bP(\omega)=\frac{\varepsilon}{5}\bP(G).
  \end{align}
  Combining \eqref{eq:P(H) as lower bound} and \eqref{eq:35} we
  conclude that $\bP(G)=0$ and, therefore $\bP(K^-)=0$. By property
  \eqref{eq:30} of $K^-$ this implies that $\bP(S^-_\varepsilon)=0$
  and finally $\bP(B^-_\varepsilon) \le M^{-c\log M}$.

  Next we turn to the event $B^+_\varepsilon$. As before we set
  $M_\varepsilon=\frac{\varepsilon}{6d^2}M$ and assume that $\Delta_0$
  is centred at the origin. Define
  $\Delta^+_0\coloneqq \{x\in \bZ^d :\norm{x} < M + M_\varepsilon \}$
  and let
  \begin{align*}
    S^+_\varepsilon = \bigl\{\omega\in B^+_\varepsilon: \sigma_{(x,0)}\omega
    \in F_{M_\varepsilon},\forall x\in\Delta^+_0 \bigr\},
  \end{align*}
  where $F_{M_\varepsilon}$ is, as before, the event from
  Lemma~\ref{lem:upper_bound_D(1)M_and_D(2)M}. Due to property (1) of
  $F_{M_\varepsilon}$
  \begin{align*}
    \bP(S^+_\varepsilon)
    &\ge \bP(B^+_\varepsilon) -
      \abs{\Delta^+_0}\bP(F^\compl_{M_\varepsilon})
      \ge  \bP(B^+_\varepsilon)
      -(1+\frac{\varepsilon}{6d^2})^{d}M^{d}(M_\varepsilon)^{-c\log M_\varepsilon}\\
    & \ge \bP(B^+_\varepsilon) - M^{-\tilde{c}\log M}
  \end{align*}
  and again it is enough to show that $\bP(S^+_\varepsilon)=0$. As for
  $S^-_\varepsilon$ we claim that there exists an event $K^+\subset
  S^+_\varepsilon$ such that
  \begin{align}
    \label{eq:33}
    & \bP(K^+)\ge \bP(S^+_\varepsilon)\cdot((4d)^{d}\abs{\Delta^+_0})^{-1}\\
    \intertext{and}
    \label{eq:34}
    & \text{if $\omega,\omega'\in K^+$ with $\omega\neq \omega'$, then
    $\dist(\omega,\omega')>4(M+M_\varepsilon)$.}
  \end{align}

  Let
  \begin{align*}
    H = \bigcup_{x \in \Delta_0}\sigma_{(x,0)} K^+ \quad \text{ and}
    \quad  H^+=\bigcup_{x \in \Delta^+_0} \sigma_{(x,0)} K^+.
  \end{align*}
  Both are, by property \eqref{eq:34} of $K^+$ disjoint unions.
  Therefore we have for $\varepsilon>0$ small enough
  \begin{align}
    \label{eq:prob_H_and_H+}
    \begin{split}
      \bP(H) & = \abs{\Delta_0}\bP(K^+)\quad \text{ and}\\
      \bP(H^+) & = \abs{\Delta^+_0}\bP(K^+)=\left(
        1+\frac{\varepsilon}{6d^2} \right)^{d}\abs{\Delta_0}\bP(K^+)
      < \left( 1 +\frac{\varepsilon}{5} \right)\bP(H).
    \end{split}
  \end{align}
  From $K^+ \subset S^+_\varepsilon\subset B^+_\varepsilon$ we obtain
  \begin{align}
    \label{eq:placeholder2}
    \begin{split}
      Q_N(H)
      & = \int_H \frac{dQ_N}{d\bP}(\omega)\,d\bP(\omega)
      = \sum_{x \in \Delta_0}\int_{\sigma_{(x,0)} K^+}
      \frac{dQ_N}{d\bP}(\omega)\,d\bP(\omega)\\
      & = \int_{K^+} \sum_{x \in \Delta_0}
      \frac{dQ_N}{d\bP}(\sigma_{(x,0)} \omega)\, d\bP(\omega)\\
      &> \int_{K^+} \abs{\Delta_0}(1+\varepsilon)\,d \bP(\omega)
      =(1+\varepsilon)\abs{\Delta_0}\bP(K^+) = (1+\varepsilon) \bP(H).
    \end{split}
  \end{align}
  Combination of this with \eqref{eq:prob_H_and_H+}, for small enough
  $\varepsilon>0$ then yields
  \begin{align}
    \label{eq:placeholder3}
      Q_N(H) &> (1+\varepsilon)\bP(H) =
      \frac{1+\varepsilon}{1+\varepsilon/5}\left( 1+ \frac{\varepsilon}{5}
      \right)\bP(H) > \frac{1+\varepsilon}{1+\varepsilon/5}\bP(H^+)
      >\left( 1+ \frac{\varepsilon}{3} \right) \bP(H^+).
  \end{align}
  Let $A^+ \coloneqq \{(\omega,\omega'):\omega \notin H^+,\omega'\in H \}$. Then
  by \eqref{eq:placeholder3}
  \begin{align}
    \begin{split}
      \Psi_N(A^+) &\ge Q_N(H) - \bP(H^+) > Q_N(H) -
      \frac{1}{1+\varepsilon/3}Q_N(H) =
      \frac{\varepsilon/3}{1+\varepsilon/3}Q_N(H)\\
      &\ge \frac{\varepsilon}{4}Q_N(H).
    \end{split}
  \end{align}
  By the construction of $K^+$, for every $(\omega,\omega')\in A^+$ we
  have $\dist(\omega,\omega')>M_\varepsilon$ and, therefore,
  \begin{align}
    \label{eq:Q(H)_as_lower_bound}
    \begin{split}
      \int_H D^{(2)}_{M_\varepsilon,N}(\omega)\,dQ_N(\omega)
      & = \int_{\Omega\times H} D^{(2)}_{M_\varepsilon,N}(\omega')
      \,d\Psi_N(\omega,\omega')\\
      & = \int_{\Omega\times\Omega} D^{(2)}_{M_\varepsilon,N}(\omega')
      \ind{\Omega\times H}(\omega,\omega')\,d\Psi_N(\omega,\omega')\\
      & = \int_{\Omega\times\Omega} E_{\Psi_N}[\ind{\dist(\omega,\omega')>M_{\varepsilon}}\,\vert\,
      \mathfrak{F}_{\omega'}](\omega') \ind{\Omega\times H}(\omega,\omega')\,d\Psi_N(\omega,\omega')\\
      & = \int_{\Omega\times\Omega} E_{\Psi_N}[\ind{\dist(\omega,\omega')>M_{\varepsilon}}\ind{\Omega\times
        H}(\omega,\omega')\,\vert\, \mathfrak{F}_{\omega'}](\omega') \,d\Psi_N(\omega,\omega')\\
      & = \int_{\Omega \times \Omega} \ind{\dist(\omega,\omega')>M_{\varepsilon}}\ind{\Omega\times
        H}(\omega,\omega')\,d\Psi_N(\omega,\omega')\\
      &\ge \int_{\Omega \times \Omega} \ind{\dist(\omega,\omega')>M_{\varepsilon}}
      \indset{A^+}(\omega,\omega')\,d\Psi_N(\omega,\omega')\\
      & = \int_{\Omega \times \Omega} \indset{A^+}(\omega,\omega')\,d\Psi_N(\omega,\omega')\\
      & = \Psi_N(A^+) \ge \frac{\varepsilon}{4}Q_N(H).
    \end{split}
  \end{align}
  Since $H\subset F_{M_\varepsilon}$ by definition, $\bP(H)\le Q_N(H)$
  by \eqref{eq:placeholder2}, and using
  Lemma~\ref{lem:upper_bound_D(1)M_and_D(2)M} with $M_\varepsilon$ and
  $\frac{\varepsilon}{5}$ instead of $M$ and $\varepsilon$ we obtain
  \begin{align}
    \label{eq:Q(H)_as_upper_bound}
    \begin{split}
      \int_H D^{(2)}_{M_\varepsilon,N}\,dQ_N(\omega)
      & \le \int_{H\cap B_N} \Bigl(\frac{dQ_N}{d\bP}\Bigr)^{-1}
      \Bigl[\frac{\varepsilon}{5}\indset{F_{M_\varepsilon}\cap B_N} +
      \indset{(F_{M_\varepsilon}\cap B_N)^\compl} \Bigr]\,dQ_N(\omega)\\
      & = \int_{H\cap B_N} \Bigl(\frac{dQ_N}{d\bP}\Bigr)^{-1}
      \Bigl[\frac{\varepsilon}{5}\indset{F_{M_\varepsilon}\cap B_N} +
        \indset{(F_{M_\varepsilon}\cap B_N)^\compl} \Bigr]\,dQ_N(\omega)\\
      & = \int_{H\cap B_N} \Bigl[\frac{\varepsilon}{5}\indset{F_{M_\varepsilon}\cap B_N} +
        \indset{(F_{M_\varepsilon}\cap B_N)^\compl} \Bigr]\,d\bP(\omega)\\
      & = \int_{H\cap B_N} \frac{\varepsilon}{5}\,d\bP(\omega)\\
      & = \frac{\varepsilon}{5}\bP(H\cap B_N) \le
      \frac{\varepsilon}{5}\bP(H)
      \le \frac{\varepsilon}{5}Q_N(H),
    \end{split}
  \end{align}
  where we recall from Lemma~\ref{lem:upper_bound_D(1)M_and_D(2)M}
  that $B_N=\{\omega: \frac{dQ_N}{d\bP}(\omega)\neq 0 \}$ and note
  that $B_N^\compl$ is a $Q_N$ null set. Combining
  \eqref{eq:Q(H)_as_lower_bound} and \eqref{eq:Q(H)_as_upper_bound},
  we conclude that $Q_N(H)=0$ and, therefore, by
  \eqref{eq:placeholder2} we have $\bP(H)=0$. It follows that
  $\bP(K^+)=0$, which by property \eqref{eq:33} of $K^+$ implies that
  $\bP(S^+_\varepsilon)=0$ and finally that \eqref{eq:29} holds.
\end{proof}

\begin{proof}[Proof of Corollary~\ref{cor:con_prop_Q}]
  To show that Proposition~\ref{lem:box_average_estimate} holds for
  $Q$ as well we define $\Psi$ as the weak limit of
  $\{\frac{1}{n} \sum_{N=0}^{n-1}\Psi_N \}_{n=1}^\infty$ along any
  converging sub-sequence $\{n_k\}_{k\geq 1}$ (tightness of $\Psi_N$
  follows similarly to the discussion below
  Corollary~\ref{cor:finite_moments_prefactor}). Note that $\Psi$ is a
  coupling of $\bP$ and $Q$ on $\Omega\times\Omega$. Furthermore let
  \begin{align*}
    &D^{(i)}_M(\omega_i) :=
      E_\Psi[\indset{\dist(\omega_1,\omega_2)>dM}\,\vert\,
      \mathcal{F}_{\omega_i}](\omega_i), &i=1,2.
  \end{align*}

  Now we want to prove inequality \eqref{eq:28} from
  Lemma~\ref{lem:upper_bound_D(1)M_and_D(2)M} for $D^{(1)}_M$ and
  $D^{(2)}_M$. It is enough to show that along some sub-sequence
  $\{n_\ell\}_{l\ge 1}$ of $\{n_k\}_{k\ge 1}$
  \begin{align}
    \label{eq:85}
    D^{(1)}_M(\omega) = \lim_{\ell\to\infty} \frac{1}{n_\ell}
    \sum_{N=0}^{n_\ell-1} D^{(1)}_{M,N}(\omega) \quad \bP\text{-a.s.}
  \end{align}
  and
  \begin{align}
    \label{eq:86}
    D^{(2)}_M(\omega) =\Big( \frac{dQ}{d\bP}(\omega) \Big)^{-1}
    \lim_{\ell\to\infty} \frac{1}{n_\ell}\sum_{N=0}^{n_\ell-1}
    \frac{dQ_N}{d\bP}(\omega)D^{(2)}_{M,N}(\omega) \quad Q\text{-a.s.}
  \end{align}
  In fact, if the above equalities hold, then for $\bP$-almost every
  $\omega$ we have
  \begin{align*}
    D^{(1)}_M(\omega)
    & = \lim_{\ell\to\infty} \frac{1}{n_\ell} \sum_{N=0}^{n_\ell-1} D^{(1)}_{M,N}(\omega)\\
    & =\lim_{\ell\to\infty} \frac{1}{n_\ell} \Big[
      \sum_{N=0}^{M-1}D^{(1)}_{M,N}(\omega) +
      \sum_{N=M}^{n_\ell-1}D^{(1)}_{M,N}(\omega) \Big]\\
    & \leq \lim_{\ell\to\infty} \frac{1}{n_\ell} \Big[ M + \sum_{N=M}^{n_\ell-1}D^{(1)}_{M,N}(\omega)\Big]\\
    & \leq \lim_{\ell\to\infty} \frac{1}{n_\ell} \Big[ M +
      \sum_{N=M}^{n_\ell-1} (\varepsilon \indset{F_M}(\omega) +
      \indset{F_M^\compl}(\omega))  \Big]\\
    & =\varepsilon \indset{F_M}(\omega) + \indset{F_M^\compl}(\omega).
  \end{align*}
  In addition for $D^{(2)}_M$ we have for $Q$ almost all $\omega$
  \begin{align*}
    \frac{dQ}{d\bP}(\omega) D^{(2)}_M (\omega)
    & = \lim_{\ell\to\infty} \frac{1}{n_\ell} \sum_{N=0}^{n_\ell-1} \frac{dQ_N}{d\bP}(\omega)D^{(2)}_{M,N}(\omega)\\
    & \leq \lim_{\ell\to\infty} \frac{1}{n_\ell} \Big[
      \sum_{N=0}^{M-1}\frac{dQ_N}{d\bP}(\omega) + \sum_{N=M}^{n_\ell-1}
      \frac{dQ_N}{d\bP}(\omega)D^{(2)}_{M,N}(\omega)  \Big]\\
    & \leq \lim_{\ell\to\infty} \frac{1}{n_\ell}
      \Big[\sum_{N=0}^{M-1}\frac{dQ_N}{d\bP}(\omega) +
      \sum_{N=M}^{n_\ell-1}(\varepsilon \indset{F_M}(\omega) +
      \indset{F_M^\compl}(\omega)) \Big]\\
    & \leq \varepsilon \indset{F_M}(\omega) + \indset{F_M^\compl}(\omega).
  \end{align*}

  Let us now prove \eqref{eq:85} and \eqref{eq:86}. Starting with
  \eqref{eq:85} let $A\subset \Omega$ be a measurable event. We have
  \begin{align*}
    \bE&[ D^{(1)}_M(\omega_1)\indset{A}(\omega_1) ] \\
       &=E_\Psi[\ind{\dist(\omega_1,\omega_2)>dM} \indset{A\times\Omega}(\omega_1,\omega_2)]\\
       &=\Psi(\{(\omega_1,\omega_2)\in\Omega\times\Omega:\dist(\omega_1,\omega_2)>dM \}\cap A\times\Omega)\\
       &=\lim_{\ell\to\infty} \frac{1}{n_\ell} \sum_{N=0}^{n_\ell-1}\Psi_N(\{(\omega_1,\omega_2)\in\Omega\times\Omega:\dist(\omega_1,\omega_2)>dM \}\cap A\times\Omega)\\
       &=\lim_{\ell\to\infty} \frac{1}{n_\ell} E_{\Psi_N}[\ind{\dist(\omega_1,\omega_2)>dM} \indset{A\times\Omega}(\omega_1,\omega_2)]\\
       &=\lim_{\ell\to\infty}      \frac{1}{n_\ell}\sum_{N=0}^{n_\ell-1} \bE[D^{(1)}_{M,N}(\omega_1)\indset{A}(\omega_1)]\\
       &=\lim_{\ell\to\infty} \bE\Big[\frac{1}{n_\ell}\sum_{N=0}^{n_\ell-1}D^{(1)}_{M,N}(\omega_1)\indset{A}(\omega_1)\Big]
  \end{align*}
  where we used the definitions of $\Psi$ and of $D^{(1)}_{M,N}$ as
  the conditional expectation. This implies convergence of
  $\frac{1}{n_\ell}\sum_{N=0}^{n_\ell-1}D^{(1)}_{M,N}$ to $D^{(1)}_M$
  in $L^1(\bP)$. Thus, by standard arguments we can choose a
  subsequence that converges $\bP$-almost surely. For $D^{(2)}_M$ we
  get in a similar way
  \begin{align*}
    E_Q&[D^{(2)}_M(\omega)\indset{A}(\omega_2)]\\
       &=E_\Psi[\indset{\dist(\omega_1,\omega_2)>dM}\indset{\Omega\times A}(\omega_1,\omega_2)]\\
       &=\Psi(\{\dist(\omega_1,\omega_2)>dM\}\cap \Omega\times A) \\
       &=\lim_{\ell\to\infty} \frac{1}{n_\ell} \sum_{N=0}^{n_\ell-1}\Psi_{N}(\{\dist(\omega_1,\omega_2)>dM\}\cap \Omega\times A)\\
       &=\lim_{\ell\to\infty} \frac{1}{n_\ell} \sum_{N=0}^{n_\ell-1} E_{\Psi_N}[\indset{\dist(\omega_1,\omega_2)>dM}\indset{\Omega\times A}(\omega_1,\omega_2)]\\
       &=\lim_{\ell\to\infty} \frac{1}{n_\ell} \sum_{N=0}^{n_\ell-1}E_{\Psi_N}[D^{(2)}_{M,N}(\omega_2)\indset{\Omega\times A}(\omega_1,\omega_2)]\\
       &=\lim_{\ell\to\infty}\frac{1}{n_\ell} \sum_{N=0}^{n_\ell-1}E_{Q_N}[D^{(2)}_{M,N}(\omega_2)\indset{ A}(\omega_2)]\\
       &=\lim_{\ell\to\infty}\frac{1}{n_\ell} \sum_{N=0}^{n_\ell-1}E_{Q}\Big[\big( \frac{dQ}{d\bP}(\omega_2) \big)^{-1}\cdot \frac{dQ_N}{d\bP}(\omega_2)\cdot D^{(2)}_{M,N}(\omega_2)\indset{ A}(\omega_2)\Big]\\
       &=\lim_{\ell\to\infty}E_Q\Big[\big( \frac{dQ}{d\bP}(\omega_2) \big)^{-1}\cdot\frac{1}{n_\ell} \sum_{N=0}^{n_\ell-1} \frac{dQ_N}{d\bP}(\omega_2)\cdot D^{(2)}_{M,N}(\omega_2)\indset{ A}(\omega_2)\Big]
  \end{align*}
  $Q$-almost surely. Thus, Lemma~\ref{lem:upper_bound_D(1)M_and_D(2)M}
  holds for $D^{(1)}_M$ and $D^{(2)}_M$ instead of $D^{(1)}_{M,N}$ and
  $D^{(2)}_{M,N}$ respectively.

  Since the only tools we need for the proof of
  Proposition~\ref{lem:box_average_estimate} are Lemma~\ref{claim:1}
  and Lemma~\ref{lem:upper_bound_D(1)M_and_D(2)M}, we can walk through
  the proof of Proposition~\ref{lem:box_average_estimate} and repeat
  the same steps for $\frac{dQ}{d\bP}$ to show
  Corollary~\ref{cor:con_prop_Q}.
\end{proof}

The following proposition is an analogue to Proposition~7.1 from
\cite{BergerCohenRosenthal2016}. Note that the assertion
  expresses a general property of the density of a measure which is
  invariant for the point of view of the particle in the setting of a
  random walk in random environment. It is not model-specific.

\begin{prop}
  \label{prop:2}
  For $\bP$-almost every $\omega$, every $n \in\bN_0$, every $x\in\bZ^d$
  and all $k\le n$
  \begin{align*}
    \varphi(\sigma_{(x,n)}\omega) =  \sum_{y\in\bZ^d}
    P^{(x+y,n-k)}_\omega(X_n=x) \varphi(\sigma_{(x+y,n-k)}\omega).
  \end{align*}
\end{prop}
\begin{proof}
  Let $n\in\bN$. First we consider the case $k=1$. For every bounded
  measurable function $h:\Omega\to \bR$ we have (recall the notation
  in~\eqref{eq:transitionKernel} and \eqref{eq:19})
  \begin{align*}
    \int_\Omega h(\omega)\varphi(\sigma_{(x,n)}\omega)\,d\bP(\omega)
    & =\int_\Omega h(\sigma_{(-x,-n)}\omega)\varphi(\omega)\,d\bP(\omega)\\
    &=\int_\Omega h(\sigma_{(-x,-n)}\omega)\,dQ(\omega)\\
    &=\int_\Omega \mathfrak{R}h(\sigma_{(-x,-n)}\omega)\,dQ(\omega)\\
    &=\int_\Omega (\mathfrak{R}h(\sigma_{(-x,-n)}\omega))\varphi(\omega)\,d\bP(\omega)\\
    &=\int_\Omega \sum_{\norm{y}\le 1}
      g(\omega,y)h(\sigma_{(-x+y,1-n)}\omega)\varphi(\omega)\,d\bP(\omega)\\
    &=\int_\Omega \sum_{\norm{y}\le 1}
      g(\sigma_{(x-y,n-1)}\omega,y)h(\omega)\varphi(\sigma_{(x-y,n-1)}\omega)\,d\bP(\omega).
  \end{align*}
  Thus
  \begin{align*}
    \varphi(\sigma_{(x,n)}\omega)
    & = \sum_{\norm{y}\le 1} g(\sigma_{(x-y,n-1)}\omega)\varphi(\sigma_{(x-y,n-1)}\omega)\\
    & = \sum_{\norm{y}\le 1}
      P^{(0,0)}_{\sigma_{(x-y,n-1)}\omega}(X_1=y)\varphi(\sigma_{(x-y,n-1)}\omega)\\
    & = \sum_{\norm{y}\le 1} P^{(x-y,n-1)}_\omega(X_1
      =x)\varphi(\sigma_{(x-y,n-1)}\omega)\\
    & = \sum_{y\in\bZ^d} P^{(x+y,n-1)}_\omega(X_1
      =x)\varphi(\sigma_{(x+y,n-1)}\omega).
  \end{align*}
  By applying the operator $\mathfrak{R}$ a second time we see that
  \begin{align*}
    \int_\Omega h(\omega)\varphi(\sigma_{(x,n)}\omega)\,d\bP
    & = \int_\Omega h(\omega) \sum_{\norm{y_1}\le 1}
      P^{(x+y_1,n-1)}_\omega(X_1 =x)\varphi(\sigma_{(x+y_1,n-1)}\omega)\,d\bP(\omega)\\
    & = \int_\Omega h(\sigma_{(-x-y_1,-n+1)}\omega)
      \sum_{\norm{y_1}\le 1} P^{(x+y_1,n-1)}_{\sigma_{(-x-y_1,-n+1)}\omega}(X_1 =x)
      \varphi(\omega)\,d\bP(\omega)\\
    & = \int_\Omega \Bigl[\Bigl(\mathfrak{R}(h(\sigma_{(-x-y_1,-n+1)}\omega)
      \sum_{\norm{y_1}\le 1}
      P^{(x+y_1,n-1)}_{\sigma_{(-x-y_1,-n+1)}\omega} (X_1 =x))\Bigr)\Bigr]
      \varphi(\omega)\,d\bP(\omega)\\
    & = \int_\Omega \sum_{\norm{y_2}\le 1} g(\omega,y_2) h(\sigma_{(-x-y_1+y_2,-n+2)}\omega) \\
    &\hspace{2cm} \sum_{\norm{y_1}\le 1}
      P^{(x+y_1,n-1)}_{\sigma_{(-x-y_1+y_2,-n+2)}\omega}(X_1 =x)
      \varphi(\omega)\,d\bP(\omega)\\
    & = \int_\Omega \sum_{\norm{y_2}\le 1}
      g(\sigma_{(x+y_1-y_2,n-2)}\omega,y_2) h(\omega) \\
    &\hspace{2cm} \sum_{\norm{y_1}\le 1}
      P^{(x+y_1,n-1)}_\omega(X_1
      =x)\varphi(\sigma_{(x+y_1-y_2,n-2)}\omega)\,d\bP(\omega)\\
    & = \int_\Omega \sum_{\norm{y_2}\le 1} P^{(x+y_1+y_2,n-2)}_\omega (X_1=x+y_1)\\
    & \hspace{2cm}\sum_{\norm{y_1}\le 1}
      P^{(x+y_1,n-1)}_\omega(X_1=x)h(\omega)
      \varphi(\sigma_{(x+y_1+y_2,n-2)}\omega)\,d\bP(\omega).
  \end{align*}
  Thus,
  \begin{align*}
    \varphi(\sigma_{(x,n)}\omega)
    & = \sum_{\norm{y_1}\le 1} \sum_{\norm{y_2}\le 1}
      P^{(x+y_1+y_2,n-2)}_\omega (X_1=x+y_1)
      P^{(x+y_1,n-1)}_\omega(X_1=x)\varphi(\sigma_{(x+y_1+y_2,n-2)}\omega)\\
    & = \sum_{y\in\bZ^d} P^{x+y,n-2}_\omega(X_2=x)
      \varphi(\sigma_{(x+y,n-2)}\omega).
  \end{align*}
  Inductively we obtain
  \begin{align*}
    \varphi(\sigma_{(x,n)}\omega) = \sum_{y\in\bZ^d}
    P^{(x+y,n-k)}_\omega(X_k=x)\varphi(\sigma_{(x+y,n-k)}\omega)
  \end{align*}
  for all $k\le n$.
\end{proof}

\section{Proof of Proposition~\ref{lem:limit_Z_omega}}
\label{sec:proof-prop-refl}
Let $\Pi$ be a partition of $\bZ^d$ into boxes of side length
$\lfloor n^\delta \rfloor$ with $0<\delta<\frac{1}{6d}$.
Since $\bP^{(0,0)}(X_n=x)=0$ for $\norm{x}>n$ only boxes in
$\Pi_n\coloneqq \{\Delta\in\Pi:\Delta\cap[-n,n]^d\neq \emptyset \}$
have to be considered. We have
\begin{align}
  \label{eq:37}
  \begin{split}
    \abs{Z_{\omega,n}-1}
    & = \Abs{\sum_{x\in\bZ^d}\bP^{(0,0)}(X_n=x)[\varphi(\sigma_{(x,n)}\omega)-1]}\\
    & = \Abs{\sum_{\Delta\in\Pi_n}\sum_{x\in\Delta}
      \bP^{(0,0)}(X_n=x)[\varphi(\sigma_{(x,n)}\omega)-1]}.
  \end{split}
\end{align}
By the annealed CLT from \cite{BirknerCernyDepperschmidtGantert2013}
for any $\varepsilon >0$ there exists a constant $C_\varepsilon>0$
such that
\begin{align*}
  \bP^{(0,0)}(\norm{X_n}\ge C_\varepsilon \sqrt{n})<\varepsilon
\end{align*}
We want to use this fact below and separate the sum in the last line
of \eqref{eq:37} into boxes in
$\widehat{\Pi}_n=\{ \Delta\in\Pi_n : \Delta\cap\{
x\in\bZ^d:\norm{x}\le C_\varepsilon\sqrt{n} \}\neq \emptyset \}$
and in $\Pi_n\setminus\widehat{\Pi}_n$. Using triangle inequality we
obtain
\begin{align}
  \label{eq:normalizing_constant_1}
  \abs{Z_{\omega,n}-1}
  & \le \Abs{\sum_{\Delta\in\Pi_n\setminus\widehat{\Pi}_n}
    \sum_{x\in\Delta}\bP^{(0,0)}(X_n=x) [\varphi(\sigma_{(x,n)}\omega)-1]} \\
  \label{eq:normalizing_constant_2}
  &\; +\Abs{\sum_{\Delta\in\widehat{\Pi}_n}\sum_{x  \in \Delta}
    \Bigl( \frac{1}{\abs{\Delta}}\sum_{y \in \Delta}
    [\bP^{(0,0)}(X_n=y)-\bP^{(0,0)}(X_n=x)]\Bigr)
    [\varphi(\sigma_{(x,n)}\omega)-1]}\\
  \label{eq:normalizing_constant_3}
  &\; +\Abs{\sum_{\Delta\in\widehat{\Pi}_n}\sum_{x  \in \Delta} \frac{1}{\abs{\Delta}}
    \sum_{y \in \Delta}\bP^{(0,0)}(X_n=y)[\varphi(\sigma_{(x,n)}\omega)-1]}.
\end{align}
We start with an upper bound of \eqref{eq:normalizing_constant_1}. By
Corollary~\ref{cor:con_prop_Q} there exists a constant
$C$, such that, due to translation invariance of $\bP$, with $\bP$
probability of a least $1-Cn^{-c\log n}$ for every $\Delta\in\Pi_n$ we
have
$\sum_{y\in\Delta}[\varphi(\sigma_{(y,n)}\omega)+1]\le C\abs{\Delta}$.
Under this event we can bound \eqref{eq:normalizing_constant_1} from
above by
\begin{align*}
  \sum_{\Delta\in\Pi_n\setminus\widehat{\Pi}_n}
  \sum_{x\in\Delta}\bP^{(0,0)}(X_n=x)[\varphi(\sigma_{(x,n)}\omega)+1]
  \le C\sum_{\Delta\in\Pi_n\setminus\widehat{\Pi}_n}
  \max_{x\in\Delta}\bP^{(0,0)}(X_n=x)\abs{\Delta}.
\end{align*}
Using Lemma~\ref{lem:additional_annealed_estimate} with $\delta>0$
replacing $\varepsilon$ there we see that
\eqref{eq:normalizing_constant_1} is bounded from above by
\begin{align*}
  C \sum_{\Delta \in \Pi_n\setminus\widehat{\Pi}_n}
  & \sum_{y\in\Delta}\Bigl[\max_{x\in\Delta}\bP^{(0,0)}(X_n=x)-\bP^{(0,0)}(X_n=y) \Bigr ]
    + C \sum_{\Delta \in \Pi_n\setminus\widehat{\Pi}_n} \sum_{y\in\Delta}\bP^{(0,0)}(X_n=y)\\
  & \le C\varepsilon +C \sum_{\Delta \in \Pi_n} \sum_{y\in\Delta}
    \Bigl[ \max_{x\in\Delta}\bP^{(0,0)}(X_n=x)-\bP^{(0,0)}(X_n=y)
    \Bigr] \\
  & \le C\varepsilon + Cn^{-\frac{1}{2}+3d\delta}.
\end{align*}
Since $\delta<\frac{1}{6d}$ it follows by the Borel-Cantelli lemma
that
\begin{align}
  \label{eq:38}
  \limsup_{n\to \infty} \Abs{\sum_{\Delta\in\Pi_n\setminus\widehat{\Pi}_n}
  \sum_{x\in\Delta}\bP^{(0,0)}(X_n=x) [\varphi(\sigma_{(x,n)}\omega)-1]}
  \le C\varepsilon, \qquad \bP\text{-a.s.}
\end{align}
Next we turn to \eqref{eq:normalizing_constant_2}. First note that by
the annealed derivative estimates from
Lemma~\ref{lem:annealed_derivative_estimates} we have for
$x,y \in \Delta$, $ \Delta\in \widehat\Pi_n$
\begin{align}
  \label{eq:39}
  \abs{\bP^{(0,0)}(X_n=x)-\bP^{(0,0)}(X_n=y)}
  \le C\norm{x-y} n^{-\frac{d+1}{2}}\le C n^{-\frac{d+1}{2}+\delta}.
\end{align}
By triangle inequality, \eqref{eq:39} and again, as above, using
Corollary~\ref{cor:con_prop_Q} for the bound
$\sum_{y\in\Delta}[\varphi(\sigma_{(y,n)}\omega)+1]\le C\abs{\Delta}$
the expression \eqref{eq:normalizing_constant_2} is bounded from above
by
\begin{align*}
  \sum_{\Delta\in\widehat{\Pi}_n}\sum_{x\in\Delta}
  & \frac{1}{\abs{\Delta}}\sum_{y\in\Delta}\abs{\bP^{(0,0)}(X_n=y)-\bP^{(0,0)}(X_n=x)}
    \abs{\varphi(\sigma_{(x,n)}\omega)-1}\\
  & \le Cn^{-\frac{d+1}{2}+\delta}
    \sum_{\Delta\in\widehat{\Pi}_n}
    \sum_{x\in\Delta}\frac{1}{\abs{\Delta}} \sum_{y\in\Delta}
    \left[\varphi(\sigma_{(x,n)}\omega)+1\right]\\
  & \le Cn^{-\frac{d+1}{2} +\delta} \sum_{\Delta\in\widehat{\Pi}_n} \sum_{y\in\Delta} C\\
  & \le \widetilde{C} (C_\varepsilon\sqrt{n})^d
    n^{-\frac{d+1}{2} +\delta}
    \le \widehat{C}_\varepsilon n^{-\frac{1}{2}+\delta}.
\end{align*}
with probability at least $1-Cn^{-c\log n}$. Thus, as $n\to\infty$, by
the Borel-Cantelli lemma the expression
\eqref{eq:normalizing_constant_2} tends to $0$ $\bP$-almost surely.

\smallskip

Finally we consider \eqref{eq:normalizing_constant_3}. By triangle
inequality and $\bP^{(0,0)}(X_n=y)\le Cn^{-d/2}$ for all $y$ we have
\begin{align*}
  \Abs{\sum_{\Delta\in\widehat{\Pi}_n}\sum_{x  \in \Delta}
  & \frac{1}{\abs{\Delta}}\sum_{y \in
    \Delta}\bP^{(0,0)}(X_n=y)[\varphi(\sigma_{(x,n)}\omega)-1]}\\
  & \le \sum_{\Delta\in\widehat{\Pi}_n} \frac{1}{\abs{\Delta}}\sum_{y\in\Delta}
    \bP^{(0,0)}(X_n=y)\Abs{\sum_{x\in\Delta}[\varphi(\sigma_{(x,n)}\omega)-1]}\\
  & \le Cn^{-d/2}\sum_{\Delta \in \widehat{\Pi}_n}
    \Abs{\sum_{x\in\Delta} [\varphi(\sigma_{(x,n)}\omega)-1]}\\
  & = Cn^{-d(1/2-\delta)}\sum_{\Delta\in\widehat{\Pi}_n}
    \frac{1}{\abs{\Delta}} \Abs{\sum_{x\in\Delta} [\varphi(\sigma_{(x,n)}\omega)-1]}.
\end{align*}
Using Corollary~\ref{cor:con_prop_Q} we obtain
\begin{align*}
  \bP\Bigl(Cn^{d(1/2-\delta)}\sum_{\Delta\in\widehat{\Pi}_n}
  & \frac{1}{\abs{\Delta}}\Abs{\sum_{x\in\Delta}
    [\varphi(\sigma_{(x,n)}\omega)-1] }>\varepsilon\Bigr)\\
  &\le \bP\Bigl( \exists \Delta\in\widehat{\Pi}_n:
    \sum_{\Delta\in\widehat{\Pi}_n}\frac{1}{\abs{\Delta}}\Abs{\sum_{x\in\Delta}
    [\varphi(\sigma_{(x,n)}\omega)-1] }>\frac{\varepsilon}{C C_\varepsilon^d}\Bigr)\\
  &\le n^{-d(1/2-\delta)}\bP\Bigl(\frac{1}{\abs{\Delta_0}}\abs{\sum_{x\in\Delta_0}
    [\varphi(\sigma_{(x,n)}\omega)-1] }>\frac{\varepsilon}{C C_\varepsilon^d}\Bigr)\\
  &\le n^{-d(1/2-\delta)} n^{-c\delta^2\log n} \le \tilde{C}n^{-\tilde{c}\log n},
\end{align*}
where $\Delta_0\in\widehat{\Pi}_n$ is an arbitrarily fixed box. Thus,
for $\varepsilon >0$ as $n\to\infty$ the $\limsup$ of
\eqref{eq:normalizing_constant_3} is bounded from above by $\varepsilon$
$\bP$-almost surely.
Combining all three bounds of
\eqref{eq:normalizing_constant_1}--\eqref{eq:normalizing_constant_3},
we see that there is a constant $\widehat C$ so that for all
$\varepsilon>0$
\begin{align*}
  \limsup_{n\to\infty} \abs{Z_{\omega,n}-1}\le \widehat C \varepsilon,
  \quad \text{$\bP$-almost surely},
\end{align*}
which concludes the proof. \hfill \qed

\section{Proof of Proposition~\ref{lem:main}}
\label{sec:proof-prop-main}

The following result is an essential tool to prove
Proposition~\ref{lem:main} and will be proven in
Section~\ref{subsec:mixing_quenched}.

\begin{lem}
  \label{lem:HIG_lemma}
  Let $0<\theta<1/2$ and $b >0$. Define the set
  \begin{align}
    \label{eq:87}
    D(n)
    \coloneqq \bigcap_{\substack{x,y\in\bZ^d\, :\\
    \norm{x},\norm{y}\le n^b,\\
    \norm{x-y}\le n^{\theta}}} \hspace{-1.5em}
    \left\{
    \norm{P^{(x,0)}_\omega(X_n\in \cdot)- P^{(y,0)}_\omega(X_n \in \cdot)}_{\mathrm{TV}}
    \le \mathrm{e}^{-c \frac{\log n}{\log \log n}} \right\}.
  \end{align}
  Then there are constants $C,c>0$ so that  $\bP(D(n))\ge 1-Cn^{-c\log
    n}$.
\end{lem}

Note that the restriction $\norm{x},\norm{y}\le n^b$ in the definition
of $D(n)$ in \eqref{eq:87} is necessary because with probability $1$
we have an environment where there exist (somewhere far out in space)
two neighbouring points $x, y\in\bZ^d$ so that the sites $(x,0)$ and
$(y,0)$ are both connected to infinity but the respective clusters do
not intersect for the first $n$ time steps.

\begin{rem}
  \label{rem:repl77}
  The above lemma is the analogue of Lemma~7.7 from
  \cite{BergerCohenRosenthal2016} in our setting.
  Note that the bound stated in Lemma~7.7 from
  \cite{BergerCohenRosenthal2016} is too optimistic to hold in
  general.
  However, its assertion can be weakened and one obtains a bound which
  is still strong enough to prove Lemma~7.5 in
  \cite{BergerCohenRosenthal2016} by going a similar route as in the
  proof of Lemma~\ref{lem:HIG_lemma} here.
\end{rem}

\begin{proof}[Proof of Proposition~\ref{lem:main}, \eqref{eq:ann_to_ann-quenched}]
  For this part we make use of the fact that, due to the annealed
  derivative estimates from
  Lemma~\ref{lem:annealed_derivative_estimates} for $\abs{x-y}\leq k$,
  $\abs{\bP^{(0,0)}(X_n=x)-\bP^{(0,0)}(X_{n-k}=y)} \leq
  Ck/(n-k)^{(d+1)/2}\approx n^{-(d+1)/2+\varepsilon}$, since
  $k=\lceil n^\varepsilon \rceil \ll n$. Furthermore we use the fact
  that by definition as a density of the invariant measure of the
  environment with respect to the point of view of the particle, the
  prefactor can be ``transported'' along the quenched transition
  probabilities; see Proposition~\ref{prop:2}. Finally we use the
  concentration property of Corollary~\ref{cor:con_prop_Q}; see
  equation \eqref{eq:29a}.

  We have to show
  \begin{multline}
    \label{eq:84}
    \lim_{n\to\infty}\sum_{x\in\bZ^d}
    \Abs{\frac{1}{Z_{\omega,n}}\bP^{(0,0)}(X_n=x)\varphi(\sigma_{(x,n)}\omega)\\
      -\frac{1}{Z_{\omega,n-k}}\sum_{y\in\bZ^{d}}\bP^{(0,0)}(X_{n-k}=y)
      \varphi(\sigma_{(y,n-k)}\omega)P_{\sigma_{(y,n-k)}\omega}^{(0,0)}(X_k=x-y)}=0.
  \end{multline}
  Note that the by the triangle inequality the sum on the left hand
  side is bounded from above by
  \pagebreak[1]
  \begin{align*}
    \sum_{x\in\bZ^d}
    \Abs{\frac{1}{Z_{\omega,n}} & - \frac{1}{Z_{\omega,n-k}}}
     \bP^{(0,0)}(X_n=x)\varphi(\sigma_{(x,n)}\omega)\\
    &   + \frac{1}{Z_{\omega,n-k}}\sum_{x\in\bZ^{d}}\Abs{\bP^{(0,0)}(X_{n}=x)
      \varphi(\sigma_{(x,n)}\omega) \\
    & \qquad \qquad \qquad \qquad  - \sum_{y\in\bZ^{d}}\bP^{(0,0)}(X_{n-k}=y)
    \varphi(\sigma_{(y,n-k)}\omega)P_{\sigma_{(y,n-k)}\omega}^{(0,0)}(X_k=x-y)}.
  \end{align*}
  By definition of $Z_{\omega,n}$, recall from
  Definition~\ref{defn:auxiliary_measures}, the first sum in the above
  display equals to
  \begin{align*}
    \Abs{\frac{1}{Z_{\omega,n}} -
    \frac{1}{Z_{\omega,n-k}}}Z_{\omega,n},
  \end{align*}
  which by Proposition~\ref{lem:limit_Z_omega} almost surely goes to
  $0$ as $n$ and $n-k$ both tend to $\infty$. Thus, taking also into
  account the trivial deterministic bound on the speed of the random
  walk, for \eqref{eq:84} it suffices to show
  \begin{multline}
    \label{eq:43}
    \lim_{n\to\infty}\sum_{x\in\bZ^d\cap[-n,n]^d}
    \Abs{\bP^{(0,0)}(X_n=x)\varphi(\sigma_{(x,n)}\omega) \\
      -\sum_{y\in\bZ^{d}\cap[-n,n]^d}\bP^{(0,0)}(X_{n-k}=y)
      \varphi(\sigma_{(y,n-k)}\omega)P_{\sigma_{(y,n-k)}\omega}^{(0,0)}(X_k=x-y)}=0.
  \end{multline}
  Denoting by $B_n=\{x\in\bZ^d:\norm{x} \le \sqrt{n}\log^3n \}$ and
  using the triangle inequality an upper bound of the sum in
  \eqref{eq:43} is given by
  \begin{align}
    \label{eq:L1_1st}
      & \sum_{x\in B_n} \Abs{\sum_{y\in \bZ^d\cap[-n,n]^d}
        \bigl[ \bP^{(0,0)}(X_n=x)-\bP^{(0,0)}(X_{n-k}=y) \bigr]\\
    \notag
      & \hspace{4cm} \times\varphi(\sigma_{(y,n-k)}\omega)
        P^{(0,0)}_{\sigma_{(y,n-k)}\omega}(X_k=x-y)}\\
    \label{eq:L1_2nd}
      &\hspace{0.5cm}+\sum_{x\in B_n} \bP^{(0,0)}(X_n=x)\\
    \notag
      &\hspace{2cm}\times\Abs{\varphi(\sigma_{(x,n)}\omega)
        - \sum_{y\in \bZ^d\cap[-n,n]^d}\varphi(\sigma_{(y,n-k)}\omega)
        P^{(0,0)}_{\sigma_{(y,n-k)}\omega}(X_k=x-y)} \\
    \label{eq:L1_3rd}
      &\hspace{0.5cm} + \sum_{x\in\bZ^d\cap[-n,n]^d\setminus B_n}
        \Abs{\bP^{(0,0)}(X_n=x)\varphi(\sigma_{(x,n)}\omega)\\
    \notag
      &\hspace{4cm}
        - \sum_{y\in\bZ^{d}\cap[-n,n]^d}\bP^{(0,0)}(X_{n-k}=y)
        \varphi(\sigma_{(y,n-k)}\omega)P_{\sigma_{(y,n-k)}\omega}^{(0,0)}(X_k=x-y)}.
  \end{align}
  By the annealed derivative estimates \eqref{eq:L1_1st} is bounded
  from above by
  \begin{align*}
    \sum_{x\in B_n}
    & \Abs{\sum_{\substack{y\in\bZ^d\\ \norm{x-y}\le k}}
    \bigl[ \bP^{(0,0)}(X_n=x)-\bP^{(0,0)}(X_{n-k}=y) \bigr]\\
    & \hspace{1.5cm}\times  \varphi(\sigma_{(y,n-k)}\omega)
      P^{(0,0)}_{\sigma_{(y,n-k)}\omega}(X_k=x-y)}\\
    & \le \frac{2Ck}{(n-k)^{(d+1)/2}} \sum_{x\in B_n} \sum_{\substack{y\in\bZ^d\\ \norm{x-y}\le k}}
    \varphi(\sigma_{(y,n-k)}\omega)P^{(0,0)}_{\sigma_{(y,n-k)}\omega}(X_k=x-y)\\
    &\le \frac{2Ck (\sqrt{n}\log^3n +k)^d }{(n-k)^{(d+1)/2}}
      \frac{1}{(\sqrt{n}\log^3n+k)^d} \sum_{\substack{y\in\bZ^d\\
    \dist(y,B_n)\le k}} \varphi(\sigma_{(y,n-k)}\omega).
  \end{align*}
  Now using Corollary~\ref{cor:con_prop_Q} and the fact that
  $k=\lceil n^\varepsilon \rceil <n^{1/4}$ for $\bP$-almost every
  $\omega$ the last term tends to zero as $n$ tend to infinity.

  \medskip

  Next we deal with \eqref{eq:L1_2nd}. Recall that by
  Proposition~\ref{prop:2} we have
  \begin{align*}
    \varphi(\sigma_{(x,n)}\omega)=
    \sum_{y\in\bZ^d\cap[-n,n]^d}\varphi(\sigma_{(x,n-k)}\omega)P^{(y,n-k)}_\omega(X_k=x)
  \end{align*}
  for every $x\in\bZ^d$ such that
  $x+[-k,k]^d\cap\bZ^d \subset [-n,n]^d\cap\bZ^d$. This holds for
  every $x\in B_n$ and therefore the expression \eqref{eq:L1_2nd}
  equals $0$.

  Finally, for \eqref{eq:L1_3rd}, using Lemma~3.6 from
  \cite{SteibersPhD2017}, we have
  $\bP^{(0,0)}(X_n\notin B_n)\le Cn^{-c\log n}$. Recall that
  $k=\lceil n^\varepsilon\rceil$ and note that if
  $P^{(y,n-k)}_\omega(X_k=x)>0$ then $\norm{x-y}\le k$. Thus, for
  $x\in[-n,n]^d\cap\bZ^d\setminus B_n$ and large enough $n$
  \begin{align*}
    \norm{y} \ge \norm{x} - \norm{x-y} \ge \sqrt{n}\log^3n -
    k \ge \frac{1}{2}\sqrt{n}\log^3n.
  \end{align*}
  This implies, again due to Lemma~3.6 from \cite{SteibersPhD2017}
  that $\bP^{(0,0)}(X_{n-k}=y)\le Cn^{-c\log n}$. Therefore, the
  expression \eqref{eq:L1_3rd} is bounded from above by
  \begin{align*}
    \sum_{x\in\bZ^d\cap[-n,n]^d\setminus B_n}
    & \bP^{(0,0)}(X_n=x)\varphi(\sigma_{(x,n)}\omega)\\
    & \; \; + \sum_{x\in\bZ^d\cap[-n,n]^d\setminus B_n}\sum_{y\in\bZ^{d}\cap[-n,n]^d}
      \bP^{(0,0)}(X_{n-k}=y)\varphi(\sigma_{(y,n-k)}\omega)P^{(y,n-k)}_\omega(X_k=x)\\
    & \le  Cn^{-c\log n}\sum_{x\in\bZ^d\cap[-n,n]^d\setminus B_n}\varphi(\sigma_{(x,n)}\omega)\\
    & \; \; +Cn^{-c\log n}\sum_{x\in\bZ^d\cap[-n,n]^d\setminus B_n}\sum_{y\in\bZ^{d}\cap[-n,n]^d}
      \varphi(\sigma_{(y,n-k)}\omega)P^{(y,n-k)}_\omega(X_k=x)\\
    & \le Cn^{-c\log n}\sum_{x\in\bZ^d\cap[-n,n]^d} \varphi(\sigma_{(x,n)}\omega)
      + Cn^{-c\log n}\sum_{y\in\bZ^d\cap[-n,n]^d} \varphi(\sigma_{(y,n-k)}\omega).
  \end{align*}
  By Corollary~\ref{cor:con_prop_Q} we have
  \begin{align*}
    & \bP \Bigl(\sum_{x\in\bZ^d\cap[-n,n]^d}
      \varphi(\sigma_{(x,n)}\omega)\le 2n^d\Bigr)>1-n^{-c\log n},\\
    \intertext{as well as}
    & \bP\Bigl(\sum_{y\in\bZ^d\cap[-n,n]^d}
      \varphi(\sigma_{(y,n-k)}\omega)\le 2n^d\Bigr)>1-Cn^{-c\log n}.
  \end{align*}
  Thus, the probability of the event that \eqref{eq:L1_3rd} is bounded
  above by $4C n^{-c\log n}n^d$ converges to $1$ super-algebraically
  fast. Hence
  the expression \eqref{eq:L1_3rd} converges to $0$ $\bP$-almost
  surely.
\end{proof}

\begin{proof}[Proof of Proposition~\ref{lem:main}, \eqref{eq:ann-quenched_to_box-quenched}]
  First note that, it is enough to show that
  \begin{align*}
  \norm{\nu^{\annpre}_\omega - \nu^{\boxpre}_\omega}_{1,n-k}
    \xrightarrow{n\to\infty} 0,
  \end{align*}
  since the last $k$ steps are according to the quenched law for both
  hybrid measures. Then, as the measure
  $\nu^{\mathrm{box-que}\times\pre}$ suggests, we make use of the
  comparison between the quenched and the annealed laws on the level
  of boxes we derived from Lemma~\ref{claim:1}. We also use the
  concentration properties of $\varphi$ from
  Corollary~\ref{cor:con_prop_Q}.

  Let $k \in \{0,\dots,n\}$ be fixed. Note that we have
  \begin{multline*}
    \norm{(\nu^{\annpre}\ast \nu^{\que})_{\omega,k}
      -(\nu^{\boxpre}\ast\nu^{\que})_{\omega,k}}_{1,n}
    \le \norm{\nu^{\annpre}_\omega - \nu^{\boxpre}_\omega}_{1,n-k}\\
    = \sum_{x\in\bZ^d} \varphi(\sigma_{(x,n-k)}\omega)
    \Abs{\frac{\bP^{(0,0)}(X_{n-k}=x)}{Z_{\omega,n-k}} -
      \frac{P^{(0,0)}_\omega(X_{n-k}\in\Delta_x)}{\sum_{y\in\Delta_x}
        \varphi(\sigma_{(y,n-k)}\omega)}}.
  \end{multline*}
  By Proposition~\ref{lem:limit_Z_omega} it is enough to show that
  $\bP$-almost surely
  \begin{align}
    \label{eq:44}
    \lim_{n\to\infty} \sum_{x\in\bZ^d} \varphi(\sigma_{(x,n-k)}\omega)
    \Abs{\bP^{(0,0)}(X_{n-k}=x)
    - \frac{P^{(0,0)}_\omega(X_{n-k}\in\Delta_x)}{\sum_{y\in\Delta_x}
    \varphi(\sigma_{(y,n-k)}\omega)}} =0.
  \end{align}
  Let $A_n=\{x\in\bZ^d :\norm{x}\le C_\varepsilon\sqrt{n} \}$, with
  $C_\varepsilon$ chosen so that
  $\bP^{(0,0)}(\norm{X_{n-k}}>\frac{C_\varepsilon}{2}\sqrt{n-k})<
  \varepsilon$ for $n$ large enough. Using the triangle inequality the
  sum in \eqref{eq:44} is bounded by
  \begin{align}
    \label{eq:L2_1st}
    \sum_{x\in\bZ^d\cap[-n,n]^d\setminus A_n}
    & \varphi(\sigma_{(x,n-k)}\omega)
      \Abs{\bP^{(0,0)}(X_{n-k}=x)
      - \frac{P^{(0,0)}_\omega(X_{n-k}\in\Delta_x)}{\sum_{y\in\Delta_x}
      \varphi(\sigma_{(y,n-k)}\omega)}} \\
    \label{eq:L2 2nd}
    & + \sum_{x\in A_n} \varphi(\sigma_{(x,n-k)}\omega)
      \Abs{\bP^{(0,0)}(X_{n-k}=x)-\frac{\bP^{(0,0)}(X_{n-k}\in\Delta_x)}{\abs{\Delta_x}}}\\
    \label{eq:L2_3rd}
    & + \sum_{x\in A_n} \varphi(\sigma_{(x,n-k)}\omega)
      \Abs{\frac{1}{\abs{\Delta_x}}\bP^{(0,0)}(X_{n-k}\in\Delta_x)
      - \frac{\bP^{(0,0)}(X_{n-k}\in\Delta_x)}{\sum_{y\in\Delta_x}\varphi(\sigma_{(y,n-k)}\omega)} }\\
    \label{eq:L2_4th}
    & + \sum_{x\in A_n} \varphi(\sigma_{(x,n-k)}\omega)
      \Abs{\frac{\bP^{(0,0)}(X_{n-k}\in\Delta_x)}{\sum_{y\in\Delta_x}\varphi(\sigma_{(y,n-k)}\omega)}
      - \frac{P^{(0,0)}_\omega(X_{n-k}\in\Delta_x)}{\sum_{y\in\Delta_x}
      \varphi(\sigma_{(y,n-k)}\omega)}}.
  \end{align}

  Now we deal with the four terms separately. Expression
  \eqref{eq:L2_1st} is bounded from above by
  \begin{align*}
    \sum_{x\in\bZ^d\cap[-n,n]^d\setminus A_n} \bP^{(0,0)}(X_{n-k}=x)\varphi(\sigma_{(x,n-k)}\omega)
    +P^{(0,0)}_\omega(\norm{X_{n-k}}>C_\varepsilon\sqrt{n}).
  \end{align*}
  The term
  $\sum_{x\in\bZ^d\cap[-n,n]^d\setminus A_n}\bP^{(0,0)}(X_{n-k}=x)
  \varphi(\sigma_{(x,n-k)}\omega)$ goes to zero as $n$ goes to
  infinity by the same arguments used to bound
  \eqref{eq:normalizing_constant_1} in the proof of
  Proposition~\ref{lem:limit_Z_omega}. For the second term we can
  argue as in the proof of Claim~2.15 from
  \cite{BergerCohenRosenthal2016}, to obtain that for a set of
  environments, with $\bP$ probability $>1-\sqrt{\varepsilon}$, for
  large enough $n$
  \begin{align*}
    P^{(0,0)}_\omega(\norm{X_{n-k}}>C_\varepsilon\sqrt{n}) \le
    P^{(0,0)}_\omega\Bigl( \norm{X_{n}} >
    \frac{C_\varepsilon}{2}\sqrt{n}\Bigr) \le \sqrt{\varepsilon}.
  \end{align*}
  Since $\varepsilon>0$ was arbitrary, this proves that
  \eqref{eq:L2_1st} goes to zero as $n$ goes to infinity.

  \medskip

  Next we turn to \eqref{eq:L2 2nd}. The annealed derivative estimates
  yield that it is bounded from above by
  \begin{align*}
    \sum_{x\in A_n} \varphi(\sigma_{(x,n-k)}\omega)
    & \frac{1}{\abs{\Delta_x}}
      \sum_{y\in\Delta_x} \abs{\bP^{(0,0)}(X_{n-k}=x)-\bP^{(0,0)}(X_{n-k}=y)}\\
    & \le C\sum_{x\in A_n}\varphi(\sigma_{(x,n-k)}\omega)
      \frac{1}{\abs{\Delta_x}}\sum_{y\in\Delta_x} \frac{1}{(n-k)^{(d+1)/2}}\norm{x-y}\\
    & \le Cdn^\delta \frac{1}{(n-k)^{(d+1)/2}}\sum_{x\in A_n} \varphi(\sigma_{(x,n-k)}\omega)\\
    & =
      \frac{Cn^{\delta+d/2}}{(n-k)^{(d+1)/2}}\Bigl(\frac{1}{n^{d/2}}
      \sum_{x\in A_n} \varphi(\sigma_{(x,n-k)}\omega)
      \Bigr)\xrightarrow{n\to\infty} 0, \quad \bP\text{-a.s.},
  \end{align*}
  where for the limit we use Proposition~\ref{lem:box_average_estimate}, the
  fact that $k=\lceil n^\varepsilon \rceil$ and
  $\delta<\varepsilon<\frac{1}{4}$.

  \medskip

  Next we deal with \eqref{eq:L2_3rd}. Writing
  $\widehat{\Pi}_n=\{\Delta\in\Pi:\Delta\cap A_n\neq \emptyset\}$,
  using annealed derivative estimates and
  Corollary~\ref{cor:con_prop_Q} we see that
  \eqref{eq:L2_3rd} is bound by
  \begin{align*}
    \sum_{x\in A_n}
    & \varphi(\sigma_{(x,n-k)}\omega)
      \frac{1}{\abs{\Delta_x}} \bP^{(0,0)}(X_{n-k}\in\Delta_x)
      \Abs{1 - \frac{1}{\frac{1}{\abs{\Delta_x}}
      \sum_{y\in\Delta_x}\varphi(\sigma_{(y,n-k)}\omega)}}\\
    & \le \frac{C}{(n-k)^{d/2}}\sum_{x\in A_n}\varphi(\sigma_{(x,n-k)}\omega)
      \Abs{1 - \frac{1}{\frac{1}{\abs{\Delta_x}}
      \sum_{y\in\Delta_x}\varphi(\sigma_{(y,n-k)}\omega)}}\\
    & \le C\Bigl( \frac{n-k}{n} \Bigr)^{-d/2}\frac{1}{n^{d/2}}
      \sum_{\Delta \in \widehat{\Pi}_n}\sum_{x\in\Delta} \varphi(\sigma_{(x,n-k)}\omega)
      \Abs{1 - \frac{1}{\frac{1}{\abs{\Delta_x}}
      \sum_{y\in\Delta_x}\varphi(\sigma_{(y,n-k)}\omega)}}\\
    & = C\Bigl(1-\frac{k}{n}\Bigr)^{-d/2}
      \frac{1}{n^{d/2}} \sum_{\Delta\in\widehat{\Pi}_n}
      \sum_{x\in\Delta} \frac{\varphi(\sigma_{(x,n-k)}\omega) \abs{\Delta_x}}{\sum_{y\in\Delta_x}
      \varphi(\sigma_{(y,n-k)}\omega)}
      \Abs{\frac{1}{\abs{\Delta_x}}\sum_{y\in\Delta_x} \varphi(\sigma_{(x,n-k)}\omega) -1}\\
    & = C\Bigl(1-\frac{k}{n}\Bigr)^{-d/2} \frac{1}{n^{(d/2)(1-2\delta)}}
      \sum_{\Delta\in\widehat{\Pi}_n} \Abs{\frac{1}{\abs{\Delta}}
      \sum_{x\in\Delta}\varphi(\sigma_{(x,n-k)}\omega)-1}.
  \end{align*}
  Using the same argument that was used for
  \eqref{eq:normalizing_constant_3}, we get that by the Borel-Cantelli
  lemma the last term goes to zero $\bP$-a.s.

  Finally, we estimate \eqref{eq:L2_4th}. It is bounded from above by
  \begin{multline*}
    \sum_{x\in A_n}
    \frac{\varphi(\sigma_{(x,n-k)}\omega)}{\sum_{y\in\Delta_x}
      \varphi(\sigma_{(y,n-k)}\omega)}\abs{\bP^{(0,0)}(X_{n-k}\in\Delta_x)
      -P^{(0,0)}_\omega(X_{n-k}\in\Delta_x)}\\
    =\sum_{\Delta\in\widehat{\Pi}_n}
    \abs{\bP^{(0,0)}(X_{n-k}\in\Delta)-P^{(0,0)}_\omega(X_{n-k}\in\Delta)}.
  \end{multline*}
  For the last term we can use Theorem~\ref{thm:Steibers_Thm3.24}
  which implies that it is bounded by $Cn^{-\frac{1}{3}\delta}$ for
  $\bP$-almost every $\omega$ and large enough $n$. Therefore $\bP$
  almost surely it converges to zero as $n$ tends to infinity.
\end{proof}

\begin{proof}[Proof of Proposition~\ref{lem:main}, \eqref{eq:box-quenched_to_quenched}]
  Note that the first measure chooses, at time $n-k$, a box according
  to the quenched law and a point in that box weighted by the
  prefactor, whereas the second measure chooses a box and a point in
  that box according to the quenched law at time $n-k$. These points
  are then the starting points for the quenched random walks for the
  remaining $k$ steps. We use the fact that, given enough time (much
  more than the square of the starting distance), the total variation
  distance for two quenched random walks starting from any pair of
  sites in a box with side length $\lceil n^\ell \rceil$ is, given
  enough time, i.e. much more than the square of the side length of
  the box, is small with high probability, see
  Lemma~\ref{lem:HIG_lemma}.\\
  The proof follows along the same lines as in
  \cite{BergerCohenRosenthal2016}. We will highlight the point in the
  proof where we deviate. We have
  \begin{align}
    \nonumber
    & \norm{(\nu^{\boxpre}\ast\nu^{\que})_{\omega,k} - (\nu^{\que}\ast \nu^{\que})_{\omega,k}}_{1,n}\\
    \nonumber
    & = \sum_{x\in\bZ^d}\big\lvert (\nu^{\boxpre}\ast\nu^{\que})_{\omega,k}(x,n)
      -(\nu^{\que}\ast \nu^{\que})_{\omega,k}(x,n)\big\rvert\\
    \nonumber
    & = \sum_{x\in\bZ^d}\Abs{\sum_{y\in\bZ^d}P^{(0,0)}_\omega(X_{n-k}\in \Delta_y)
      \frac{\varphi(\sigma_{(y,n-k)} \omega)}{\sum_{z\in\Delta_y}\varphi(\sigma_{(z,n-k)}\omega)}
      P^{(0,0)}_{\sigma_{(y,n-k)}\omega}(X_k=x-y)\\
    \nonumber
    & \hspace{3cm} - \sum_{y\in\bZ^d}P^{(0,0)}_\omega(X_{n-k}=y)
      P^{(0,0)}_{\sigma_{(y,n-k)}\omega}(X_k=x-y)}\\
    \nonumber
    & = \sum_{x\in\bZ^d}\Abs{\sum_{\Delta\in\Pi} \sum_{y\in\Delta}
      P^{(y,n-k)}_\omega(X_k=x)P^{(0,0)}_\omega(X_{n-k}\in\Delta)\\
    \nonumber
    & \hspace{3cm} \cdot \Bigl(\frac{\varphi(\sigma_{(y,n-k)} \omega)}{\sum_{z\in\Delta}
      \varphi(\sigma_{(z,n-k)}\omega)} -
      P^{(0,0)}_\omega(X_{n-k}=y\,\vert\,X_{n-k}\in\Delta)\Bigr)}\\
    \nonumber
    & \le \sum_{x\in\bZ^d} \sum_{\Delta\in\Pi}\Abs{\sum_{y\in\Delta}
      P^{(y,n-k)}_\omega(X_k=x)P^{(0,0)}_\omega(X_{n-k}\in\Delta)\\
    \label{eq:box-quenched_to_quenched_1}
    &\hspace{3cm} \cdot \Bigl(\frac{\varphi(\sigma_{(y,n-k)} \omega)}{\sum_{z\in\Delta}
      \varphi(\sigma_{(z,n-k)}\omega)} - P^{(0,0)}_\omega(X_{n-k}=y\,\vert\,X_{n-k}\in\Delta)\Bigr)}.
  \end{align}
  Since for every $\Delta\in\Pi$ and $x\in\bZ^d$ we have
  \begin{align*}
    \sum_{y\in\Delta}\frac{1}{\abs{\Delta}}\sum_{v\in\Delta}
    P^{(v,n-k)}_\omega(X_k=x) \Bigl[
    \frac{\varphi(\sigma_{(y,n-k)} \omega)}{\sum_{z\in\Delta} \varphi(\sigma_{(z,n-k)}\omega)} -
    P^{(0,0)}_\omega(X_{n-k}=y\,\vert\,X_{n-k}\in\Delta) \Bigr]=0
  \end{align*}
  it follows that \eqref{eq:box-quenched_to_quenched_1} equals
  \begin{align}
    \nonumber
    & \sum_{x\in\bZ^d} \sum_{\Delta\in\Pi}P^{(0,0)}_\omega(X_{n-k}\in\Delta)
      \Abs{\sum_{y\in\Delta}\Bigl[P^{(y,n-k)}_\omega(X_k=x)
      - \Bigl( \frac{1}{\abs{\Delta}}\sum_{w\in\Delta}P^{(w,n-k)}_\omega(X_k=x) \Bigr)\Bigr]\\
    \nonumber
    & \hspace{3cm} \Bigl(\frac{\varphi(\sigma_{(y,n-k)} \omega)}{\sum_{z\in\Delta}
      \varphi(\sigma_{(z,n-k)}\omega)} - P^{(0,0)}_\omega(X_{n-k}=y\,\vert\,X_{n-k}\in\Delta)\Bigr)}\\
    \nonumber
    & = \sum_{x\in\bZ^d} \sum_{\Delta\in\Pi}P^{(0,0)}_\omega(X_{n-k}\in\Delta)
      \Abs{\frac{1}{\abs{\Delta}}\sum_{y\in\Delta}\sum_{w\in\Delta}\Bigl[P^{(y,n-k)}_\omega(X_k=x)
      - P^{(w,n-k)}_\omega(X_k=x) \Bigr]\\
    \nonumber
    & \hspace{3cm} \Bigl(\frac{\varphi(\sigma_{(y,n-k)}\omega)}{\sum_{z\in\Delta}
      \varphi(\sigma_{(z,n-k)}\omega)} -
      P^{(0,0)}_\omega(X_{n-k}=y\,\vert\,X_{n-k}\in\Delta)\Bigr)} \\
    \nonumber
    & \le\sum_{\Delta\in\Pi}\sum_{x\in\bZ^d} P^{(0,0)}_\omega(X_{n-k}\in\Delta)\sum_{y\in\Delta}
      \frac{1}{\abs{\Delta}}\sum_{w\in\Delta}\Abs{P^{(y,n-k)}_\omega(X_k=x) - P^{(w,n-k)}_\omega(X_k=x)}\\
    \label{eq:box-quenched_to_quenched_2}
    &\hspace{3cm} \Abs{\frac{\varphi(\sigma_{(y,n-k)} \omega)}{\sum_{z\in\Delta}
      \varphi(\sigma_{(z,n-k)}\omega)} - P^{(0,0)}_\omega(X_{n-k}=y\,\vert\,X_{n-k}\in\Delta)}
  \end{align}
  Until this point the steps are basically the same as in
  \cite{BergerCohenRosenthal2016}. Here we deviate from their proof.
  Note that $P^{(0,0)}_\omega(X_{n-k}\in\Delta)=0$ if
  $\Delta \cap [-n + k,n-k]^d =\emptyset$. For
  $\Delta \cap [-n + k,n-k]^d \ne \emptyset$ we have $y, w \in \Delta$
  implies that $\norm{y}, \norm{w} \le n = k^{1/\varepsilon}$ and
  $\norm{y-w} \le n^\delta =k^{\delta/\varepsilon}$.

  Using Lemma~\ref{lem:HIG_lemma} we see that
  \eqref{eq:box-quenched_to_quenched_2} is bounded from above by
  \begin{align*}
    & \sum_{\Delta\in\Pi}P^{(0,0)}_\omega(X_{n-k}\in\Delta)\sum_{y\in\Delta}
      \Abs{\frac{\varphi(\sigma_{(y,n-k)}\omega)}{\sum_{z\in\Delta}
      \varphi(\sigma_{(z,n-k)}\omega)}-P^{(0,0)}_\omega(X_{n-k}=y\,\vert\,X_{n-k}\in\Delta) }\\
    & \hspace{3cm} \frac{1}{\abs{\Delta}}\sum_{w\in\Delta}\sum_{x\in\bZ^d}
      \Abs{P^{(y,n-k)}_\omega(X_k=x)-P^{(w,n-k)}_\omega(X_k=x)} \\
    & \le \mathrm{e}^{-c\frac{\log k}{\log\log k}} \sum_{\Delta\in\Pi}P^{(0,0)}_\omega(X_{n-k}\in\Delta)
      \sum_{y\in\Delta}\Abs{\frac{\varphi(\sigma_{(y,n-k)}\omega)}{\sum_{z\in\Delta}
      \varphi(\sigma_{(z,n-k)}\omega)}-P^{(0,0)}_\omega(X_{n-k}=y\,\vert\,X_{n-k}\in\Delta) }\\
    & \le 2 \mathrm{e}^{-c\frac{\log k}{\log\log k}} \sum_{\Delta\in\Pi}P^{(0,0)}_\omega(X_{n-k}\in\Delta)
      = 2 \mathrm{e}^{-c\frac{\log k}{\log\log k}} \le C\mathrm{e}^{-\tilde{c}\frac{\log n}{\log\log n}}
  \end{align*}
  since $k=\lceil n^{\varepsilon}\rceil$. The right hand side goes to
  0 for $n\to \infty$.
\end{proof}

\section{Proof of Proposition~\ref{prop_main_thm_5.1}}
\label{sec:proof-prop-refpr-1}

The starting point is a result from \cite{SteibersPhD2017}. Define
\begin{align}
  \label{eq:69}
  \cP(N) \coloneqq
  \Bigl(\Bigl[-\frac{1}{24}\sqrt{N}\log^3 N,\frac{1}{24}\sqrt{N}\log^3 N \Bigr]^d
  \times \Bigl[0,\frac{1}{3}N\Bigr]\Bigr) \cap (\bZ^d \times \bZ).
\end{align}
For $\theta \in (0,1)$ and $(x,m) \in \cP(N)$ let $G'((x,m),N)$ denote
the event that for every box $\Delta \subset \bZ^d$ of side length
$N^{\theta/2}$ we have
\begin{align}
  \label{eq:49}
  \big\lvert P_{\omega}^{(x,m)} (X_{m+N} \in \Delta ) - \bP^{(x,m)}(X_{m+N}
  \in \Delta) \big\rvert \le CN^{-d(1-\theta)/2- \frac{1}{6}\theta}.
\end{align}
Furthermore set
\begin{align}
  \label{eq:48}
  G'(N)\coloneqq \bigcap_{(x,m) \in \cP(N)} \bigl(G'((x,m),N\bigr) \cup \{\xi_m(x)=0\}).
\end{align}
\begin{thm}[Theorem~3.24 in \cite{SteibersPhD2017}]
  \label{thm:Steibers_Thm3.24}
  Let $d\ge 3$. There exist positive constants $c$ and $C$, such that
  for all $(x,m) \in \cP(N)$ we have
  \begin{align}
    \bP^{(x,m)}\bigl(G'((x,m),N) \bigr) \ge 1 - C N^{-c\log N}
  \end{align}
  and
  \begin{align}
    \label{eq: lower bound for G4(N)}
    \bP\bigl(G'(N) \bigr) \ge 1- CN^{-c\log N}.
  \end{align}
\end{thm}

The following notion of \emph{good sites} and \emph{good boxes} will
be needed in the proof of Proposition~\ref{prop_main_thm_5.1}.
On such boxes the annealed and quenched laws are ``close'' to each
other. Recall the process $\xi=(\xi_n)_{n\in \bZ}$ from \eqref{eq:45}
and the definition of $n_k$ from the beginning of
Section~\ref{sect:box_level_comparisons}. Recall also that $\Pi_k$ is
a partition of $\bZ^d$ into the boxes of side length
$\lfloor n_k^\theta \rfloor$.
\begin{defn}
  \label{def:good boxes}
  For a given realisation $\omega \in \Omega$, we say that
  $(x,m) \in \bZ^d\times\bZ$ is $(k-1,\theta,\varepsilon)$-\emph{good}
  if either $\xi_m(x;\omega)=0$ or $\xi_m(x;\omega)=1$ and the
  following two conditions are satisfied
  \begin{align}
    \label{eq:46}
    \sup_{\Delta' \in \Pi_{k}} \big | P_{\omega}^{(x,m)} (X_{m+n_k} \in \Delta') -
    \bP^{(x,m)}(X_{m+n_k} \in \Delta')\big|
    & \le n_k^{\theta d -\frac{d}{2} - \varepsilon},\\
    \label{eq:47}
    P_{\omega}^{(x,m)} \Bigl(\max_{s \le n_k} \norm{X_{m+s}-x} >\sqrt{n_k} \log^3 n_k\Bigr)
    & \le Cn_k^{-c\log n_k}.
  \end{align}
  Otherwise the site is said to be
  $(k-1,\theta,\varepsilon)$-\emph{bad}. We say that for
  $\Delta\in \Pi_{k-1}$ and $m \in \bZ$ the box $\Delta \times \{m\}$
  is $(k-1,\theta,\varepsilon)$-\emph{good} if each
  $(x,m) \in \Delta \times \{m\}$ is
  $(k-1,\theta,\varepsilon)$-\emph{good}. Otherwise we say that
  $\Delta \times \{m\}$ is $(k-1,\theta,\varepsilon)$-\emph{bad}.
\end{defn}

The following lemma is a direct consequence of
Theorem~\ref{thm:Steibers_Thm3.24}.
\begin{lem}
  \label{lem:lowerBoundGoodPoint}
  For all $\Delta \in \Pi_{k-1}$ there are positive constants $C$
  and $c$ so that
  \begin{align}
    \label{eq:18}
    \bP\left( \Delta \text{ is } (k-1,\theta,\varepsilon)\text{-good} \right)
    \ge 1 - Cn_k^{-c\log n_k}.
  \end{align}
\end{lem}

The assertion of Proposition~\ref{prop_main_thm_5.1} is the analogue
of the inequality (5.1) in \cite{BergerCohenRosenthal2016}. The
strategy of the proof there is as follows. First, using the triangle
inequality and the Markov property an upper bound of $\lambda_k$ is
obtained which is given by a sum of four terms (5.2) -- (5.5) in
\cite{BergerCohenRosenthal2016}. Second, for each of these four terms
an upper bound is shown. Three of these upper bounds, the ones for
(5.2), (5.4) and (5.5), are not difficult and can be proven in the
same way as in \cite{BergerCohenRosenthal2016}. For (5.3) Berger et.\
al use a notion of ``good'' boxes and the fact that they are
independent at a large but finite distance. The definition of those
good boxes translates to our Definition~\ref{def:good boxes}, where it
is clear that the dependence on $\xi$ prevents us from directly using
any argument hinging on independence at a finite distance. We
circumvent this problem by defining a new type of boxes for which we
are able to work with independence, see the ideas below
Proposition~\ref{prop_key}. Using those boxes as an approximation for
the good boxes we prove a lower bound on the probability of hitting a
good box in Proposition~\ref{prop_key}.

\smallskip

\noindent
\begin{proof}[Proof of the analogue of an upper bound of (5.2) in
  \cite{BergerCohenRosenthal2016}] Consider
  \begin{multline}
    \label{eq:firstInequality}
    \sum_{\Delta \in \Pi_{k}} \sum_{\Delta' \in \Pi_{k-1}}\Big\vert
    \sum_{u \in \Delta'}
    P^{(u,N_{k-1})}_{\omega}(X_{N_k} \in \Delta)\\
    \times \bigl[P^{(0,0)}_\omega(X_{N_{k-1}}=u ) -
    \bP^{(0,0)}(X_{N_{k-1}} \in \Delta')P^{(0,0)}_\omega(X_{N_{k-1}}=u
    \vert X_{N_{k-1}} \in \Delta')\bigr] \Big\vert.
  \end{multline}
  To get an upper bound for \eqref{eq:firstInequality} the arguments
  in \cite{BergerCohenRosenthal2016} do not require any specific
  properties of the model and apply to our model as well. The steps
  are as follows: by the triangle inequality followed by elementary
  computations \eqref{eq:firstInequality} is bounded from above by
  \begin{align}
    \label{eq:17}
    \begin{split}
      & \sum_{\Delta \in \Pi_{k}} \sum_{\Delta' \in \Pi_{k-1}} \sum_{u
        \in \Delta'}
      P^{(u,N_{k-1})}_{\omega}(X_{N_k} \in \Delta)\\
      & \qquad \times \big\vert P^{(0,0)}_\omega(X_{N_{k-1}}=u ) -
      \bP^{(0,0)}(X_{N_{k-1}} \in
      \Delta')P^{(0,0)}_\omega(X_{N_{k-1}}=u
      \vert X_{N_{k-1}} \in \Delta') \big\vert\\
      & = \sum_{\Delta' \in \Pi_{k-1}} \sum_{u \in \Delta'} \vert
      P^{(0,0)}_\omega(X_{N_{k-1}}=u ) 
      - \bP^{(0,0)}(X_{N_{k-1}} \in
      \Delta')P^{(0,0)}_\omega(X_{N_{k-1}}=u \vert X_{N_{k-1}} \in
      \Delta') \vert\\
      & = \sum_{\Delta' \in \Pi_{k-1}} \sum_{u \in \Delta'}
      P^{(0,0)}_\omega(X_{N_{k-1}}=u\vert X_{N_{k-1}} \in \Delta') 
      \vert P^{(0,0)}_\omega(X_{N_{k-1}} \in \Delta') -
      \bP^{(0,0)}(X_{N_{k-1}} \in \Delta') \vert\\
      & = \sum_{\Delta' \in \Pi_{k-1}} \vert
      P^{(0,0)}_\omega(X_{N_{k-1}} \in \Delta') -
      \bP^{(0,0)}(X_{N_{k-1}} \in \Delta') \vert = \lambda_{k-1}.
    \end{split}
  \end{align}
\end{proof}

\begin{proof}[Proof of the analogue of an upper bound of (5.3) in
  \cite{BergerCohenRosenthal2016}] Consider
  \begin{multline}
    \label{eq:secondInequality}
    \sum_{\Delta \in \Pi_{k}} \sum_{\Delta' \in \Pi_{k-1}}\Big\vert
    \sum_{u \in \Delta'} \bP^{(0,0)}(X_{N_{k-1}} \in
    \Delta')P^{(0,0)}_\omega(X_{N_{k-1}}
    =u\vert X_{N_{k-1}} \in \Delta')\\
    \times [P^{(u,N_{k-1})}_\omega(X_{N_k} \in \Delta) -
    \bP^{(u,N_{k-1})}(X_{N_k}\in \Delta)]\Big\vert.
  \end{multline}
  First, by the triangle inequality \eqref{eq:secondInequality} is
  bounded from above by
  \begin{multline}
    \label{eq:50}
    \sum_{\Delta' \in \Pi_{k-1}} \sum_{u \in \Delta'}
    \bP^{(0,0)}(X_{N_{k-1}} \in \Delta')P^{(0,0)}_\omega(X_{N_{k-1}}
    = u\vert X_{N_{k-1}} \in \Delta')\\
    \sum_{\Delta \in \Pi_{k}} \big\vert P^{(u,N_{k-1})}_\omega(X_{N_k}
    \in \Delta) - \bP^{(u,N_{k-1})}(X_{N_k}\in \Delta) \big\vert.
  \end{multline}
  Next we define $\Pi^1_{k-1}$ as the set of boxes
  $\Delta' \in \Pi_{k-1}$ with the property
  \begin{align*}
    \Delta' \cap \{ x\in \bZ^d : \norm{x} \le \sqrt{N_{k-1}} \log^3
    N_{k-1} \}\neq \emptyset.
  \end{align*}
  By Lemma~3.6 in \cite{SteibersPhD2017} it follows
  \begin{align}
    \label{eq:51}
    \sum_{\Delta' \notin \Pi^1_{k-1}}\bP^{(0,0)}(X_{N_{k-1}}\in \Delta')
    \le CN_{k-1}^{-c\log N_{k-1}}
  \end{align}
  and consequently \eqref{eq:secondInequality} is bounded from above
  by
  \begin{multline}
    \label{eq:52}
    C N_{k-1}^{-c\log N_{k-1}} + \sum_{\Delta' \in \Pi^1_{k-1}}
    \sum_{u \in \Delta'} \bP^{(0,0)}(X_{N_{k-1}}\in \Delta')
    P_{\omega}^{(0,0)}(X_{N_{k-1}}=u \vert X_{N_{k-1}}\in \Delta')\\
    \sum_{\Delta \in \Pi_{k}} \big\vert
    P_\omega^{(u,N_{k-1})}(X_{N_k}\in \Delta) -
    \bP^{(u,N_{k-1})}(X_{N_k}\in \Delta) \big\vert.
  \end{multline}

  Recall Definition~\ref{def:good boxes}.
  We will write ``good'' for $(k-1, \theta,\varepsilon)$-good to
  simplify the notation.
  By Lemma~\ref{lem:lowerBoundGoodPoint} we have
  $\bP\bigl(\Delta \text{ is } (k-1,\theta,\varepsilon)
  \text{-good}\bigr)\ge 1- Cn_k^{-c\log n_k}$. For $u \in \bZ^d$
  define by $\Pi^{(1,u)}_{k}$ the set of boxes $\Delta \in \Pi_{k}$
  satisfying (note that $\bE^{(u,0)}[X_{n_k}]=u$)
  \begin{align}
    \label{eq:defn_Pi(1,u)}
    \Delta \cap \bigl\{ x\in \bZ^d : \big\lVert x - u  \big\rVert
    \le \sqrt{n_k} \log^3 n_k \bigr\} \neq \emptyset.
  \end{align}
  If a box $\Delta' \in \Pi^1_{k-1}$ is
  $(k-1,\theta,\varepsilon)$-good, then for $u\in \Delta'$
  \begin{eqnarray}
    \begin{split}
      &\sum_{\Delta \in \Pi_{k}}
      \abs{P_\omega^{(u,N_{k-1})}(X_{N_k}\in
        \Delta) - \bP^{(u,N_{k-1})}(X_{N_k}\in \Delta)}\\
      & =\sum_{\Delta \in \Pi_{k}^{(1,u)}}
      \abs{P_\omega^{(u,N_{k-1})}(X_{N_k}\in \Delta) -
        \bP^{(u,N_{k-1})}(X_{N_k}\in \Delta)}\\
      &\hspace{0.5cm} + \sum_{\Delta \in
        \Pi_{k}\setminus\Pi^{(1,u)}_{k}}
      \abs{P_\omega^{(u,N_{k-1})}(X_{N_k} \in \Delta) - \bP^{(u,N_{k-1})}(X_{N_k}\in \Delta)}\\
      & \le \sum_{\Delta \in \Pi_{k}^{(1,u)}}
      \abs{P_\omega^{(u,N_{k-1})}(X_{N_k}\in \Delta) -
        \bP^{(u,N_{k-1})}(X_{N_k}\in \Delta)} + Cn_k^{-c\log n_k}\\
      & \le \abs{\Pi^{(1,u)}_k}Cn_k^{\theta d
        -\frac{d}{2}-\varepsilon} +
      Cn_k^{-c\log n_k}\\
      & \le C n_k^{\frac{d}{2}-\theta d + \theta d -
        \frac{d}{2}-\varepsilon}(\log n_k)^{3d} + Cn_k^{-c\log n_k}\\
      & \le C(n_k^{-\varepsilon}(\log n_k)^{3d}+n_k^{-c\log n_k}) \le
      C n_k^{-\varepsilon/2},
    \end{split}
  \end{eqnarray}
  where we used in the first inequality that by Lemma~3.6 from
  \cite{SteibersPhD2017}
  \begin{align*}
    \bP^{(0,0)}\bigl(\norm{X_n} > \sqrt{n} \log^3 n\bigr) \le Cn^{-c\log n}
  \end{align*}
  and that
  $\abs{\Pi^{(1,u)}_{k}} \le Cn_k^{d/2 - \theta d}(\log n_k)^{3d}$.

  It follows that \eqref{eq:secondInequality} is bounded from above by
  \begin{eqnarray}
    \begin{split}
      & C N_{k-1}^{-c\log N_{k-1}}+ \sum_{\substack{\Delta' \in
          \Pi^1_{k-1}\\ \text{is good}}} \sum_{u \in
        \Delta'}\bP^{(0,0)}(X_{N_{k-1}}\in
      \Delta')P_\omega^{(0,0)}(X_{N_{k-1}}=u\vert
      X_{N_{k-1}}\in\Delta')Cn_k^{-\varepsilon/2}\\
      &\hspace{1cm} +\sum_{\substack{\Delta' \in \Pi^1_{k-1}\\
          \text{is bad}}} \sum_{u \in
        \Delta'}\bP^{(0,0)}(X_{N_{k-1}}\in
      \Delta')P_\omega^{(0,0)}(X_{N_{k-1}}
      = u\vert X_{N_{k-1}}\in\Delta')\\
      &\hspace{3cm}\times\sum_{\Delta \in \Pi_{k}}
      \abs{P_\omega^{(u,N_{k-1})}(X_{N_k}\in \Delta)-
        \bP^{(u,N_{k-1})}(X_{N_k}\in \Delta)}\\
      & \quad \le C N_{k-1}^{-c\log N_{k-1}} + Cn_k^{-\varepsilon/2} +
      C\sum_{\substack{\Delta' \in \Pi^1_{k-1}\\\text{is
            bad}}}\bP^{(0,0)}(X_{N_{k-1}}\in\Delta').
    \end{split}
  \end{eqnarray}
  Now we want to find an estimate for the probability of hitting a bad
  box. For some $\beta>0$, to be chosen later, we consider the
  following event
  \begin{align}
    \label{eq:defn_G(N,nk)}
    G_{N,n_{k-1}} \coloneqq \Bigl\{ \sum_{\Delta \in \Pi_{k-1}}
    \ind{\Delta \text{ is}(k-1,\theta,\varepsilon)\text{-good}}
    \bP^{(0,0)}\bigl( X_{N_{k-1}} \in \Delta  \bigr) \ge 1-
    C'n_k^{-\beta} \Bigr\}
  \end{align}
  and define
  \begin{align}
    \label{eq:67}
    G_N \coloneqq \bigcap_{k=1}^{r(N)} G_{N,n_k}.
  \end{align}
  We want to mimic the proof in \cite{BergerCohenRosenthal2016} and
  for that we need to define a new type of boxes to approximate the
  density of bad boxes. The problem with following the proof in
  \cite{BergerCohenRosenthal2016} arises from the fact that our
  environment is, due to the dependence on infinitely long open paths,
  not i.i.d. To overcome that problem the idea is to exchange the
  environment $\xi$ with a process that only has finite range
  dependencies. We will use this idea to show in
  Proposition~\ref{prop_key} below that
  \begin{align}
    \label{eq:sec_inequ_lower_bound_G(N)}
    \bP(G_N) \ge 1- CN^{-c\log(N)}.
  \end{align}

  Note that $n_{k-1} =n_k^2$. Thus, on $G_N$ the expression
  \eqref{eq:secondInequality} is bounded from above by
  \begin{align}
    \label{eq:54}
    \begin{split}
      C N_{k-1}^{-c\log N_{k-1}} + Cn_k^{-\varepsilon/2} & +
      C\sum_{\substack{\Delta' \in \Pi^1_{k-1}\\ \text{is bad}}}
      \bP^{(0,0)}(X_{N_{k-1}}\in\Delta')\\
      & \le C N_{k-1}^{-c\log N_{k-1}} + Cn_k^{-\varepsilon/2} +
      C'n_{k-1}^{-\beta} \le C'' n_{k-1}^{-\varepsilon/4}.
    \end{split}
  \end{align}
  As can be seen in the proof of Proposition~\ref{prop_key} we can
  choose $\beta \ge \varepsilon/4$ to obtain the last inequality in
  \eqref{eq:54}.
\end{proof}

\begin{proof}[Proof of the analogue of an upper bound of (5.4) in
    \cite{BergerCohenRosenthal2016}] Consider
  \begin{multline}
    \label{eq:thirdInequality}
    \sum_{\Delta \in \Pi_{k}} \sum_{\Delta' \in \Pi_{k-1}}\Big\vert
    \sum_{u \in \Delta'} \bP^{(u,N_{k-1})}(X_{N_k} \in \Delta)\\
    \times [\bP^{(0,0)}(X_{N_{k-1}} \in
    \Delta')P^{(0,0)}_\omega(X_{N_{k-1}}=u\vert X_{N_{k-1}} \in
    \Delta') - \bP^{(0,0)}(X_{N_{k-1}}=u)].
  \end{multline}

  For any two probability measures $\mu$ and $\tilde \mu$ on $\bZ^d$
  we have
  \begin{align*}
    \sum_{u \in \Delta'} f(u) \mu(u) - \sum_{u \in \Delta'} f(u)
    \tilde{\mu}(u) \le \max_{u \in \Delta'} f(u) - \min_{u \in \Delta'}f(u).
  \end{align*}
  Thus, the expression \eqref{eq:thirdInequality} can be bounded from
  above by
  \begin{align}
    \label{eq:55}
    \begin{split}
      & \sum_{\Delta \in \Pi_{k}} \sum_{\Delta' \in \Pi_{k-1}}
      \bP^{(0,0)}(X_{N_{k-1}} \in \Delta') \big\vert \max_{u \in
        \Delta'} \bP^{(u,N_{k-1})}(X_{N_k}\in \Delta) - \min_{u \in
        \Delta'} \bP^{(u,N_{k-1})}(X_{N_k}\in
      \Delta) \big\vert\\
      & \quad \le  \sum_{\Delta' \in \Pi_{k-1}} \bP^{(0,0)}(X_{N_{k-1}} \in \Delta')\\
      & \qquad \times \sum_{\Delta \in \Pi^{(1,u)}_{k}} \big\vert
      \max_{u \in \Delta'} \bP^{(u,N_{k-1})}(X_{N_k}\in \Delta) -
      \min_{u \in \Delta'} \bP^{(u,N_{k-1})}(X_{N_k}\in \Delta)
      \big\vert + Cn_k^{-c\log n_k},
    \end{split}
  \end{align}
  where $\Pi^{(1,u)}_{k}$ is the set defined above.

  Using
  $\bP^{(u,N_{k-1})}(X_{N_k}\in \Delta) = \sum_{v \in \Delta}
  \bP^{(u,N_{k-1})}(X_{N_k} =v)$ we have
  \begin{align}
    \begin{split}
      \max_{u \in \Delta'} \bP^{(u,N_{k-1})}(X_{N_k}\in \Delta) &
      -\min_{u \in \Delta'} \bP^{(u,N_{k-1})}(X_{N_k}\in \Delta) \\
      & \le \sum_{v \in \Delta} \max_{u \in \Delta'}
      \bP^{(u,N_{k-1})}(X_{N_k} = v) - \min_{u \in \Delta'}
      \bP^{(u,N_{k-1})}(X_{N_k} = v)\\
      & \le \sum_{v \in \Delta} \operatorname{diam}(\Delta') \frac{C}{n_k^{(d+1)/2}}\\
      & \le (n_k^{\theta})^d n_{k-1}^{\theta} \frac{C}{n_k^{(d+1)/2}},
    \end{split}
  \end{align}
  where the second to last inequality follows by the annealed
  derivative estimates from Lemma~3.9 in \cite{SteibersPhD2017}.
  Altogether the expression \eqref{eq:thirdInequality} is bounded from
  above by
  \begin{align}
    \label{eq:88}
    \begin{split}
      & \sum_{\Delta' \in \Pi_{k-1}} \bP^{(0,0)}(X_{N_{k-1}} \in
      \Delta') \sum_{\Delta \in \Pi^{1,u}_{(k)}} (n_k^{\theta})^d
      n_{k-1}^{\theta} \frac{C}{n_k^{(d+1)/2}} + Cn_k^{-c\log n_k}\\
      & \quad \le C \sum_{\Delta' \in \Pi_{k-1}}
      \bP^{(0,0)}(X_{N_{k-1}} \in \Delta') \Bigl(
      \frac{dn_{k-1}^{\theta} \sqrt{n_k}(\log n_k)^3}{n_k^{\theta}}
      \Bigr)^d (n_k^{\theta})^d
      n_{k-1}^{\theta} \frac{C}{n_k^{(d+1)/2}} + Cn_k^{-c\log n_k}\\
      & \quad \le C (\log n_k)^{3d} \frac{n_{k-1}^{\theta}}{n_k^{1/2}}
      + Cn_k^{-c\log n_k} \le C \frac{(\log
        n_k)^{3d}}{n_k^{1/2-2\theta}} + Cn_k^{-c\log n_k}.
    \end{split}
  \end{align}
\end{proof}

\begin{proof}[Proof of the analogue of an upper bound of (5.5) in
  \cite{BergerCohenRosenthal2016}] Consider
  \begin{align}
    \label{eq:fourthInequality}
    \sum_{\Delta \in \Pi_{k}} \sum_{\Delta' \in \Pi_{k-1}} \Big\vert
    \sum_{u \in \Delta'} \bP^{(0,0)}(X_{N_{k-1}} =u)
    \bP^{(u,N_{k-1})}(X_{N_k} \in \Delta)
    - \bP^{(0,0)}(X_{N_k} \in \Delta , X_{N_{k-1}}\in \Delta') \Big\vert.
  \end{align}
  Recall the regeneration times introduced in
  \cite{BirknerCernyDepperschmidtGantert2013}. There they are defined
  for a random walk on the backbone of the oriented percolation
  cluster, whereas we allow the random walk to start outside the
  cluster. In Remark 2.3 Birkner~et.~al note that the local
  construction, which they use to obtain the regeneration times, can
  be extended to starting points outside the cluster. Let
  $B_{m,\tilde{m}}$ be the event that the first regeneration time
  greater than $m$ will happen before $m+\tilde{m}^{\beta}$, for some
  small constant $\beta>0$ to be tuned appropriately later. By
  Lemma~2.5 from \cite{BirknerCernyDepperschmidtGantert2013} the
  distribution of the regeneration increments has exponential tail
  bounds, and thus $\bP(B_{m,\tilde{m}})\le C\mathrm{e}^{-cm^\beta}$.
  First, note that by the theorem of total probability and the
  triangle inequality \eqref{eq:fourthInequality} is bounded from
  above by
  \begin{align}
    \label{eq:56}
    \begin{split}
      \sum_{\Delta' \in \Pi_{k-1}} & \sum_{u \in \Delta'}
      \bP^{(0,0)}(X_{N_{k-1}} =u) \sum_{\Delta \in \Pi_{k}}
      \abs{\bP^{(u,N_{k-1})}(X_{N_k}\in \Delta) - \bP^{(0,0)}(X_{N_k}
        \in \Delta\vert X_{N_{k-1}} = u)}\\
      & \le \sum_{\Delta' \in \Pi_{k-1}} \sum_{u \in \Delta'}
      \bP^{(0,0)}(X_{N_{k-1}} =u) \\
      & \qquad \quad \times \sum_{\Delta \in \Pi_{k}} \bigl(\big\vert
      \bP^{(u,N_{k-1})}(X_{N_k}\in \Delta) -\bP^{(0,0)}(X_{N_k} \in
      \Delta, B_{N_{k-1},n_k}\vert X_{N_{k-1}} = u) \big\vert\\
      & \qquad \qquad \qquad \qquad \qquad + \bP^{(0,0)}(X_{N_k} \in
      \Delta, B_{N_{k-1},n_k}^\compl\vert X_{N_{k-1}} = u)\bigr)
    \end{split}
  \end{align}
  First note that
  \begin{multline*}
    \sum_{\Delta' \in \Pi_{k-1}} \sum_{u \in \Delta'}
    \bP^{(0,0)}(X_{N_{k-1}} =u)\sum_{\Delta \in
      \Pi_{k}}\bP^{(0,0)}(X_{N_k} \in \Delta,
    B_{N_{k-1},n_k}^\compl\vert X_{N_{k-1}} = u)\\
    =\bP(B_{N_{k-1},n_k}^\compl)\le C \mathrm{e}^{-cn_k^\beta}.
  \end{multline*}
  The remaining part of the right hand side of \eqref{eq:56} is
  bounded from above by
  \begin{multline*}
    \sum_{\Delta' \in \Pi_{k-1}} \sum_{u \in \Delta'}
    \bP^{(0,0)}(X_{N_{k-1}} =u) \sum_{\Delta \in \Pi_{k}}
    \bigl(\bP^{(u,N_{k-1})}(X_{N_k}\in \Delta, B_{0,n_k}^\compl) \\
    +\big\vert\bP^{(u,N_{k-1})}(X_{N_k}\in \Delta, B_{0,n_k})
    -\bP^{(0,0)}(X_{N_k} \in \Delta, B_{N_{k-1},n_k}\vert X_{N_{k-1}}
    = u) \big\vert \bigr).
  \end{multline*}
  Using the same arguments as above we obtain
  \begin{align*}
    \sum_{\Delta' \in \Pi_{k-1}} \sum_{u \in \Delta'} \bP^{(0,0)}(X_{N_{k-1}} =u)
    \sum_{\Delta \in \Pi_{k}}
    \bP^{(u,N_{k-1})}(X_{N_k}\in \Delta, B_{0,n_k}^\compl) =
    \bP(B_{N_{k-1},n_k}^\compl)\le C \mathrm{e}^{-cn_k^\beta}
  \end{align*}
  and thus it remains to find a suitable upper bound for
  \begin{multline*}
    \sum_{\Delta' \in \Pi_{k-1}} \sum_{u \in \Delta'} \bP^{(0,0)}(X_{N_{k-1}} =u)\\
    \sum_{\Delta \in \Pi_{k}}\big\vert\bP^{(u,N_{k-1})}(X_{N_k}\in
    \Delta, B_{0,n_k}) -\bP^{(0,0)}(X_{N_k} \in \Delta,
    B_{N_{k-1},n_k}\vert X_{N_{k-1}} = u) \big\vert
  \end{multline*}

  Let $\tilde{\tau}_{N_{k-1}}$ denote the first regeneration time
  greater than $N_{k-1}$. By splitting the probabilities above into
  the sum over the possible times at which the regeneration can occur
  and the possible sites at which the random walk can be at the time
  of the regeneration we see that the term in the above display equals
  to
  \begin{multline}
    \label{eq:58}
    \begin{split}
      \sum_{\Delta' \in \Pi_{k-1}} & \sum_{u \in \Delta'}
      \bP^{(0,0)}(X_{N_{k-1}} =u) \\
      & \quad \times \sum_{\Delta \in \Pi_{k}} \Big\vert
      \sum_{\substack{t \in [N_{k-1},N_{k-1}+n_k^\beta]\\ v \in
          \bZ^d:\norm{u-v}\le n_k^\beta }} \bP^{(v,t)}(X_{N_k}\in
      \Delta)\bP^{(u,N_{k-1})}(\tilde{\tau}_{N_{k-1}}=t,X_{\tilde{\tau}_{N_{k-1}}}=v)
      \\
      & \qquad \qquad - \sum_{\substack{t \in
          [N_{k-1},N_{k-1}+n_k^\beta]\\v \in \bZ^d:\norm{u-v}\le
          n_k^{\beta}}} \bP^{(v,t)}(X_{N_k} \in
      \Delta)\bP^{(0,0)}(\tilde{\tau}_{N_{k-1}}=t,X_{\tilde{\tau}_{N_{k-1}}}=v\vert
      X_{N_{k-1}}=u)\Big\vert.
    \end{split}
  \end{multline}
  The modulus in the last two lines of the above display is bounded
  from above by
  \begin{align}
    \label{eq:59}
    \begin{split}
      & \Big\vert\max_{\substack{t \in [N_{k-1},N_{k-1}+n_k^\beta]\\
          v\in \bZ^d:\norm{u-v}\le n_k^\beta}}\bP^{(v,t)}(X_{N_k}\in
      \Delta) \sum_{\substack{t \in [N_{k-1},N_{k-1}+n_k^\beta]\\v \in
          \bZ^d:\norm{u-v}\le n_k^\beta
        }}\bP^{(u,N_{k-1})}(\tilde{\tau}_{N_{k-1}}=t,X_{\tilde{\tau}_{N_{k-1}}}=v)\\
      &\hspace{0.5cm}- \min_{\substack{t \in [N_{k-1},N_{k-1}+n_k^\beta]\\
          v\in \bZ^d:\norm{u-v}\le n_k^\beta}}\bP^{(v,t)}(X_{N_k}\in
      \Delta) \sum_{\substack{t \in [N_{k-1},N_{k-1}+n_k^\beta]\\v \in
          \bZ^d:\norm{u-v}\le n_k^\beta
        }}\bP^{(0,0)}(\tilde{\tau}_{N_{k-1}}=t,X_{\tilde{\tau}_{N_{k-1}}}=v\vert
      X_{N_{k-1}}=u)\Big\vert\\
      & \le\Big\vert \max_{\substack{t \in [N_{k-1},N_{k-1}+n_k^\beta]\\
          v\in \bZ^d:\norm{u-v}\le n_k^\beta}}\bP^{(v,t)}(X_{N_k}\in
      \Delta) - \min_{\substack{t \in [N_{k-1},N_{k-1}+n_k^\beta]\\
          v\in \bZ^d:\norm{u-v}\le n_k^\beta}}\bP^{(v,t)}(X_{N_k}\in
      \Delta)\Big\vert\\
      &\hspace{1cm} + \max_{\substack{t \in
          [N_{k-1},N_{k-1}+n_k^\beta]\\ v\in \bZ^d:\norm{u-v}\le
          n_k^\beta}}\bP^{(v,t)}(X_{N_k}\in
      \Delta)\bP^{(u,N_{k-1})}(\tilde{\tau}_{N_{k-1}}
      >N_{k-1}+n_k^\beta) \\
      &\hspace{1cm}+ \min_{\substack{t \in [N_{k-1},N_{k-1}+n_k^\beta]\\
          v\in \bZ^d:\norm{u-v}\le n_k^\beta}}\bP^{(v,t)}(X_{N_k}\in
      \Delta)\bP^{(0,0)}(\tilde{\tau}_{N_{k-1}}
      >N_{k-1}+n_k^\beta\vert X_{N_{k-1}}=u)
    \end{split}
  \end{align}
  Plugging that into the sums in \eqref{eq:58} we obtain that an upper
  bound of \eqref{eq:fourthInequality} is given by
  \begin{align}
    \label{eq:60}
    \begin{split}
      & \sum_{\Delta' \in \Pi_{k-1}} \sum_{u \in \Delta'}
      \bP^{(0,0)}(X_{N_{k-1}} =u) \\
      &\hspace{2.5cm}\sum_{\Delta \in \Pi_{k}} \Big\vert
      \max_{\substack{t \in [N_{k-1},N_{k-1}+n_k^\beta]\\ v\in
          \bZ^d:\norm{u-v}\le n_k^\beta}}\bP^{(v,t)}(X_{N_k}\in
      \Delta) - \min_{\substack{t \in [N_{k-1},N_{k-1}+n_k^\beta]\\
          v\in \bZ^d:\norm{u-v}\le n_k^\beta}}\bP^{(v,t)}(X_{N_k}\in
      \Delta)\Big\vert\\
      &+ \sum_{\Delta' \in \Pi_{k-1}} \sum_{u \in \Delta'}
      \bP^{(0,0)}(X_{N_{k-1}} =u)\\
      &\hspace{2.5cm}\sum_{\Delta \in \Pi_{k}}\max_{\substack{t \in
          [N_{k-1},N_{k-1}+n_k^\beta]\\ v\in \bZ^d:\norm{u-v}\le
          n_k^\beta}}\bP^{(v,t)}(X_{N_k}\in
      \Delta)\bP^{(u,N_{k-1})}(\tilde{\tau}_{N_{k-1}}
      >N_{k-1}+n_k^\beta)\\
      &+ \sum_{\Delta' \in \Pi_{k-1}} \sum_{u \in \Delta'}\sum_{\Delta
        \in \Pi_{k}}\min_{\substack{t \in
          [N_{k-1},N_{k-1}+n_k^\beta]\\ v\in \bZ^d:\norm{u-v}\le
          n_k^\beta}}\bP^{(v,t)}(X_{N_k}\in
      \Delta)\bP^{(0,0)}(\tilde{\tau}_{N_{k-1}} >N_{k-1}+n_k^\beta,
      X_{N_{k-1}}=u)\\
      &\hspace{4.5cm}+ C\mathrm{e}^{-cn_k^\beta}.
    \end{split}
  \end{align}

  Now define $\Pi_{k}^{1,u,\beta}$ as the set boxes $\Delta \in \Pi_k$
  for which
  \begin{align}
    \label{eq:61}
    \Delta \cap \Big(\bigcup_{v\colon \norm{v-u} \le n_k^{\beta}} \{x\in \bZ^d:\norm{x - v}
    \le \sqrt{n_k}\log^3 n_k \}\Big) \neq \emptyset.
  \end{align}
  Using Lemma~3.6 from \cite{SteibersPhD2017} we obtain
  \begin{align*}
    \sum_{\Delta \notin \Pi_{k}^{1,u,\beta}}\bP^{(v,0)}(X_{N_k-t} \in
    \Delta) \le \bP^{(v,0)}\bigl(\abs{X_{N_k-t}-v}>\sqrt{N_k-t}\log^3 N_k-t\bigr)
    \le Cn_k^{-c\log n_k}
  \end{align*}
  for all $v\in \bZ^d$ with $\norm{v-u}\le n_k^\beta$ and all
  $t\in [N_{k-1},N_{k-1}+n_k^\beta]$. Using this it follows
  \begin{align}
    \label{eq:62}
    \begin{split}
      & \sum_{\Delta' \in \Pi_{k-1}} \sum_{u \in \Delta'}
      \bP^{(0,0)}(X_{N_{k-1}} =u)\\
      &\hspace{2cm}\sum_{\Delta \in \Pi_{k}}\max_{\substack{t \in
          [N_{k-1},N_{k-1}+n_k^\beta]\\ v\in \bZ^d:\abs{u-v}\le
          n_k^\beta}}\bP^{(v,t)}(X_{N_k}\in
      \Delta)\bP^{(u,N_{k-1})}(\tilde{\tau}_{N_{k-1}}
      >N_{k-1}+n_k^\beta)\\
      & \le
      \abs{\Pi^{1,u,\beta}_k}\bP^{(0,0)}(\tilde{\tau}_{N_{k-1}}>N_{k-1}+n_k^\beta)
      + Cn_k^{-c\log n_k} \\
      & \le n_k^{\beta d} n_k^{d/2(1-2\theta)}(\log
      n_k)^{3d}C\mathrm{e}^{-cn_k^\beta} + Cn_k^{-c\log n_k}\le
      Cn_k^{-c\log n_k},
    \end{split}
  \end{align}
  where we have used the fact that, by the definition of
  $\Pi_k^{(1,u)}$ in \eqref{eq:defn_Pi(1,u)},
  $\vert \Pi_k^{1,u,\beta} \vert \le n_k^{\beta d}\vert \Pi_k^{(1,u)}
  \vert$. Similarly
  \begin{align}
    \label{eq:63}
    \begin{split}
      & \sum_{\Delta' \in \Pi_{k-1}} \sum_{u \in \Delta'}\sum_{\Delta
        \in \Pi_{k}}\min_{\substack{t \in [N_{k-1},N_{k-1}+n_k^\beta]\\
          v\in \bZ^d:\abs{u-v}\le n_k^\beta}}\bP^{(v,t)}(X_{N_k}\in
      \Delta)\bP^{(0,0)}(\tilde{\tau}_{N_{k-1}} >N_{k-1}+n_k^\beta,
      X_{N_{k-1}}=u)\\
      & \qquad \le \abs{\Pi^{1,u,\beta}_k}
      \bP^{(0,0)}(\tilde{\tau}_{N_{k-1}}>N_{k-1}+n_k^\beta) +
      Cn_k^{-c\log n_k}\\
      & \qquad \le n_k^{\beta d}n_k^{d/2(1-2\theta)}(\log
      n_k)^{3d}C\mathrm{e}^{-cn_k^\beta} + Cn_k^{-c\log n_k}\le
      Cn_k^{-c\log n_k}.
    \end{split}
  \end{align}
  Altogether it follows that \eqref{eq:fourthInequality} is bounded
  from above by
  \begin{multline}
    \label{eq:64}
    \sum_{\Delta' \in \Pi_{k-1}}
    \sum_{u \in \Delta'} \bP^{(0,0)}(X_{N_{k-1}} =u) \\
    \times \sum_{\Delta \in \Pi^{1,u,\beta}_{(k)}} \Big\vert
    \max_{\substack{t \in [N_{k-1},N_{k-1}+n_k^\beta]\\ v\in
        \bZ^d:\norm{u-v}\le n_k^\beta}}\bP^{(v,t)}(X_{N_k}\in \Delta)
    - \min_{\substack{t \in [N_{k-1},N_{k-1}+n_k^\beta]\\ v\in
        \bZ^d:\norm{u-v}\le n_k^\beta}}\bP^{(v,t)}(X_{N_k}\in
    \Delta)\Big\vert\\
    + Cn_k^{-c\log n_k} + C\mathrm{e}^{-cn_k^\beta}
  \end{multline}
  Using the annealed derivative estimates from
  Lemma~\ref{lem:annealed_derivative_estimates} we obtain
  \begin{align}
        \notag
      &\Big\vert \max_{\substack{t \in [N_{k-1},N_{k-1}+n_k^\beta]\\
    v\in \bZ^d:\norm{u-v}\le n_k^\beta}}\bP^{(v,t)}(X_{N_k}\in
    \Delta) - \min_{\substack{t \in [N_{k-1},N_{k-1}+n_k^\beta]\\ v\in
    \bZ^d:\norm{u-v}\le n_k^\beta}}\bP^{(v,t)}(X_{N_k}\in
    \Delta)\Big\vert\\
    \notag
      & \quad \le \abs{\Delta} \Big\vert \max_{\substack{t \in
        [N_{k-1},N_{k-1}+n_k^\beta]\\ v\in \bZ^d:\norm{u-v}\le
    n_k^\beta\\x\in \Delta}}\bP^{(v,t)}(X_{N_k}=x)-
    \min_{\substack{t \in [N_{k-1},N_{k-1}+n_k^\beta]\\
    v\in \bZ^d:\norm{u-v}\le n_k^\beta\\y\in
    \Delta}}\bP^{(v,t)}(X_{N_k}=y)\Big\vert\\
    \notag
      & \quad \le \abs{\Delta} C(4n_k^\beta +
        n_k^{\theta})n^{-\frac{d+1}{2}}\\
    \label{eq:65}
      & \quad \le n_k^{d\theta}C(4n_k^\beta +
        n_k^{\theta})n_k^{-\frac{d+1}{2}}.
  \end{align}
  Now if we choose $\beta =\theta$ and $\theta$ small enough, we get
  that the above expression is smaller than $Cn_k^{-\frac{2d+1}{4}}$.
  Putting everything together we get the upper bound
  \begin{align*}
    C\mathrm{e}^{-cn_k^\theta}
    & + Cn_k^{-c\log n_k}+ \sum_{\Delta' \in
      \Pi_{k-1}} \sum_{u \in \Delta'} \bP^{(0,0)}(X_{N_{k-1}} =u)
      \sum_{\Delta \in \Pi_{k}^{1,u}} n_k^{-\frac{d}{2}- \frac{1}{4}}\\
    & \le C\mathrm{e}^{-cn_k^\theta} + Cn_k^{-c\log n_k} + \sum_{\Delta' \in
      \Pi_{k-1}} \bP^{(0,0)}(X_{N_{k-1}} \in \Delta')\abs{\Pi_{k}^{1,u}}
      n_k^{-\frac{d}{2}- \frac{1}{4}}\\
    & \le C\mathrm{e}^{-cn_k^\theta} + Cn_k^{-c\log n_k} + Cn_k^{(1/2 -
      \theta)d}(\log n_k)^{3d} n_k^{-\frac{d}{2}- \frac{1}{4}} \\
    & = C\mathrm{e}^{-cn_k^\theta} + Cn_k^{-c\log n_k} + C(\log n_k)^{3d}2n_k^{\theta-1/4}
  \end{align*}
  Thus, recalling equation \eqref{eq:fourthInequality}, we obtain
  \begin{multline}
    \label{eq:89}
    \sum_{\Delta \in \Pi_{k}} \sum_{\Delta' \in \Pi_{k-1}} \Big\vert
    \sum_{u \in \Delta'} \bP^{(0,0)}(X_{N_{k-1}} =u)
    \bP^{(u,N_{k-1})}(X_{N_k} \in \Delta)
    - \bP^{(0,0)}(X_{N_k} \in \Delta , X_{N_{k-1}}\in \Delta') \Big\vert\\
    \le Cn_k^{-c}
  \end{multline}
  for some constants $C,c>0$.
\end{proof}

\begin{proof}[Proof of Proposition~\ref{prop_main_thm_5.1}]
  To prove Proposition~\ref{prop_main_thm_5.1} we need to show
  inequality \eqref{eq:prop_main_thm_5.1} which we recall here
  \begin{align*}
    \lambda_k \le \lambda_{k-1} +Cn_k^{-\alpha},\;
    \quad \forall \: 1\le k \le r(N).
  \end{align*}
  for some positive constants $\alpha$ and $C$ on the event $G(N)$
  from \eqref{eq:67}.

  Fix $\omega\in G(N)$. Recall the definition
  \begin{align*}
    \lambda_k = \sum_{\Delta \in \Pi_{k}}
    \big|P_{\omega}^{(0,0)}(X_{N_k} \in \Delta)- \bP^{(0,0)} (X_{N_k} \in \Delta)\big|
  \end{align*}
  from equation \eqref{eq:42}. Furthermore, we recall
  \eqref{eq:firstInequality}, \eqref{eq:secondInequality},
  \eqref{eq:thirdInequality} and \eqref{eq:fourthInequality} for which
  we just estimated upper bounds.
  \begin{multline*}
    \eqref{eq:firstInequality} = \sum_{\Delta \in \Pi_{k}} \sum_{\Delta' \in \Pi_{k-1}}\Big\vert
    \sum_{u \in \Delta'}
    P^{(u,N_{k-1})}_{\omega}(X_{N_k} \in \Delta)\\
    \times \bigl[P^{(0,0)}_\omega(X_{N_{k-1}}=u ) - \bP^{(0,0)}(X_{N_{k-1}}
    \in \Delta')P^{(0,0)}_\omega(X_{N_{k-1}}=u \vert X_{N_{k-1}} \in
    \Delta')\bigr] \Big\vert,
  \end{multline*}
  \begin{multline*}
    \eqref{eq:secondInequality} = \sum_{\Delta \in \Pi_{k}}
    \sum_{\Delta' \in \Pi_{k-1}}\Big\vert \sum_{u \in \Delta'}
    \bP^{(0,0)}(X_{N_{k-1}} \in \Delta')P^{(0,0)}_\omega(X_{N_{k-1}}
    =u\vert X_{N_{k-1}} \in \Delta')\\
    \times [P^{(u,N_{k-1})}_\omega(X_{N_k} \in \Delta) -
    \bP^{(u,N_{k-1})}(X_{N_k}\in \Delta)]\Big\vert,
  \end{multline*}
  \begin{multline*}
    \eqref{eq:thirdInequality} = \sum_{\Delta \in \Pi_{k}} \sum_{\Delta' \in \Pi_{k-1}}\Big\vert
    \sum_{u \in \Delta'} \bP^{(u,N_{k-1})}(X_{N_k} \in \Delta)\\
    \times [\bP^{(0,0)}(X_{N_{k-1}} \in
    \Delta')P^{(0,0)}_\omega(X_{N_{k-1}}=u\vert X_{N_{k-1}} \in
    \Delta') - \bP^{(0,0)}(X_{N_{k-1}}=u)],
  \end{multline*}
  \begin{multline*}
    \eqref{eq:fourthInequality} = \sum_{\Delta \in \Pi_{k}} \sum_{\Delta' \in \Pi_{k-1}} \Big\vert
    \sum_{u \in \Delta'} \bP^{(0,0)}(X_{N_{k-1}} =u)
    \bP^{(u,N_{k-1})}(X_{N_k} \in \Delta)
    - \bP^{(0,0)}(X_{N_k} \in \Delta , X_{N_{k-1}}\in \Delta') \Big\vert.
  \end{multline*}
  Note that for $\lambda_k$, by the triangle inequality, we obtain
  \begin{align*}
    \lambda_k \leq \eqref{eq:firstInequality} +
    \eqref{eq:secondInequality} + \eqref{eq:thirdInequality} +
    \eqref{eq:fourthInequality}.
  \end{align*}
  Thus, using the proven estimates, \eqref{eq:17}, \eqref{eq:54},
  \eqref{eq:88} and \eqref{eq:89}, for each of the summands
  respectively we gain
  \begin{align*}
    \lambda_k \leq \lambda_{k-1} + C''n_{k-1}^{-\epsilon/4} +  C
    \frac{(\log n_k)^{3d}}{n_k^{1/2-2\theta}}
    + Cn_k^{-c\log n_k} + C n_k^{-c}\le \lambda_{k-1} + \tilde{C}n_k^{-\alpha}
  \end{align*}
  for appropriate choices of $\alpha>0$ and $\tilde{C}>0$. The fact
  that $\bP(G_N) \ge 1 - CN^{-c\log N}$ is proved in
  Proposition~\ref{prop_key}.
\end{proof}

\begin{prop}
  \label{prop_key}
  For the events $G_N$ from \eqref{eq:67} there exists $N_0 \in \bN$
  such that, for all $N\ge N_0$ we have that
  \begin{align}
    \label{eq:70}
    \bP(G_N) \ge 1 - CN^{-c\log N}.
  \end{align}
\end{prop}
Let $\beta >0$ and put $f(n_k)=\log^2 n_k$. First we need another
notion of \emph{good} sites. Given a realization $\omega$ we define
for all $(x,\ell) \in \bZ^d \times \bZ$ the set $C_{m}(x,\ell)$ as the
set of sites at time $\ell + m \in \bZ$ which can be reached from
$(x,\ell)$ via an open path w.r.t.\ $\omega$. We start by defining for
$k = 1,2,\dots$ a field
$\tilde{\xi}^k\coloneqq (\tilde{\xi}^k_t(x))_{t\in\bZ^d}$ as follows
\begin{enumerate}[(i)]
\item $\tilde{\xi}^k_t(x)=\xi_t(x)$ for all
  $(x,t)\in\bZ^d\times\{n_k+f(n_k),n_k+f(n_k)+1,\dots\}$
\item For all
  $(x,t)\in \bZ^d\times \{\dots,n_k+f(n_k)-2,n_k+f(n_k)-1\}$ we set
  $\tilde{\xi}^k_t(x)=1$ if $C_{n_k+f(n_k)-t}(x,t)\neq \emptyset$.
  Otherwise we set $\tilde{\xi}^k_t(x)=0$.
\end{enumerate}
Note that $\xi\le \tilde{\xi}^k$ since for $(x,t)$ with
$t<n_k+f(n_k)$ we set $\tilde\xi_t(x)=1$ if $(x,t)$ has an open path
of length at least $n_k+f(n_k)-t$ instead of requiring an infinite
open path. For $\xi_t(x) \ne \tilde{\xi}_t^k(x)$ we necessarily must
have $t < n_k+f(n_k)$ and there must exist an open path started at
$(x,t)$ whose length is at least $n_k+f(n_k)-t$ but the contact
process started at $(x,t)$ has to eventually die out, i.e.\ there is
no infinite open path starting in $(x,t)$.

The following lemma gives us an upper bound on that probability. The
result is well known in the oriented percolation and contact process
world. For a proof see for instance in Lemma~A.1.\ in
\cite{BirknerCernyDepperschmidtGantert2013}.
\begin{lem}
  \label{lem:lower_bound_on_survival_prob}
  For $p>p_c$ there exist $C,c>0$ such that for all
  $(x,t)\in\bZ^d\times\bZ$
  \begin{align*}
    \bP\Big((x,t)\to^\omega \bZ^d\times\{t+n\}\text{ and }
    (x,t)\nrightarrow^\omega \bZ^d\times\{\infty\}\Big) \leq
    C \mathrm{e}^{-c n},\quad n\in\bN.
  \end{align*}
\end{lem}

As a direct consequence we get the following corollary.
\begin{cor}
  \label{cor:coupling_xi_and_tilde_xi}
  For $x\in\bZ^d$ define
  \begin{align*}
    D_{n_k}(x)\coloneqq \bigl(x+
    [-n_{k-1}^\theta-n_k,n_{k-1}^\theta+n_k]^d\times[0,n_k]\bigr) \cap (\bZ^d\times\bZ).
  \end{align*}
  For $p>p_c$ there exist constants $C,c>0$ such that
  \begin{align}
    \label{eq:68}
    \bP\Big(\tilde{\xi}^k_t(y)=\xi_t(y)\text{ for all }(y,t)\in
    D_{n_k}(x) \Big) \ge 1 - C \mathrm{e}^{-c \log^2 n_k}.
  \end{align}
\end{cor}
\begin{proof}
  Note that $\theta>0$ is a small constant and can be chosen such that
  we have $n_{k-1}^\theta=n_k^{2\theta}\le n_k$ and thus
  $\abs{D_{n_k}(x)}\le 2^dn_k^{d+1}$. By definition of $\tilde\xi^k$
  $\tilde{\xi}^k_t(y)\ne\xi_t(y)$ implies that there is at least one
  open but finite paths whose length is larger that $f(n_k)$. Using
  Lemma~\ref{lem:lower_bound_on_survival_prob} the assertion
  \eqref{eq:68} follows by the choice of $f(n_k)=\log^2 n_k$. (Here
  one can see that other choices of $f(n_k)$ are possible as well.)
\end{proof}

Let $(\tilde{X})$ be a random walk in the environment $\tilde{\xi}^k$
with transition probabilities given by
\begin{align}
  P_{\omega,\tilde{\xi}^k}(\tilde{X}_{n+1} = x \, \vert \, \tilde{X}_n = y) =
  \begin{cases}
    \abs{U(x,n) \cap \tilde{\mathcal{C}}^k}^{-1} &\text{if } (x,n)
    \in
    \tilde{\mathcal{C}}^k \text{ and } (y,n+1) \in U(x,n)\cap
    \tilde{\mathcal{C}}^k,\\
    \abs{U(x,n)}^{-1} &\text{if } (x,n) \notin \tilde{\mathcal{C}}^k
    \text{ and } (y,n+1) \in U(x,n), \\
    0 &\text{otherwise,}
  \end{cases}
\end{align}
where
$\tilde{\mathcal{C}}^k\coloneqq \{(x,n)\in\bZ^d\times \bZ:
\tilde{\xi}^k_n(x)=1\}$.

Given a realisation $\omega$, we say that $(x,m)$ is
$(k-1,\theta,\varepsilon,\tilde{\xi}^k)$-good if it satisfies the
conditions from Definition~\ref{def:good boxes} with $\xi$ replaced by
$\tilde{\xi}^k$ and $X$ replaced by $\tilde X$ in the quenched
probabilities.

\begin{lem}
  For \label{lem:lowerBoundApproxGoodPoint} all
  $(x,t)\in \bZ^d\times \bZ$ we have that
  \begin{align}
    \bP((x,t) \text{ is }
    (k-1,\theta,\varepsilon,\tilde\xi^k) \text{-good})\ge 1 -
    Cn_k^{-c\log n_k}.
  \end{align}
\end{lem}
\begin{proof}
  Due to Lemma~\ref{lem:lowerBoundGoodPoint} it suffices to show that
  with probability at least $1-Cn_k^{-c\log n_k}$ we have
  $\tilde{\xi}^k_t(y)=\xi_t(y)$ for all $(y,t)\in D_{n_k}(x)$. This
  exactly the assertion of
  Corollary~\ref{cor:coupling_xi_and_tilde_xi}. On that event $(x,t)$
  is $(k-1,\theta,\varepsilon)$-good iff $(x,t)$ is
  $(k-1,\theta,\varepsilon,\tilde{\xi}^k)$-good.
\end{proof}

\begin{proof}[Proof of Proposition~\ref{prop_key}]
  Recall the definition of $G_{N,n_{k-1}}$ from
  \eqref{eq:defn_G(N,nk)}. To estimate the probability of hitting a
  bad box we can now mimic the proof in
  \cite{BergerCohenRosenthal2016} since we get a lower bound by
  estimating the probability for the
  $(k-1,\theta,\varepsilon,\tilde{\xi}^k)$-good boxes. By construction
  those boxes are independent of each other at distance $>5n_k$.
  Define
  \begin{align}
    \Pi^{(0)}_{k-1} = \{\Delta' \in \Pi_{k-1}^1 \colon
    \text{dist}(\Delta',\underline{0})\le \floor{\sqrt{N_{k-1}}} \}
  \end{align}
  and for $r\ge 1$ let
  \begin{align}
    \Pi_{k-1}^{(r)} = \{ \Delta'\in \Pi_{k-1}^1 \colon
    \floor{2^{r-1}\sqrt{N_{k-1}}} <\text{dist}(\Delta',\underline{0})\le
    \floor{2^r\sqrt{N_{k-1}}}  \}.
  \end{align}
  $(\Pi_{k-1}^{(r)})_{r\ge 0}$ is a partition of $\Pi_{k-1}^1$ into
  disjoint subsets according to the distance of the boxes from the
  origin which allows us to estimate the hitting probabilities of the
  bad boxes. Using the annealed local CLT
  (Theorem~\ref{lem:annealed_local_CLT}), we have
  \begin{multline}
    \sum_{\substack{\Delta' \in \Pi_{k-1}^1 \\ \text{is bad}}}
    \bar{\bP}^{(0,0)}(X_{N_{k-1}} \in \Delta')\\
    \le \sum_{r=0}^{\lceil \log_2 (\log N_{k-1})^3 \rceil}
    \abs{\Pi_{k-1}^{(r)} \cap \{
    (k-1,\theta,\varepsilon,\tilde\xi^k)\text{-bad boxes}\}}
    CN_{k-1}^{-d/2}\mathrm{e}^{-cr^2}
  \end{multline}
  holds for some constants $C,c >0$ and $\bar{\bP}$ is the measure for
  the changed environments $\tilde{\xi}^k$.

  In order to estimate the number of bad boxes in each
  $\Pi^{(r)}_{k-1}$ we define the event
  $\widetilde{G}_N = \widetilde{G}_N(C)$ by
  \begin{align}
    \widetilde{G}_N \coloneqq  \bigcap_{k=1}^{r(N)} \bigcap_{r=0}^{\lceil \log_2
    (\log N_{k-1})^3 \rceil} \left\{ \abs{\Pi^{(r)}_{k-1}\cap
    \{(k-1,\theta,\varepsilon,\tilde\xi^k)\text{-bad boxes}\}}
    \le C\abs{\Pi^{(r)}_{k-1}}n_{k-1}^{-\beta} \right\},
  \end{align}
  where $\beta>0$ is a constant to be tuned later. Let
  $\tilde{p}_{k-1}$ be the probability for a box
  $\Delta' \in \Pi_{k-1}$ to be
  $(k-1,\theta,\varepsilon,\tilde\xi^k)$-bad. Note that
  $\tilde{p}_{k} \in \mathcal{O}(n_k^{-c\log n_k})$ and on the event
  $\widetilde{G}_N$
  \begin{multline}
    \sum_{\substack{\Delta' \in \Pi_{k-1}^1 \\ \text{is bad}}}
    \bar{\bP}^{(0,0)}(X_{N_{k-1}} \in \Delta') \le
    \sum_{r=0}^{\lceil \log_2 (\log N_{k-1})^3 \rceil} C
    \abs{\Pi_{k-1}^{(r)}}n_{k-1}^{-\beta} N_{k-1}^{-d/2}\mathrm{e}^{-cr^2}\\
    \le \sum_{r=0}^{\lceil \log_2 (\log N_{k-1})^3
    \rceil}C2^{dr}(\sqrt{N_{k-1}}/n_{k-1}^\theta)^d N_{k-1}^{-d/2}
    \mathrm{e}^{-cr^2}n_{k-1}^{-\beta} \le C n_{k-1}^{-(\beta+\theta d)}.
  \end{multline}
  Now it suffices to show that
  $\bP(\widetilde{G}_N(C))\ge 1 -CN^{-c\log(N)}$ for some constant
  $C>0$. To do so, fix $k\ge 1$ and note that boxes
  $\Delta' \in \Pi_{k-1}$ at distance $5n_k$ are, by construction of
  $\tilde{\xi}^k$, good or bad independently of each other. To see
  this note that $2(n_{k-1}^\theta+n_k+f(n_k))<5n_k$ and recall that
  $\tilde\xi^k_t(y)=1$ if there exists an open path connecting $(y,t)$
  to $\bZ^d\times\{n_k+f(n_k)\}$ and $\tilde\xi^k_t(y)=0$ otherwise.
  Let $(\Pi^{r,j}_{k-1})_j$ be a partition of $\Pi^{(r)}_{k-1}$ into
  at most $(5n_k)^d$ subsets of boxes so that the distance between
  each pair of boxes in $\Pi^{r,j}_{k-1}$ is bigger than $5n_k$, for
  every $j$, and the number of boxes in $\Pi^{r,j}_{k-1}$ is between
  $\vert\Pi^{(r)}_{k-1}\vert/(2(5n_k)^d)$ and
  $2\vert\Pi^{(r)}_{k-1}\vert/(5n_k)^d$.

  If the number of $(k-1,\theta,\varepsilon,\tilde\xi^k)$-bad boxes in
  $\Pi^{(r)}_{k-1}$ is bigger than
  $C\vert\Pi^{(r)}_{k-1}\vert n_{k-1}^{-\beta}$, then there exists at
  least one $j$ so that the number of bad boxes in $\Pi^{r,j}_{k-1}$
  is larger than $C\vert\Pi^{r,j}_{k-1}\vert n_{k-1}^{-\beta}$. Since
  the boxes in $\Pi^{r,j}_{k-1}$ are good or bad independently of each
  other, their number is bounded and they are bad with probability
  $\tilde{p}_{k-1}$, it follows by Hoeffding's inequality that
  \begin{align}
    \label{eq:Bound on density of bad boxes k geq 4}
    \begin{split}
      \bar{\bP} (\vert\Pi^{(r)}_{k-1} \cap
      & \{(k-1,\theta,\varepsilon,\tilde\xi^k)\text{-bad boxes}
      \}\vert>C\vert\Pi^{(r)}_{k-1}\vert n_{k-1}^{-\beta})\\
      & \le (5n_k)^d\bar{\bP}(\vert\Pi^{r,1}_{k-1}\cap
      \{(k-1,\theta,\varepsilon,\tilde\xi^k)\text{-bad boxes} \}\vert
      \ge \lceil C\vert\Pi^{(r)}_{k-1}\vert n_{k-1}^{-\beta}/(5n_k)^d
      \rceil)\\
      & \le (5n_k)^d \exp(-(Cn_{k-1}^{-\beta}-\tilde{p}_{k-1})^2 \vert
      \Pi^{(r)}_{k-1}\vert /(5n_k)^d)\\
      & \le \tilde{C} (5n_k)^d \exp(-Cn_{k-1}^{-2\beta}\vert
      \Pi^{(r)}_{k-1}\vert /(5n_k)^d)\\
      & \le \tilde{C} (5n_k)^d
      \exp(-C2^{rd}N^{\frac{-2\beta}{2^{k-1}}+\frac{d}{2} -\frac{d
          \theta}{2^{k-1}} -\frac{d}{2^k}})\\
      & = \tilde{C}(5n_k)^d \exp(-C2^{rd} N^{\frac{d}{2}
        -(\frac{4\beta + 2d\theta +d}{2^k})}),
    \end{split}
  \end{align}
  where the right hand side decays stretched exponentially in $N$ for
  $k\ge 4$ if $\beta$ is small enough, e.g.\ $\beta =1$ (which is still
  sufficient for the proof of \eqref{eq:secondInequality}). For
  $1\le k \le 3$ notice that
  \begin{align}
    \label{eq:Bound on density of bad boxes k leq 3}
    \begin{split}
      \bar{\bP}&(\vert \Pi^{(r)}_{k-1}\cap
      \{(k-1,\theta,\varepsilon,\tilde\xi^k)\text{-bad boxes} \} \vert
      > C\vert \Pi^{(r)}_{k-1} \vert n_{k-1}^{-\beta})\\
      & \le \bar{\bP}(\{(k-1,\theta,\varepsilon,\tilde\xi^k)\text{-bad
        boxes} \}\neq \emptyset)\\
      & \le \vert \Pi^{(r)}_{k-1} \vert \tilde{p}_{k-1} \le
      (\sqrt{N}\log^3(N))^d\tilde{p}_{k-1}\le(\sqrt{N}\log^3(N))^d N^{-c\log(N)}
      \le C N^{-c\log(N)}.
    \end{split}
  \end{align}
  Using the estimates above together with the definition of
  $\widetilde{G}_N$ shows that
  \begin{align}
    \begin{split}
      \bar{\bP}(\widetilde{G}_N^\compl)
      & = \bar{\bP}\biggl( \bigcup_{k=1}^{r(N)}
        \bigcup_{r=0}^{\lceil \log_2 (\log N_{k-1})^3 \rceil} \left\{
          \vert \Pi_{k-1}^{(r)} \cap \{(k-1,\theta,\varepsilon,\tilde\xi^k)\text{-bad boxes} \}
          \vert > C \vert \Pi_{k-1}^{(r)}\vert n_{k-1}^{-\beta}\right\} \biggr)\\
      & \le \sum_{k=1}^{r(N)} \sum_{r=0}^{\lceil \log_2 (\log N_{k-1})^3
        \rceil} \bar{\bP} \left( \vert \Pi_{k-1}^{(r)} \cap \{
        (k-1,\theta,\varepsilon,\tilde\xi^k)\text{-bad boxes} \} \vert
        > C \vert \Pi_{k-1}^{(r)}\vert n_{k-1}^{-\beta} \right)\\
      & \le r(N)\lceil \log_2 (\log N_{k-1})^3 \rceil CN^{-c\log(N)} \le
      C\log\log(N)\cdot\log(N)^{5/6}N^{-c\log(N)}\\
      & \le N^{-\tilde{c}\log(N)}.
    \end{split}
  \end{align}
  Next we show that the number of $(k-1,\theta,\varepsilon)$-bad boxes
  in $\xi$ is on the same order as the number of
  $(k-1,\theta,\varepsilon,\tilde\xi^k)$-bad boxes in $\tilde\xi^k$
  with high probability. First we define, in a slight abuse of
  notation, the sets
  \begin{align*}
    D_{n_k}(\Delta)
    & \coloneqq \{ (x,t) \in \bZ^d\times\bZ : \dist(x,\Delta)\le n_k, t\in[0,n_k] \},\\
    A_{k,\Delta}
    & \coloneqq \{ \omega\in\Omega : \xi_t(x) = \tilde{\xi}^k_t(x)
      \text{ for all } (x,t)\in D_{n_k}(\Delta)  \}
  \end{align*}
  for all $\Delta \in \Pi^{(r)}_{k-1}$. Note that $D_{n_k}(\Delta)$ is
  the same box as $D_{n_k}(x)$ if $x$ is the center of $\Delta$. Using
  the above defined partitions $(\Pi^{r,j}_{k-1})_j$ we see that for
  every choice of $\Delta,\Delta'\in\Pi^{r,j}_{k-1}$ the events
  $A_{k,\Delta}$ and $A_{k,\Delta'}$ are independent, since
  $\dist(\Delta,\Delta')>5n_k$. Since $\xi\le \tilde\xi^k$ the number
  of $(k-1,\theta,\varepsilon)$-good boxes in $\xi$ is less or equal
  to the number of $(k-1,\theta,\varepsilon,\tilde\xi^k)$-bad boxes in
  $\tilde\xi^k$.

  To shorten the notation we say for a box $\Delta\in \Pi^{(r)}_{k-1}$
  that it is good in $\xi$ if it is $(k-1,\theta,\varepsilon)$-good
  and good in $\tilde\xi^k$ if it is
  $(k-1,\theta,\varepsilon,\tilde\xi^k)$-good. A box can only be bad
  in $\xi$ and good in $\tilde\xi^k$ for $\omega \in A_{k,\Delta}^\compl$.
  Using Corollary~\ref{cor:coupling_xi_and_tilde_xi} we get
  $\bP(A_{k,\Delta}^\compl)\le Cn_k^{-c\log n_k}$, and thus, again by
  Hoeffding's inequality,
  \begin{align}
    \begin{split}
      \bP&\Big(\abs{\Pi^{(r)}_{k-1} \cap\{\text{bad in } \xi \}} -
      \abs{\Pi^{(r)}_{k-1} \cap\{\text{bad in } \tilde\xi^k\}} \ge
      C\abs{\Pi^{(r)}_{k-1}}n_{k-1}^{-\beta}\Big)\\
      & \le \bP\Big(\exists j \text{ s.t. }\abs{\Pi^{r,j}_{k-1}
        \cap\{\text{bad in } \xi \}} - \abs{\Pi^{r,j}_{k-1}
        \cap\{\text{bad in } \tilde\xi^k\}} \ge
      C\abs{\Pi^{(r)}_{k-1}}n_{k-1}^{-\beta}\frac{1}{(5n_k)^d}\Big)\\
      & \le (5n_k)^d \bP\Big(\abs{\Pi^{r,j}_{k-1} \cap\{\text{bad in }
        \xi \}} - \abs{\Pi^{r,j}_{k-1} \cap\{\text{bad in }
        \tilde\xi^k\}} \ge
      C\abs{\Pi^{(r)}_{k-1}}n_{k-1}^{-\beta}\frac{1}{(5n_k)^d}\Big)\\
      & \le (5n_k)^d \bP\Big(\sum_{\Delta \in\Pi^{r,j}_{k-1}}
      \mathbbm{1}_{A_{k,\Delta}^\compl} \ge
      C\abs{\Pi^{(r)}_{k-1}}n_{k-1}^{-\beta}\frac{1}{(5n_k)^d} \Big)\\
      & \le \tilde{C}(5n_k)^d \exp\Big(
      -C2^{rd}N^{\frac{d}{2}-(\frac{4\beta+2d\theta +d}{2^k})} \Big).
    \end{split}
  \end{align}
  Again the right hand side decays stretched exponentially in $N$ for
  $k\ge 4$ for $\beta>0$ small enough. For $k\le 3$ we can repeat the
  ideas of \eqref{eq:Bound on density of bad boxes k leq 3}. The
  reason we can prove an upper bound in the same way as in
  \eqref{eq:Bound on density of bad boxes k geq 4} and \eqref{eq:Bound
    on density of bad boxes k leq 3} is that the probability for a box
  to be bad in $\tilde\xi^k$ is of the same order as
  $\bP(A_{k,\Delta}^\compl)$, namely $n_k^{-c\log n_k}$. Define
  \begin{align}
    A_N \coloneqq \bigcap_{k=1}^{r(N)}
    \bigcap_{r=0}^{\lceil\log_2(\log
    N_{k-1})^3\rceil}\Big\{\abs{\Pi^{(r)}_{k-1} \cap\{\text{bad in }
    \xi \}} - \abs{\Pi^{(r)}_{k-1} \cap\{\text{bad in } \tilde\xi^k\}}
    \ge C\abs{\Pi^{(r)}_{k-1}}n_{k-1}^{-\beta} \Big\}
  \end{align}
  then by the same arguments as above we also get
  \begin{align}
    \bP(A^\compl_N) \le N^{-c\log N}.
  \end{align}
  Since $\widetilde{G}_N\cap A_N \subset G_N$ the claim follows.
\end{proof}

\section{Mixing properties of the quenched law: proof of
  Lemma~\ref{lem:HIG_lemma}}
\label{subsec:mixing_quenched}
\begin{defn}
  \label{def:social boxes}
  Let $\Pi_M$ be a partition of $\bZ^d$ into boxes of side lengths
  $M$, let $C>0$ and let $\omega$ be a realisation of the environment.
  We call a box $\Delta \in \Pi_M$ \emph{social with respect to
    $\omega$} 
  at time $N \in \bN$, if for any pair of points $x,y\in\Delta$ there
  exists $z \in \bZ^d$ such that
  \begin{align*}
    P^{(x,N)}_\omega(X_{N+\lceil CM\rceil } =z)>0,
    \quad \text{and} \quad
    P^{(y,N)}_\omega(X_{N+\lceil CM \rceil}=z)>0.
  \end{align*}
\end{defn}
Note that if $P^{(x,N)}_\omega(X_{N + \lceil CM \rceil} =z)>0$, then by
construction
$P^{(x,N)}_\omega(X_{N+\lceil CM \rceil} =z) \ge (3^{-d})^{CM}$.

\medskip

The next result shows that the density of social boxes is suitably
high.

\begin{lem}
  \label{lem:high_density_of_social_boxes}
  For every $\varepsilon>0$ there exists $M_0\in \bN$ and constants
  $c,C>0$ such that for all $M\ge M_0$ there exists a set of
  environments $S_M$ satisfying
  \begin{align*}
    \sum_{\substack{\Delta \in \Pi_M\\ \Delta \text{ is not
    social}}}\bP^{(x,0)}(X_n \in \Delta) < \varepsilon \quad \text{for all $\omega \in S_M$}
  \end{align*}
  and $\bP(S_M)\ge 1- C\mathrm{e}^{-c\log n}$. (Recall that the
  property of $\Delta$ being
  social depends on $\omega$.)
\end{lem}

\begin{cor}
  \label{cor:quenched_density_of_social_boxes}
  For every $\varepsilon>0$ there exists $M_0 \in \bN$ so that for all
  $M > M_0$ there are environments $\bar{S}_M$ such that
  \begin{align*}
    \sum_{\substack{\Delta \in \Pi_M\\ \Delta \text{ is not social}}}
    P^{(x,0)}_\omega(X_n \in \Delta) < 2\varepsilon
  \end{align*}
  for all $\omega \in \bar{S}_M$ and $\bP(\bar{S}_M)\ge 1 - Cn^{-c\log n}$.
\end{cor}
\begin{proof}
  Combine Lemma~\ref{lem:high_density_of_social_boxes} and
  Lemma~\ref{claim:1}.
\end{proof}

\begin{proof}[Proof of Lemma~\ref{lem:high_density_of_social_boxes}]
  The proof idea is similar to the one we have used to prove the high
  density of good boxes; see the proof of Proposition~\ref{prop_key}.
  We set
  \begin{align*}
    p_M \coloneqq \bP(\Delta \text{ is not social}).
  \end{align*}
  As a direct consequence of Lemma~\ref{cor:opc_connectedness} for
  every $\Delta\in\Pi_M$ we have that $p_M \le C\mathrm{e}^{-cM}$ for
  some positive constants $C,c$. We define
  \begin{align}
  \label{eq:76}
  S_M \coloneqq  \bigcap_{r=0}^{\log_2 \log^3 n} \left\{
  \abs{\Pi^{(r)}_M\cap\{\text{not social
                boxes}\}}<C\abs{\Pi^{(r)}_M}p_M \right\},
  \end{align}
  where
  \begin{align*}
  \Pi^{(0)}_M & = \{\Delta \in \Pi_M:\dist(\Delta,0)\le \sqrt{n} \}, \\
  \Pi^{(r)}_M & = \{ \Delta \in \Pi_M : 2^{r-1}\sqrt{n}<\dist(\Delta,0)\le
  2^r\sqrt{n} \} \quad \text{for $r\ge 1$}.
  \end{align*}
  By Lemma~3.6 from \cite{SteibersPhD2017} we have
  $\bP^{(0,0)}(\norm{X_n}\ge \sqrt{n}\log^3 n)\le Cn^{-c\log n}$ and
  so for $\omega\in S_M$ (note that being social depends on $\omega$)
  \begin{align*}
  \sum_{\substack{\Delta \in \Pi_M\\ \Delta \text{ is not social}}} \bP^{(0,0)}(X_n\in \Delta)
  & \le Cn^{-c\log n}+ \sum_{r=0}^{\log_2\log^3n}\sum_{\substack{\Delta \in \Pi^{(r)}_M\\ \Delta \text{ is not social}}}
  \bP^{(0,0)}(X_n\in \Delta)\\
  & \le \sum_{r=0}^{\log_2\log^3 n} C\abs{\Pi^{(r)}_M}p_M
  \frac{1}{n^{d/2}}\exp\left(- \frac{1}{2n}(2^{r-1}\sqrt{n})^2 \right)\\
  & \le C\sum_{r=0}^{\log_2\log^3 n}\left( \frac{2^r\sqrt{n}}{M} \right)^d \frac{1}{n^{d/2}} \exp(-cr^2)p_M\\
  & \le C p_M\sum_{r=0}^{\log_2\log^3 n}\frac{1}{M^d} \exp(-cr^2 +rd\log 2)\\
  & \le C' p_M
  \end{align*}
  where we used the annealed local CLT in the second inequality. It
  remains to show that $\bP^{(0,0)}(S_M)\ge 1 - C\mathrm{e}^{-c\log n}$. We
  have
  \begin{align*}
  \bP^{(0,0)}(S_M^c)
  & = \bP^{(0,0)}\bigl(\exists r \le \log_2\log^3 n :
  \abs{\Pi^{(r)}_M\cap \{ \text{not social boxes}}>C\abs{\Pi^{(r)}_M}p_M\bigr)\\
  & \le \sum_{r=0}^{\log_2\log^3 n} \bP^{(0,0)}\bigl( \abs{\Pi^{(r)}_M\cap
        \{ \text{not social boxes}}>C\abs{\Pi^{(r)}_M}p_M \bigr).
  \end{align*}
  Next, let $(\Pi^{r,j}_M)_{j \in J}$ be a further partition of
  $\Pi^{(r)}_M$
  so that for each $j \in J$ the distance between any pair of distinct
  boxes in $\Pi^{r,j}_M$ is bigger than $3CM$ and
  \begin{align*}
  \frac{\abs{\Pi^{(r)}_M}}{2(3CM)^d} \le \abs{\Pi^{r,j}_M} \le
  \frac{2\abs{\Pi^{(r)}_M}}{(3CM)^d}.
  \end{align*}
  Note that the index set $J=J(M,r)$ is finite (in fact we have
  $\abs{J} \le 2(3CM)^d$) and that by construction the boxes in
  $\Pi^{r,j}_M$ are social or not social independently of each other.
  If
  $\abs{\Pi^{r}_M\cap \{ \text{not social
                boxes}\}}>C\abs{\Pi^{(r)}_M}p_M$ then there exists a $j$ such
  that
  $\abs{\Pi^{r,j}_M\cap \{\text{not social
                boxes}\}}>C\abs{\Pi^{(r)}_M}p_M/(3CM)^d$. Using Hoeffding's
  inequality for $r\ge 1$ we obtain
  \begin{align*}
  \bP^{(0,0)}
  & \bigl( \abs{\Pi^{(r)}_M\cap \{ \text{not social boxes}}>C\abs{\Pi^{(r)}_M}p_M \bigr)\\
  & \le \sum_{j\in J} \bP^{(0,0)}\Bigl(\abs{\Pi^{r,j}_M\cap\{\text{not
                social boxes}}>\frac{C\abs{\Pi^{(r)}_M}p_M}{(3CM)^d}\Bigr)\\
  & = \sum_{j\in J} \bP^{(0,0)}\Bigl(\abs{\Pi^{r,j}_M\cap\{\text{not social boxes}}
  - \abs{\Pi^{r,j}_M}p_M>\Bigl(\frac{C\abs{\Pi^{(r)}_M}}{(3CM)^d}
  - \abs{\Pi^{r,j}_M}\Bigr)p_M\Bigr)\\
  & \le \sum_{j \in J} \exp\Bigl( -2p_M^2 \Bigl(
  C\frac{\abs{\Pi^{(r)}_M}}{(3CM)^d}-\abs{\Pi^{r,j}_M} \Bigr)^2
  \Bigr)\\
  & \le \sum_{j \in J}\exp\Bigl( -2p_M^2 (C-2)\frac{\abs{\Pi^{(r)}_M}^2}{(3CM)^{2d}} \Bigr)\\
  & \le 2(3CM)^d\exp\Bigl( -Cp_M^2 \frac{(2^{r-1}\sqrt{n})^{2d}}{(3CM)^{2d}} \Bigr).
  \end{align*}
  Similarly for $r=0$ we have
  \begin{align*}
  \bP^{(0,0)}\bigl( \abs{\Pi^{(0)}_M\cap \{ \text{not social boxes}}
  > C\abs{\Pi^{(0)}_M}p_M \bigr)
  \le 2(3CM)^d\exp\Bigl( -Cp_M^2 \frac{\sqrt{n}^{2d}}{(3CM)^{2d}} \Bigr).
  \end{align*}
  Using the above estimates we obtain
  \begin{align*}
  \bP^{(0,0)}(S_M^c)
  & \le 2(3CM)^d\exp\Bigl( -Cp_M^2\frac{\sqrt{n}^{2d}}{(3CM)^{2d}} \Bigr)
  + \sum_{r=1}^{\log_2\log^3n} 2(3CM)^d\exp\Bigl( -Cp_M^2
  \frac{(2^{r-1}\sqrt{n})^{2d}}{(3CM)^{2d}} \Bigr)\\
  & \le \log_2\log^3(n) \cdot \exp\Bigl( -Cp_M^2\frac{\sqrt{n}^{2d}}{(3CM)^{2d}} \Bigr)\\
  & \le C n^{-c\log n}.
  \end{align*}
\end{proof}

\begin{proof}[Proof of Lemma~\ref{lem:HIG_lemma}]
  The proof relies on a construction of a suitable coupling of
  $P^{(x,0)}_\omega(X_n \in \cdot)$ and
  $P^{(y,0)}_\omega(X_n \in \cdot)$. First we show that there is a
  coupling on the level of boxes with side length $M$, where $M$ is a
  constant. Let $\Pi_M$ be a partition of $\bZ^d$ in boxes of side
  length $M$ and fix $x$ and $y$. Set
  \begin{align*}
    F_{n^\theta}
    & \coloneqq \bigcap_{k\ge n^\theta} \Bigl\{ \omega : 
      \forall z\in [-k,k]^d\cap \bZ^d, \\[-2ex]
    & \qquad \qquad \qquad  \sum_{\Delta \in \Pi_M}
      \abs{\bP^{(z,0)}(X_{k} \in \Delta) - P^{(z,0)}_\omega(X_{k} \in
      \Delta)}\le \frac{C_1}{k^{c_2}}+ \frac{C_1}{M^{c_2}} \Bigr\}.
  \end{align*}
  and
  \begin{align*}
  F(x,y) &\coloneqq \bigcap_{\substack{(\tilde{x},m)\in\bZ^d\times\bN_0\\
                \norm{\tilde{x}-x}\le n\\ m\le n}} \sigma_{(\tilde{x},m)}F_{n^\theta} \cap \bigcap_{\substack{(\tilde{y},m)\in\bZ^d\times\bN_0\\
                \norm{\tilde{y}-y}\le n\\ m\le n}} \sigma_{(\tilde{y},m)}F_{n^\theta}
  \end{align*}
  By Lemma~\ref{claim:1} we have
  $\bP(F_{n^\theta})\ge 1-n^{-c\log n}$ and thus $\bP(F(x,y))\ge 1- Cn^{-c\log n}$. In the following we assume
  that the indices of the random walks are integers, otherwise we take
  the integer part. Now choosing $M$ and $n$ large enough for
  $\norm{x-y}\le n^\theta$
  on the event $F(x,y)$ we obtain
  \begin{align*}
  \sum_{\Delta \in \Pi_M}
  & \abs{P^{(x,0)}_\omega(X_{n^{2\theta}\log^{8d}n^\theta} \in
        \Delta)-P^{(y,0)}_\omega(X_{n^{2\theta}\log^{8d}n^\theta} \in \Delta)} \\
  & \le \sum_{\Delta \in \Pi_M} \abs{P^{(x,0)}_\omega(X_{n^{2\theta}\log^{8d}n^\theta} \in \Delta)
        - \bP^{(x,0)}(X_{n^{2\theta}\log^{8d}n^\theta} \in \Delta)} \\
  & \quad + \sum_{\Delta \in \Pi_M} \abs{P^{(y,0)}_\omega(X_{n^{2\theta}\log^{8d}n^\theta} \in \Delta)
        - \bP^{(y,0)}(X_{n^{2\theta}\log^{8d}n^\theta} \in \Delta)} \\
  & \quad + \sum_{\Delta \in \Pi_M} \abs{\bP^{(x,0)}(X_{n^{2\theta}\log^{8d}n^\theta} \in \Delta)-
        \bP^{(y,0)}(X_{n^{2\theta}\log^{8d}n^\theta} \in \Delta)} \\
  & \le \frac{1}{8} + \frac{1}{8} +\sum_{\Delta \in \Pi_M}
  \abs{\bP^{(x,0)}(X_{n^{2\theta}\log^{8d}n^\theta} \in \Delta)
        - \bP^{(y,0)}(X_{n^{2\theta}\log^{8d}n^\theta} \in \Delta)} \\
  & \le \frac{1}{4}+ \sum_{\Delta \in \Pi^{x,y}_M(n^{2\theta}\log^{8d}n^\theta)}
  \abs{\bP^{(x,0)}(X_{n^{2\theta}\log^{8d}n^\theta} \in \Delta)
        - \bP^{(y,0)}(X_{n^{2\theta}\log^{8d}n^\theta} \in \Delta)} + Cn^{-c\log n}\\
  & \le \frac{1}{4}+ Cn^{-c\log n}+ \abs{\Pi^{x,y}_M(n^{2\theta}\log^{8d}n^\theta)}dn^\theta
  C(n^{2\theta}\log^{8d}n^\theta)^{-\frac{d+1}{2}}\\
  & \le\frac{1}{4}+Cn^{-c\log n}+2 \Bigl(n^{\theta}\log^{4d}(n^\theta)\log^3(n^{2\theta}
  \log^{8d}n^\theta) \Bigr)^ddn^\theta C(n^{2\theta}\log^{8d}n^\theta)^{-\frac{d+1}{2}}\\
  & =\frac{1}{4}+Cn^{-c\log n} + C\Bigl( \log(n^{2\theta}\log^{8d}n^\theta) \Bigr)^{3d}\log^{-4d}(n^\theta)\\
  & < \frac{1}{2},
  \end{align*}
  for $n$ large enough, where
  \begin{align*}
  \Pi^{x,y}_M(m)\coloneqq \Bigl\{\Delta \in \Pi_M :
  \Delta \cap
  \{z\in\bZ^d:\min(\norm{x-z},\norm{y-z})\le
  \sqrt{m}\log^3m \} \neq \emptyset \Bigr\}
  \end{align*}
  and we used Lemma~3.6 from \cite{SteibersPhD2017} and the annealed
  derivative estimates (Lemma~3.9 from \cite{SteibersPhD2017}). The
  number of steps we chose might seem a bit strange at first. The
  choice becomes more clear by looking at the last inequality above.
  There we see that, with the methods we use, we need a bit more steps
  than the square of the current distance. One can calculate that any
  additional factor $\log^m(n^\theta)$ with $m>6d$ is enough to get
  the estimate. So there exists a coupling
  $\Xi^{x,y}_{\omega,n^{2\theta}\log^{8d}n^\theta}$ of
  $P^{(x,0)}_\omega(X_{n^{2\theta}\log^{8d}n^\theta}\in\cdot)$ and
  $P^{(y,0)}_\omega(X_{n^{2\theta}\log^{8d}n^\theta}\in\cdot)$ on
  $\Pi_M\times \Pi_M$ such
  that for $\omega \in F(x,y)$
  \begin{align*}
  \Xi^{x,y}_{\omega,n^{2\theta}\log^{8d}n^\theta}(\{(\Delta,\Delta): \Delta \in \Pi_M \})>\frac{1}{2}.
  \end{align*}
  Recall
  $\bar S_M$ from
  Corollary~\ref{cor:quenched_density_of_social_boxes}.
  We have for
  \begin{align*}
  \omega \in H(x,y):=F(x,y)\cap \bigcap_{\substack{(\tilde{x},m)\bZ^d\times\bN_0\\\norm{\tilde{x}-x}\le n\\m\le n}}\sigma_{(\tilde{x},m)}\bar{S}_M \cap \bigcap_{\substack{(\tilde{y},m)\bZ^d\times\bN_0\\\norm{\tilde{y}-y}\le n\\m\le n}}\sigma_{(\tilde{y},m)}\bar{S}_M
  \end{align*}
  that
  \begin{align*}
  \sum_{\substack{\Delta \in \Pi_M\\\Delta \text{ is social}}}
  \Xi^{x,y}_{\omega,n^{2\theta}\log^{8d}n^\theta}(\Delta,\Delta)
  > \frac{1}{2}-\varepsilon(M)>\frac{1}{4}.
  \end{align*}
  By Corollary~\ref{cor:quenched_density_of_social_boxes} we obtain $\bP(H(x,y))\ge 1-Cn^{-c\log n}$.
  Thus, by the definition of social boxes (Definition~\ref{def:social
    boxes}), we can construct a coupling
  $\widetilde{\Xi}^{x,y}_{\omega,n^\theta}$
  of
  $P^{(x,0)}_\omega(X_{n^{2\theta}\log^{8d}n^{\theta}+CM} \in \cdot)$
  and
  $P^{(y,0)}_\omega(X_{n^{2\theta}\log^{8d}n^{\theta}+CM} \in \cdot)$
  satisfying $\widetilde{\Xi}^{x,y}_{\omega,n^\theta}(\{(z,z):z\in
  \bZ^d \})>\frac{1}{4}(\frac{1}{3^d})^{2CM}$. If this coupling is
  successful, we let the random walks go along the same path until
  time $n$. In case it isn't, we try to couple from their current
  position. Note that $\omega \in H(x,y)$ ensures, that we can repeat
  the coupling attempt at the new positions.

  For the rest of the proof let
  $n_k \coloneqq n^\theta\log^{k(4d+3)}n$, $k\in \bN_0$ and
  $s_k \coloneqq n_k^2\log^{8d}n_k+CM$. The $n_k$ will represent the
  distance between the walkers at the start of an attempt at coupling and $s_k$
  will be the number of steps necessary for the attempt. Furthermore let $S_k := \sum_{i=0}^{k} s_i$.

  By Lemma~3.6 from \cite{SteibersPhD2017}, we know that with
  probability of at least $1-Cn^{-c\log n}$ the distance between the
  random walks will only be
  \begin{align*}
  \bigl(n^{2\theta}\log^{8d}n^\theta\bigr)^{1/2}\log^3\bigl(n^{2\theta}\log^{8d}n^\theta \bigr)
  \le n^\theta \log^{4d}(n^\theta)\log^3(n)\le n^\theta \log^{4d+3} n=n_1,
  \end{align*}
  as long as $8d \le (1-2\theta)\frac{\log n}{\log \log n^\theta}$.
  This condition is not a restriction, since we will let
  $n\to \infty$.

  Let us now iterate the coupling procedure. If the coupling in step
  $k-1$ is not successful, i.e.\ if the walks are not at the same
  point, we try to couple again starting from the current positions.
  This leads to an iterative coupling $\widehat{\Xi}$ of the following
  form:
  $\widehat{\Xi}^{x,y}_{\omega,0} = \widetilde{\Xi}^{x,y}_{\omega,n_0}
  = \widetilde{\Xi}^{x,y}_{\omega,n^\theta}$ and for $k \ge 1$
  \begin{align*}
    \widehat{\Xi}^{x,y}_{\omega,k} (z_1,z_2) = \sum_{a,b \in \bZ^d}
     \widehat{\Xi}^{x,y}_{\omega,k-1}(a,b) \cdot \Big[
    & \ind{a=b}\ind{z_1=z_2}P^{(a,S_{k-1})}_\omega(X_{S_k} =z_1)\\
    & \; + \ind{0<\norm{a-b}\le n_{k}} \widetilde{\Xi}^{a,b}_{\omega,n_k}(z_1,z_2)\\
    & \; + \ind{\norm{a-b}>n_{k}}
      P^{(a,S_{k-1})}_\omega(X_{S_k}=z_1)P^{(b,S_{k-1})}_\omega(X_{S_k}=z_2)\Big],
  \end{align*}
  where $\widetilde{\Xi}^{a,b}_{\omega,n_k}$ is a coupling of
  $P^{(a,S_{k-1})}_\omega(X_{S_k} \in \cdot)$ and
  $P^{(b,S_{k-1})}_\omega(X_{S_k} \in \cdot)$. The idea is
  that the random walks will stay together once they are at the same
  site. We try to couple them via $\widetilde{\Xi}^{a,b}_{\omega,n_k}$
  if their distance is not too large and we let them evolve
  independently otherwise.

  Since at distance $n_k$ for the next coupling we walk $s_k$ steps
  and with high probability have at most a distance of
  $s_k^{1/2}\log^3 s_k$, the above coupling will work as long as
  $k\le \frac{(1-2\theta)\log n}{(8d+6)\log\log n}- \frac{8d}{8d+6}$
  holds, which we show below. We obtain
  \begin{align*}
  s_k^{1/2} \log^3 s_k
  = \Bigl(n_k^2\log^{8d}n_k+CM\Bigr)^{1/2}\log^3\Bigl(n_k^2\log^{8d}n_k+CM\Bigr).
  \end{align*}
  Now for $k\le \frac{(1-2\theta)\log n}{(4d+3)\log\log n}$ and $n$ large enough
  \begin{align*}
  n_k^2\log^{8d}n_k +CM \le  n_k^2 \log^{8d}n
  \end{align*}
  and
  \begin{align*}
  \log^{4d} n_k = \log^{4d}\Bigl(n^\theta\log^{k(4d+3)}(n) \Bigr) \le \log^{4d} n.
  \end{align*}
  Thus, we have
  \begin{align*}
  s_k^{1/2}\log^3 s_k \le n_k\log^{4d}(n)\log^3\Bigl( n_k^2\log^{8d}n \Bigr).
  \end{align*}
  Furthermore, if
  $k \le \frac{(1-2\theta)\log n}{(8d+6)\log\log n}- \frac{8d}{8d+6}$
  then
  \begin{align*}
  2\log n_k + 8d\log\log n
  & = 2\log\bigl( n^\theta\log^{k(4d+3)} n \bigr) + 8d\log\log n \\
  & = 2\theta\log n + k(8d+6)\log\log n + 8d\log\log n \le \log n
  \end{align*}
  It follows that
  \begin{align*}
  s_k^{1/2}\log^3 s_k \le 2n_k\log^{4d} n \log^3 n
  = 2n^\theta\log^{(k+1)(4d+3)}(n) = n_{k+1}.
  \end{align*}
  So after we try the $k$-th coupling we are, with high probability,
  at distance $n_{k+1}$. The probability for each try to be successful
  is bounded from below by $\frac{1}{4}(\frac{1}{3^d})^{2CM}$ and we
  have $\frac{(1-2\theta)\log n}{(8d+6)\log\log n}-1$ attempts. So the
  time we need for those attempts is
  \begin{align}
  \label{eq:time for coupling}
  \begin{split}
  \sum_{k=0}^{\frac{(1-2\theta)\log n}{(8d+6)\log\log n}-1} s_k
  & = \sum_{k=0}^{\frac{(1-2\theta)\log n}{(8d+6)\log\log n}-1} n_k^2\log^{8d}n_k +CM \\
  & \le \sum_{k=0}^{\frac{(1-2\theta)\log n}{(8d+6)\log\log n}-1}
  n^{2\theta}\log^{k(8d+6)}(n)\log^{8d}(n) +CM\\
  & = \frac{(1-2\theta)\log n}{(8d+6)\log\log n}CM+
  n^{2\theta}\log^{8d}(n)\sum_{k=0}^{\frac{(1-2\theta)\log
                n}{(8d+6)\log\log n}-1}\left(\log^{(8d+6)}(n)\right)^k.
  \end{split}
  \end{align}
  Note that
  \begin{align*}
  (\log n )^{\frac{(1-2\theta)\log n}{(8d+6)\log\log n} (8d+6)}
  =\exp\bigl((1-2\theta)\log n \bigr) = n^{1-2\theta}
  \end{align*}
  and therefore the right hand side of \eqref{eq:time for coupling} is
  bounded from above by
  \begin{multline*}
  \frac{(1-2\theta)\log n}{(8d+6)\log\log n}CM+ n^{2\theta}\log^{8d}(n)
  \frac{n^{1-2\theta}-1}{\log^{(8d+6)}(n)-1}
  \le \frac{(1-2\theta)\log n}{(8d+6)\log\log n}CM+ \frac{n}{\log^{5}(n)}\\
  = n \left( \frac{(1-2\theta)\log n}{n(8d+6)\log\log n}CM
  +\frac{1}{\log^5 n}\right) < n,
  \end{multline*}
  for $n$ large enough. And the probability for the above coupling to
  fail is smaller than
  \begin{align*}
  (1-p^*)^{\frac{(1-2\theta)\log n}{(8d+6)\log\log n}-1} \le \mathrm{e}^{-c\frac{\log n}{\log \log n}}
  \end{align*}
  where $p^* = \frac{1}{4}(\frac{1}{3^d})^{2CM}$ and $c>0$ is
  a constant. So for a fixed pair of points $x,y$ with $\norm{x-y}\le n^\theta$ we have
  \begin{align*}
  \norm{P^{(x,0)}_\omega(X_n\in \cdot) - P^{(y,0)}_\omega(X_n\in \cdot)}_{TV} \le \mathrm{e}^{-c\frac{\log n}{\log\log n}}
  \end{align*}
  with probability at least $1-n^{-c\log n}$. Thus we get for every $b>0$
  \begin{align*}
  \bP(D(n))
  & = \bP\Bigg( \bigcap_{\substack{x,y\in\bZ^d\, :\\
                \norm{x},\norm{y}\le n^b,\\
                \norm{x-y}\le n^{\theta}}} \hspace{-1.5em}
  \left\{
  \norm{P^{(x,0)}_\omega(X_n\in \cdot)- P^{(y,0)}_\omega(X_n \in \cdot)}_{\mathrm{TV}}
  \le \mathrm{e}^{-c \frac{\log n}{\log \log n}} \right\} \Bigg)\\
  &\ge 1 - \sum_{\substack{x,y\in\bZ^d\, :\\
                \norm{x},\norm{y}\le n^b,\\
                \norm{x-y}\le n^{\theta}}} \bP\Big(
  \left\{\norm{P^{(x,0)}_\omega(X_n\in \cdot)- P^{(y,0)}_\omega(X_n
        \in \cdot)}_{\mathrm{TV}}> \mathrm{e}^{-c \frac{\log n}{\log \log n}}
  \right\} \Big)\\
  &\ge 1 - n^{d(b+\theta)}n^{-c\log n} \geq 1 - Cn^{-c'\log n}.
  \end{align*}
  Note that $b>0$ can be chosen arbitrarily large, but the constants
  $C$ and $c'$ will have to adjusted accordingly.
\end{proof}

\section{Uniqueness of the prefactor}
\label{sect:prefactor_unique}

With some minor adaptations of the ideas from
\cite[Section~7.1]{BergerCohenRosenthal2016} we can obtain the
following result.
\begin{lem}
  \label{lem:uniqueness_prefactor}
  Provided existence, the prefactor $\varphi$ in \eqref{eq:rd-density}
  is unique.
\end{lem}
\begin{proof}
  Assume that there are functions $f$ and $g$ which both satisfy
  \eqref{eq:rd-density} and denote $h=f-g$. We will check that
  $\bE[\abs{h}]=0$ and hence that $h \equiv 0$ holds $\bP$-a.s.

  By the triangle inequality we have
  \begin{align}
    \begin{split}
      \sum_{x \in \bZ^d} \bP^{(0,0)}(X_n = x)\abs{h(\sigma_{(x,n)}\omega)}
      & \le \sum_{x\in \bZ^d} \abs{P_\omega^{(0,0)} (X_n=x) -
        \bP^{(0,0)}(X_n=x) f(\sigma_{(x,n)}\omega)} \\
      & \quad + \sum_{x\in \bZ^d} \abs{P_\omega^{(0,0)} (X_n=x) -
        \bP^{(0,0)}(X_n=x) g(\sigma_{(x,n)}\omega)}
    \end{split}
  \end{align}
  which by \eqref{eq:rd-density} implies
  \begin{align}
    \label{eq:uniqueness1}
    \lim_{n \to \infty} \sum_{x \in \bZ^d} \bP^{(0,0)}(X_n = x)\abs{h(\sigma_{(x,n)}\omega)} = 0
  \end{align}
  for $\bP$ almost every $\omega$. That means
  $\lim_{n\to\infty} \bE^{(0,0)}[\abs{h(\sigma_{(X_n,n)}\omega)}] =0$
  $\bP$-a.s. Assume that $h\neq 0$, then there exists a measurable
  subset $A \subset \Omega$ and a constant $c>0$ such that $\bP(A) >0$
  and $\abs{h}>c$ on $A$. Thus, for every $n \in \bN$, an elementary
  computation shows
  \begin{align}
    \label{eq:uniqueness2}
    \begin{split}
      \bE \Big[\bE^{(0,0)}[\abs{h(\sigma_{(X_n,n)}\omega)}]\Big]
      & \ge \bE\Big[\bE^{(0,0)}[\abs{h(\sigma_{(X_n,n)}\omega)}
      \ind{\sigma_{(X_n,n)}\omega\in A}]\Big]\\
      & \ge c\bE\Big[\bE^{(0,0)}[\ind{\sigma_{(X_n,n)} \omega \in A}]\Big]\\
      & =c \bP(A) > 0.
    \end{split}
  \end{align}
  Since
  \begin{align}
    \begin{split}
      \bE\Big[\bE^{(0,0)}[\abs{h(\sigma_{(X_n,n)}\omega)}]\Big]
      & = \bE\Big[\sum_{y \in \bZ^d}\bP^{(0,0)}(X_n = y) \abs{h(\sigma_{(y,n)}\omega)}\Big]\\
      & =\sum_{y \in \bZ^d} \bP^{(0,0)}(X_n =y) \bE[\abs{h(\sigma_{(y,n)}\omega)}]\\
      & =\sum_{y \in \bZ^d} \bP^{(0,0)}(X_n =y) \bE[\abs{h(\omega)}] = \bE[\abs{h(\omega)}],
    \end{split}
  \end{align}
  the sequence $\{ \abs{h(\sigma_{(X_n,n)}\omega)} \}_{n \in \bN}$ is
  tight. Thus, \eqref{eq:uniqueness2} implies that for $\bP$ almost
  all $\omega$ we have
  $\lim_{n\to\infty} \bE^{(0,0)}[\abs{h(\sigma_{(X_n,n)}\omega)}]>0$
  which is a contradiction to \eqref{eq:uniqueness1}.
\end{proof}

\appendix

\section{Intersection of clusters of points connected to infinity}
\label{sec:inters-clust-points}

The following lemma is a quantification of Theorem~2 from
\cite{GrimmettHiermerDPRW2000} which was pointed out there without a
proof. We give a proof using a key result from
\cite{GaretMarchandLDPCPRE2013}.

\begin{lem}
  \label{cor:opc_connectedness}
  Let $d \ge 2$, $p >p_c$. Then there are positive constants $M$ and
  $C$ and $c$ such that for all $x,y\in \bZ^d$ with
  $\norm{x-y}\le M$
  \begin{align}
    \label{eq:79}
    \bP\bigl(B(x,y;M,C) |  (x,0) \rightarrow \infty,
    (y,0) \rightarrow \infty\bigr) \ge 1-\exp(-cM),
  \end{align}
  where $B(x,y;M,C)$ is the set of all $\omega \in \Omega$ for which
  there is $z \in \bZ^d$ satisfying
  \begin{align*}
    (x,0) \xrightarrow{\omega} (z,CM), \quad
    (y,0) \xrightarrow{\omega} (z,CM) \quad
    \text{and} \quad (z,CM) \xrightarrow{\omega} \infty.
  \end{align*}
\end{lem}

\begin{proof}
  For $A \subset \bZ^d$ we put
  $\eta^A_t(x) = \indset{(y,0) \to (x,t) \text{ for some } y \in A}$
  (this is the discrete time contact process starting from all sites
  in $A$ infected at time $0$). Write
  \begin{align}
    B(x,t) \coloneqq \big\{ \exists \, z \, : \, \norm{x-z} \le c_1 t \text{ and } \eta^{\{x\}}_t(z) \neq \eta^{\bZ^d}_t(z) \big\}
  \end{align}
  for the ``bad'' event that coupling in a ball around $x$ has not
  occurred at time $t$. We obtain from \cite[Thm.~1,
  Formula~(3)]{GaretMarchandLDPCPRE2013} that
  \begin{align}
    \label{eq:bound from GaretMarchand}
    \bP\big( B(x,t) \cap \{ (x,0) \to \infty \} \big) \le C \mathrm{e}^{-c t}
  \end{align}
  for certain constants $c_1, C , c \in (0,\infty)$ (which depend on
  $d$ and on $p>p_c$). Literally, the result in
  \cite{GaretMarchandLDPCPRE2013} is proved for the continuous time
  version of the contact process, but we believe that the same holds
  in discrete time.

  Now consider $x, y \in \bZ^d$ with $\norm{x-y} \le M$. Pick $C_2$ so
  large that
  \[
    J \coloneqq \{ z : \norm{z-x} \le C_2 M \text{ and } \norm{z-y} \le C_2 M
    \}
  \]
  satisfies $\# J \ge M^d.$ Applying \eqref{eq:bound from
    GaretMarchand} with $t = C_2 M$ for $x$ and for $y$ gives
  \begin{align}
    & \bP\Big( \big( B(x,C_2 M) \cup B(y,C_2 M) \big)
      \cap \{ (x,0) \to \infty, (y,0) \to \infty \} \big)  \notag \\
    & \quad \le \bP\big( B(x,C_2 M) \cap \{ (x,0) \to \infty \} \big)
      + \bP\big( B(y,C_2 M) \cap \{ (y,0) \to \infty \} \big) \le 2 C \mathrm{e}^{-c CM_2}
  \end{align}
  hence
  \begin{align}
    \bP\Big( \eta^{\{x\}}_{C_2 M}(z) = \eta^{\bZ^d}_{C_2 M}(z) = \eta^{\{y\}}_{C_2 M}(z) \text{ for all } z \in J
    \,\, \Big| \, (x,0) \to \infty, (y,0) \to \infty \Big) \ge 1 - C' \mathrm{e}^{-c C_2 M}.
  \end{align}
  Furthermore
  \begin{align}
    & \bP\Big( \exists \, z \in J \, : \, \eta^{\bZ^d}_{C_2 M}(z) = 1 \text{ and } (z,C_2M) \to \infty
      \,\, \Big| \, (x,0) \to \infty, (y,0) \to \infty \Big) \notag \\
    & \quad \ge
      \bP\Big( \exists \, z \in J \, : \, \eta^{\bZ^d}_{C_2 M}(z) = 1
      \text{ and } (z,C_2M) \to \infty \Big) \ge 1 - C\mathrm{e}^{-c M^d}
  \end{align}
  where we used the FKG inequality in the first inequality. For the
  second inequality we use the fact that extinction starting from $A$
  is exponentially unlikely in $\#A$ (see Theorem~2.30 (b) in
  \cite{LiggettSIS1999}) and the fact that $\eta^{\bZ^d}_{C_2 M}$
  dominates the upper invariant measure which itself dominates a
  product measure on $\{0,1\}^{\bZ^d}$ with some density $\rho>0$ (see
  Corollary~4.1 in \cite{LiggettSteifSD2006}).

  Combining, we find
  \begin{align}
    & \bP\Big( \exists \, z \in \bZ^d \, : \, (x,0) \to (z,C_2M), (y,0)
      \to (z,C_2M), (z,C_2M) \to \infty
      \,\, \Big| \, (x,0) \to \infty, (y,0) \to \infty \Big) \notag \\
    & \ge 1 - C' \mathrm{e}^{-c C_2 M} - C\mathrm{e}^{-c M^d}.
  \end{align}
\end{proof}

\section{Quenched random walk finds the cluster fast}
\label{sec:quenched-random-walk}
Since we allow the quenched random walk to start outside the cluster
we need some kind of control on the time it needs to hit the cluster.
The following lemma will yield exactly that.
\begin{lem}
  Let \label{lem:quenched-rw-cluster} $d\geq 1$ and define the set
  $A_n=A_n(C',c')\coloneqq\{ \omega\in\Omega:
  P^{(0,0)}_\omega(\xi_i(X_i)=0,i=1,\dots,n)\leq C'\mathrm{e}^{-c'n}
  \}$. There exist constants $C,c>0$, so that for every $p>p_c(d)$ and
  small enough $C'$ and $c'$ we have
  \begin{align*}
    \bP(A^\compl_n) \leq C\mathrm{e}^{-cn} \quad \text{for all $n=1,2,\dots$}.
  \end{align*}
\end{lem}
\begin{proof}
  Note that by our definition of the quenched law, see
  equation~\eqref{eq:defn_quenched_law}, the quenched random walk
  performs a simple random walk until it hits the cluster
  $\mathcal{C}$. Thus, on the event that the random walk doesn't hit
  the cluster, we can switch the random walk with a simple random walk
  $(Y_n)_n$ that is independent of the environment. Using Lemma 2.11
  from \cite{BirknerCernyDepperschmidtRWDRE2016} it follows
  \begin{align*}
    \bP^{(0,0)}\big(\xi_{0}(X_0)
    & =\dots=\xi_n(X_n)=0\big)\\
    & = \sum_{x_1,\dots,x_n} \bP^{(0,0)}\big((X_1,\dots,X_n)=(x_1,\dots,x_n),
      \xi_0(0)=\dots=\xi_n(x_n)=0 \big)\\
    & = \sum_{x_1,\dots,x_n} \bP^{(0,0)}\big((Y_1,\dots,Y_n)=(x_1,\dots,x_n),
      \xi_0(0)=\dots=\xi_n(x_n)=0 \big)\\
    & = \sum_{x_1,\dots,x_n} \bP^{(0,0)}\big((Y_1,\dots,Y_n)=(x_1,\dots,x_n)\big)
      \bP\big(\xi_0(0)=\dots=\xi_n(x_n)=0 \big)\\
    & \leq \tilde{C} \mathrm{e}^{-\tilde{c}n},
  \end{align*}
  where $\tilde C$ and $\tilde c$ are certain constants depending only
  on $p$ and $d$.

  Using the definition of the annealed law we get
  \begin{align*}
    \bP^{(0,0)}
    & \big(\xi_{0}(X_0)=\dots=\xi_n(X_n)=0\big)\\
    & =\int_{A_n} P^{(0,0)}_\omega(\xi_i(X_i)=0,i=1,\dots,n) \,d \bP(\omega)
      + \int_{A^\compl_n} P^{(0,0)}_\omega(\xi_i(X_i)=0,i=1,\dots,n) \, d\bP(\omega)\\
    & \geq \int_{A^\compl_n} P^{(0,0)}_\omega(\xi_i(X_i)=0,i=1,\dots,n) \, d\bP(\omega)\\
    & > \bP(A^\compl_n)C'\mathrm{e}^{-c'n}
  \end{align*}
  and since
  \begin{align*}
    \int_{A_n^\compl} P^{(0,0)}_\omega(\xi_i(X_i)=0,i=1\dots,n) \,d
    \bP(\omega) \leq \tilde{C}\mathrm{e}^{-\tilde{c}n}
  \end{align*}
  we obtain that $\bP(A_n^\compl)\leq C\mathrm{e}^{-cn}$ with $c=\tilde c -
  c'>0$ by choosing  $c'<\tilde{c}$.
\end{proof}

\subsection*{Acknowledgements}
\label{Acknowledgements}

We would like to thank Noam Berger for many useful discussions and for
sharing various versions of \cite{BergerCohenRosenthal2016}. We also
thank Nina Gantert for useful discussions and for her constant interest in
this work.

MB and TS were supported by the DFG SPP Priority Programme 1590
``Probabilistic Structures in Evolution'' through grant BI 1058/3-2.

\bibliographystyle{halpha} \bibliography{references}

\end{document}